\documentclass[12pt, reqno]{amsart}
\usepackage{amssymb,mathrsfs,amsfonts,amsmath,color}
\usepackage[linktocpage=true]{hyperref}
\usepackage{enumerate}

\begin{document}
\numberwithin{equation}{section} 

\def\1#1{\overline{#1}}
\def\2#1{\widetilde{#1}}
\def\3#1{\widehat{#1}}
\def\4#1{\mathbb{#1}}
\def\5#1{\frak{#1}}
\def\6#1{{\mathcal{#1}}}

\def\C{{\4C}}
\def\R{{\4R}}
\def\N{{\4N}}
\def\Z{{\4Z}}

\title[proper holomorphic maps]{Proper holomorphic maps between bounded symmetric domains with small rank differences}
\author[S.-Y. Kim, N. Mok, A. Seo ]{Sung-Yeon Kim, Ngaiming Mok, Aeryeong Seo}
\address{S.-Y. Kim: Center for Complex Geometry, Institute for Basic Science, 
55, Expo-ro, Yuseong-gu, Daejeon, Korea, 34126 }
\email{sykim8787@ibs.re.kr}
\address{N. Mok: Department of Mathematics, The University of Hong Kong, {Pokfulam} Road, Hong Kong }
\email{nmok@hku.hk}
\address{A. Seo: Department of Mathematics,
Kyungpook National University,
Daegu 41566, Republic of Korea}%
\email{aeryeong.seo@knu.ac.kr}
%\thanks{*This research was supported by Basic Science Research Program through the National Research Foundation of Korea(NRF) funded by the Ministry of  Science, ICT and Future Planning(grant number NRF-2015R1A2A2A11001367)}
\maketitle

%\tableofcontents

\def\Label#1{\label{#1}{\bf (#1)}~}
%\def\Label#1{\label{#1}}

% Standard sets

\def\cn{{\C^n}}
\def\cnn{{\C^{n'}}}
\def\ocn{\2{\C^n}}
\def\ocnn{\2{\C^{n'}}}

% Abbreviations

\def\dist{{\rm dist}}
\def\const{{\rm const}}
\def\rk{{\rm rank\,}}
\def\id{{\sf id}}
\def\aut{{\sf aut}}
\def\Aut{\textup{Aut}}
\def\CR{{\rm CR}}
\def\GL{{\sf GL}}
\def\Re{{\sf Re}\,}
\def\Im{{\sf Im}\,}
\def\span{\text{\rm span}}
\def\mult{\text{\rm mult\,}}
\def\reg{\text{\rm reg\,}}
\def\ord{\text{\rm ord\,}}
\def\hot{\text{\rm HOT\,}}

\def\codim{{\rm codim}}
\def\crd{\dim_{{\rm CR}}}
\def\crc{{\rm codim_{CR}}}

\def\eps{\varepsilon}
\def\d{\partial}
\def\a{\alpha}
\def\b{\beta}
\def\g{\gamma}
\def\G{\Gamma}
\def\D{\Delta}
\def\Om{\Omega}
\def\k{\kappa}
\def\l{\lambda}
\def\L{\Lambda}
\def\z{{\bar z}}
\def\w{{\bar w}}
\def\Z{{\1Z}}
\def\t{\tau}
\def\th{\theta}

\emergencystretch15pt
\frenchspacing

\newtheorem{Thm}{Theorem}[section]
\newtheorem{Cor}[Thm]{Corollary}
\newtheorem{Pro}[Thm]{Proposition}
\newtheorem{Lem}[Thm]{Lemma}

\theoremstyle{definition}\newtheorem{Def}[Thm]{Definition}

\theoremstyle{remark}
\newtheorem{Rem}[Thm]{Remark}
\newtheorem{Exa}[Thm]{Example}
\newtheorem{Exs}[Thm]{Examples}

\def\bl{\begin{Lem}}
\def\el{\end{Lem}}
\def\bp{\begin{Pro}}
\def\ep{\end{Pro}}
\def\bt{\begin{Thm}}
\def\et{\end{Thm}}
\def\bc{\begin{Cor}}
\def\ec{\end{Cor}}
\def\bd{\begin{Def}}
\def\ed{\end{Def}}
\def\br{\begin{Rem}}
\def\er{\end{Rem}}
\def\be{\begin{Exa}}
\def\ee{\end{Exa}}
\def\bpf{\begin{proof}}
\def\epf{\end{proof}}
\def\ben{\begin{enumerate}}
\def\een{\end{enumerate}}
\def\beq{\begin{equation}}
\def\eeq{\end{equation}}

\begin{abstract}
In this paper we study the rigidity of proper holomorphic maps $f\colon \Omega\to\Omega'$ between irreducible bounded symmetric domains $\Omega$ and $\Omega'$ with small rank differences: $2\leq \text{rank}(\Omega')< 2\,\text{rank}(\Omega)-1$. More precisely, if either $\Omega$ and $\Omega'$ { of the} same type or $\Omega$ is of type~III and $\Omega'$ is of type~I, then up to automorphisms, $f$ is of the form $f=\imath\circ F$, where $F = F_1\times F_2\colon \Omega\to \Omega_1'\times \Omega_2'$. Here  $\Omega_1'$, $\Omega_2'$ are bounded symmetric domains, the map $F_1\colon \Omega \to \Omega_1'$ is a standard embedding, $F_2: \Omega \to \Omega_2'$, and $\imath\colon \Omega'_1\times \Omega'_2 \to \Omega'$ is a totally geodesic holomorphic isometric embedding. Moreover we show that, under the rank condition above, there exists no proper holomorphic map $f: \Omega \to \Omega'$ if $\Omega$ is of type~I and $\Omega'$ is of type~III, 
or $\Omega$ is of type~II and $\Omega'$ is either of type~I or III.  By considering boundary values of proper holomorphic maps on maximal boundary components of $\Omega$, we construct rational maps between moduli spaces of subgrassmannians of compact duals of $\Omega$ and $\Omega'$, and induced CR maps between CR hypersurfaces of mixed signature, thereby forcing the moduli map to satisfy strong local differential-geometric constraints (or that such moduli maps do not exist), and complete the proofs from rigidity results on geometric substructures modeled on certain admissible pairs of rational homogeneous spaces of Picard number 1. 
\end{abstract}

\section{Introduction}

In this paper, we are concerned with the rigidity of proper holomorphic maps between irreducible bounded symmetric domains when differences between the ranks of the domains are small.

A map between topological spaces is said to be {\it proper} if the {pre-images} of compact subsets are compact. 
If the spaces are bounded domains in Euclidean spaces and the map extends continuously to the boundary, the properness of the map is equivalent to the boundary being mapped to the boundary. Hence if the domains have special boundary structures, the map {is expected to have} a certain rigidity. 
In the case of bounded symmetric domains in their standard realizations, which are one of the most studied geometric objects since Cartan introduced them in his celebrated dissertation, the structure of their boundaries was extensively studied by Wolf (\cite{W69, W72}).

The study of rigidity of proper holomorphic maps between bounded symmetric domains started with Poincar\'e (\cite{P07}), who discovered that any biholomorphic map between two connected open pieces of the the unit sphere in $\mathbb C^2$ is a restriction of (the extension to $\overline{\mathbb B^2}$ of) an automorphism of the 2-dimensional unit ball $\mathbb B^2$. Later, Alexander \cite{A74} and Henkin-Tumanov \cite{TK82} generalized his result to higher dimensional unit balls and higher rank { irreducible} bounded symmetric domains respectively.
For unit balls of different dimensions, proper holomorphic maps have been studied thoroughly by many mathematicians: Cima--Suffridge \cite{CS90}, Faran \cite{F86}, Forstneric \cite{F86, F89}, Globevnik \cite{G87}, Huang \cite{H99, H03}, Huang--Ji \cite{HJ01}, Huang--Ji--Xu \cite{HJX06}, Stens\o nes \cite{S96}, D'Angelo \cite{D88a, D88b, D91, D03}, D'Angelo--Kos--Riehl (\cite{DKR03}) and D'Angelo--Lebl \cite{DL09, DL16}.

In the case of bounded symmetric domains, Tsai \cite{Ts93} showed that if $f\colon \Omega\to\Omega'$ is a proper holomorphic map between bounded symmetric domains $\Omega$ and  $\Omega'$ such that $\Omega$ is irreducible and $\text{rank}(\Omega)\geq \text{rank}(\Omega'){\geq 2}$, then $\text{rank}(\Omega) = \text{rank}(\Omega')$ and $f$ is a totally geodesic isometric embedding, resolving in the affirmative a conjecture of Mok \cite[end of Chapter 6]{M89}. The proofs in Tsai (\cite{Ts93} are based on the method of Mok-Tsai \cite{MT92} on taking radial limits on $\Delta \times \Omega'$, where $\Omega'$ is a maximal characteristic subdomain of $\Omega$, in the disk factor $\Delta$ to yield boundary maps defined on maximal boundary faces, and on the idea of Hermitian metric rigidity of \cite{M87} \cite{M89}. For proper holomorphic maps with $\text{rank} (\Omega)< \text{rank}(\Omega')$ we refer the readers to  Chan \cite{C20, C21}, Faran \cite{F86}, Henkin-Novikov \cite{HN84}, Kim-Zaitsev \cite{KZ13, KZ15}, Mok \cite{M08c}, Mok-Ng-Tu \cite{MNT10}, Ng \cite{N13, N15a, N15b}, Seo \cite{S15, S16, S18} and Tu \cite{Tu02a, Tu02b}.
In particular, in \cite{KZ15}, Kim-Zaitsev showed that under the assumption that $p \ge q \ge 2, p' < 2p-1, q' < p$, any proper holomorphic map $f\colon D^I_{p,q}\to D^I_{p',q'}$ which extends smoothly to a neighborhood of a smooth boundary point must necessarily be of the form \begin{equation}\label{diagonal}
z\mapsto \left(\begin{array}{cc}
z&0\\
0&h(z)\
\end{array}\right),
\end{equation}
where $h(z)$ is an arbitrary holomorphic matrix-valued { map} satisfying 
$$
I_{q'-q} - h(z)^*h(z)>0  \text{ for any } z\in D_{p,q}^I.
$$
Here, $D^I_{p,q}$ denotes a bounded symmetric domain of type I (see \eqref{typeI}). 
Recently Chan \cite{C21} generalized their result to type~I domains by removing the smoothness assumption on the map.
Our first goal is to generalize the results of Kim--Zaitsev and Chan to cases in which $\Omega$ and $\Omega'$ are of the same type or $\Omega$ is of type III and $\Omega'$ is of type~I without requiring the existence of a smooth extension to the boundary. 

{For each Hermitian symmetric space of the compact type, there exist special subspaces which are called {\it characteristic subspaces}. They are defined using Lie { algebras} in \cite[Definition 1.4.2]{MT92}, and we also provide their detailed description in Section~\ref{Hermitian symmetric spaces}. }
\bd
Let $X$ and $X'$ be Hermitian symmetric spaces of {the} compact type.
A holomorphic map $f:X\to X'$ is called a {\em standard embedding} if there exists a characteristic subspace $X''\subset X'$ with ${\rm rank}(X'')={\rm  rank}(X)$ such that $f(X)\subset X''$ and $f:X\to X''$ is a totally geodesic isometric embedding with respect to (any choice of) the canonical K\"{a}hler-Einstein metric {of $X''$ up to normalizing constants}.
For {a nonempty connected} open set $U\subset X$, a holomorphic map $f\colon U \to X'$ is called a {\em standard embedding} if $f$ extends to $X$ as a standard embedding. 
\ed
{ It is worth mentioning that the canonical K\"ahler-Einstein metrics on $X''$ are induced from a K\"ahler-Einstein metric on $X'$.
}
\begin{Thm}\label{main}
Let $\Omega$ and $\Omega'$ be irreducible bounded symmetric domains { of} rank $q$ and $q'$, respectively.
Suppose
$$2\leq  q'<2q-1.$$
Suppose further that either
$(1)$ $\Omega$ and $\Omega'$ are of the same type or 
$(2)$ $\Omega$ is of type~III and $\Omega'$ is of type~I. 
Then, up to automorphisms of $\Omega$ and $\Omega'$, every proper holomorphic map $f\colon \Omega\to \Omega'$ is of the form
$f = \imath\circ F$, where
$$
	F=F_1\times F_2\colon \Omega\to \Omega_1'\times \Omega_2', 
$$
$\Omega_1'$ and $\Omega_2'$ are bounded symmetric domains, $F_1 \colon \Omega\to \Omega_1'$ is a standard embedding, {$F_2: \Omega\to\Omega_2'$ is a holomorphic mapping}, and $\imath: \Omega_1' \times \Omega'_2 \hookrightarrow \Omega'$ is a holomorphic totally geodesic embedding of a  bounded symmetric domain $\Omega_1' \times \Omega_2'$ into $\Omega'$ with respect to canonical K\"ahler-{Einstein} metrics. {Here, $\Omega_2
'$ is allowed to be a point.} As a consequence, every proper holomorphic map $f: \Omega \to \Omega'$, $f = \imath\circ  F$, is a holomorphic totally geodesic isometric embedding with respect to Kobayashi metrics.
\end{Thm}
We remark that in the case of type~I domains, our result (which supersedes \cite{C21}) is optimal.  In fact, when $q' = 2q-1$ there exists by Seo \cite{S15} a proper holomorphic map called a generalized Whitney map from $D_{p,q}^I$ to $D^I_{2p-1,2q-1}$ which is not equivalent to \eqref{diagonal}.  Note also that for $\Omega$ and $\Omega'$ of type~IV, both bounded symmetric domains are of rank 2 and rigidity follows from \cite{Ts93}. In the case of exceptional domains $D^{V}$ and $D^{VI}$ the theorem concerns only proper holomorphic self-maps which are again necessarily automorphisms by \cite{Ts93} (or already from the method of \cite{TK82}).

\bigskip

\bt\label{nonexistence}
There exists no proper holomorphic map from $\Omega$ to $\Omega'$, if one of the following holds:
\ben
\item $\Omega=D^I_{p,q}$ with {$q\leq p$
}, $\Omega'=D^{III}_{q'}$ and $ q'<2q-1.$\\
\item $\Omega=D^{II}_n$,  $\Omega'=D^I_{p',q'}$ {with $q'\leq p'$} or $D^{III}_{q'}$ and 
$ 2\leq q'<2[n/2]-1.$\\
\een
\et

The basic strategy for the proofs of Theorem~\ref{main} and Theorem~\ref{nonexistence} is to generalize a strategy used in the works of Mok-Tsai \cite{MT92} and Tsai \cite{Ts93} which consists of two main steps. In the first step, it was shown that any proper holomorphic map between bounded symmetric domains maps boundary components into boundary components. This result was then used in the second step under the assumption that the rank of the target domain is smaller than or equal to that of the source domain.  Under the latter assumption, a moduli map was constructed from the moduli space of maximal characteristic symmetric {subdomains} to that of characteristic symmetric subdomains of a fixed rank in the target domain, and the moduli map was proven to admit a rational extension between moduli spaces of characteristic symmetric subspaces.

If we assume that the difference between the rank of the target domain $q'$ { and} that of the source domain $q$ is positive, then for each rank $1\leq r<q$ we need to construct a moduli map $f^\flat_r\colon D_r(\Omega)\to {F}_{i_r}(\Omega')$ between the moduli spaces of subgrassmannians and show that this map also preserves the subgrassmannians $Z_\tau^r$ and  $Q_{\mu}^r$ (Lemma~\ref{Z-tau}, Lemma~\ref{Q-mu}).
It is worth pointing out that  there exists a one-to-one correspondence $r\mapsto i_r$ between the indices of the moduli spaces of the source and target domains (Lemma~\ref{increasing}), so that there exists $r$ such that $i_r=i_{r-1}+1$ {in the case of type-I and type-III Grassmannians, and $i_r=i_{r-1}+2$ in the case of type-II Grassmannians, and our rank condition is necessary to guarantee that $r$ exists.  The existence of $r$ is crucial to establish the fact that some moduli map associated with the proper holomorphic map $f: \Omega \to \Omega'$ is a trivial embedding, from which the form of $f$ as described in Theorem~\ref{main} can be recovered.}

{ The moduli map $f_r^\flat$ is holomorphic if $\Omega$ is of type II or type III for all $r$. On the other hand, if $\Omega$ is of type I, then $f^\flat_r$ for the case when $i_r=i_{r-1}+1$ is either holomorphic or anti-holomorphic. For instance, if $\Omega=\Omega'=D_{p,p}^I$, then the moduli map for { the} identity map is holomorphic while the moduli map of transpose map defined by 
$$
Z\in D_{p,p}^I \to Z^T\in D_{p,p}^I
$$
is anti-holomorphic. Here, $Z^T$ denotes the transpose matrix of $Z$. If a moduli map happens to be anti-holomorphic, we use { the} conjugate complex structure of the source moduli space. 
  }

After implementing the aforementioned strategy, for the completion of our proofs we will make use of rigidity phenomena for CR embeddings (as in \cite{K21}) in an essential way applied to certain CR hypersurfaces in moduli spaces of subgrassmannians, and rigidity results concerning geometric structures and substructures. Our lines of argumentation concord with the perspective put forth in Mok \cite{M16} of applying the theory of geometric structures and substructures modeled on varieties of minimal rational tangents to the study of proper holomorphic maps between bounded symmetric domains, and, in the special case of proper holomorphic maps from type~III domains to type~I domains, a novel element in our proof is the establishment of the rigidity phenomenon for admissible pairs of rational homogeneous manifolds not of the sub-diagram type as initiated in \cite{M19}. In the latter case our proof relies on the solution of the Recognition Problem for symplectic Grassmannians of Hwang-Li \cite{HL21}. { Both the aforementioned rigidity phenomenon and the Recognition Problem will be formulated in the framework of the geometric theory of uniruled projective manifolds $X$ equipped with minimal rational components $\mathcal K$, and we will need the basics of the theory as is given in Hwang-Mok \cite{HM99} and Mok \cite{M08b}, especially the notion of the variety of minimal rational tangents (VMRT) at a general point of $(X,\mathcal K)$ first defined in Hwang-Mok \cite{HM98}, and of the theory of sub-VMRT structures as given in Mok \cite{M16} and Mok-Zhang \cite{MZ19}.}

Our main technical result is presented in Section~\ref{Rigidity of the pair} and it deals with the rigidity of holomorphic maps which respect subgrassmannians. 
\begin{Def}
{Let $U \subset D_r(X)$ be a nonempty connected open subset.  A holomorphic map $H\colon U\to D_{r'}(X')$ is said to \emph{respect subgrassmannians} if for each $\tau \in {D}_r(X)$ and each connected component $U_\tau^\alpha$ of $U \cap Z_\tau$, $\alpha \in A$, there exists $\tau'(\alpha) \in D_{r'}(X') $ such that
\begin{enumerate}
\item
 $H(U_{\tau}^\alpha) \subset Z_{\tau'(\alpha)}$ and 
\item $H|_{U_\tau^\alpha}$ extends to a standard embedding from $Z_{\tau}$ to $Z_{\tau'(\alpha)}$.
\end{enumerate}}
\end{Def}
{Here, for the definition of $D_r(X)$ and $Z_\tau$, see \eqref{css1}, \eqref{css2}, \eqref{css3}, Definition~\ref{mathcal ZQ}, and Definition~\ref{DZQ}.}

Under the {assumptions} of Theorem~\ref{main} and the {additional condition} that $H$ maps a CR submanifold $\Sigma_r(\Omega)$ to $\Sigma_{i_r}(\Omega')$, Proposition~\ref{H respects} says that the map is a trivial embedding. 
Here $\Sigma_r(\Omega)$ and $\Sigma_{i_r}(\Omega')$ are canonically defined CR submanifolds in $D_r(X)$ and $D_{i_r}(X')$ respectively.
This generalizes a result of the first author \cite{K21} (cf. \cite{N12}) on the rigidity of CR embeddings between $SU(\ell, m)$-orbits in the Grassmannian of $q$-planes in $\mathbb C^{p+q}$ where $m=p+q-\ell$.

The proof of Proposition~\ref{H respects} will be given in several steps. First, we will show that the $1$-jet of $H$ coincides with { that of} a trivial embedding and that $H$ maps connected open subsets of projective lines { into} projective lines (Lemma~\ref{one jet of H}). We remark that if $X$ is of type~I or type~II, then for any projective line $L\subset D_r(X)$, there exists a subgrassmannian $Z_{\tau}$ such that $L\subset Z_{\tau}$. Since $H$ respects subgrassmannians, $H$ sends (open subsets of) projective lines into projective lines. 
Type III domains require special attention (Lemma~\ref{hyperquadric}). 
If the map is defined between domains of the same type, in view of {Theorem~1.1 and Proposition~3.4 of \cite{HoM10}} and Lemma~\ref{exist_subgrassmannian}, the proof is complete. 
{Theorem~1.2 of \cite{HoM10}} is a generalization of Cartan-Fubini type extension results obtained by Hwang-Mok in \cite{HM01}, to the situation of non-equidimensional holomorphic mappings modeled on pairs $(X_o,X)$ of the subdiagram type. We refer readers to \cite{KO81, HM01, HM04, M08a, HoM10, HN21} for developments in this direction.

On the other hand, if the source domain is of type~III and the target domain is of type~I, then we need to make use of \cite[Section 6]{M19}.
In \cite{M19}, the second author gave sufficient conditions for the rigidity of an admissible pair $(X_o, X)$ which is not of {the} subdiagram type. As a consequence, he used this result to prove that the admissible pair $(SGr(n, \mathbb C^{2n}), Gr(n,\mathbb C^{2n}))$ is rigid. We { partially} generalize the latter result to the admissible pair $(SGr(q, \mathbb C^{2n}), Gr(q, \mathbb C^{2n}))$, $2\leq q\leq n$, { under the additional assumption $(\sharp)$ that the support $S \subset Gr(q,\mathbb C^{2n})$ of a sub-VMRT structure $(S,\mathscr C_{X}\cap\mathbb PT_S)$ modeled on $(X,X')$ is the image of a VMRT-respecting holomorphic embedding {$H$} from a connected open subset of $SGr(q,\mathbb C^{2n})$ into  $Gr(q,\mathbb C^{2n})$ which transforms any connected open subset of a minimal rational curve into a minimal rational curve, in which case it is known that the holomorphic embedding admits a rational extension (cf. \cite{HoM10}). Alternatively, rational extension of {$H$} also follows from the Hartogs phenomenon as applied in Mok-Tsai \cite{MT92}, an argument which we have made use of in the current article to prove rational extension { in Lemma~\ref{rational extension} and Lemma~\ref{commute}.}}

It is possible, along the line of thoughts of \cite{M19}, to entirely remove the assumption $(\sharp)$ {(cf. Remark~\ref{remark_section4}~(a))}, but we will refrain from proving the (full) rigidity of the admissible pair $(SGr(q, \mathbb C^{2n}), Gr(q, \mathbb C^{2n}))$ as that is not needed for the current article. For the notion of admissible pairs and the rigidity of the admissible pairs of the subdiagram type, see \cite{MZ19}.

{ It is worth noting that 
the { analogues} of Theorem~\ref{main} and Theorem~\ref{nonexistence} involving $\Omega$ of type~I or type~III and $\Omega'$ of type~II are not covered in the current article and would be a natural continuation to our work.}

The organization of the current article is as follows.
In Section~\ref{preliminaries}, we describe the moduli
spaces $\mathcal D_r(X)$ and $D_r(X)$ of characteristic subspaces in $X$.
In Section~\ref{Subgrassmannians in the moduli spaces}, we present the subgrassmannians of $\mathcal D_r(X)$ and $D_r(X)$. Then we explain the CR structure of the unique closed orbit $\Sigma_r(X)$ in $D_r(X)$.
In Section~\ref{Rigidity of subgrassmannian respecting holomorphic maps}, we investigate the rigidity of subgrassmannian respecting holomorphic maps between $D_r(X)$ and $D_{r'}(X')$.  {For the treatment of this topic the cases where $X$ and $X'$ are of the same type I, II or III leads us eventually to the rigidity phenomenon for admissible pairs of {\it the subdiagram type} of irreducible compact Hermitian symmetric spaces, which was already established in \cite{HoM10} (in the more general context of rational homogeneous spaces), whereas the case where $X$ is a Lagrangian Grassmannian $LGr_n$ (i.e., $X$ is of type III) leads to a rigidity problem for admissible pairs of {\it non-subdiagram} type.  In order to proceed with Section~\ref{Rigidity of subgrassmannian respecting holomorphic maps} in a way that incorporate all pairs $(X,X')$ being considered in the article, we first consider in Section~\ref{Rigidity of the pair} the rigidity phenomenon for the pair $(SGr(q,\mathbb C^{2n}),Gr(q,\mathbb C^{2n}))$. Section~\ref{Rigidity of subgrassmannian respecting holomorphic maps} then consists of several lemmas to prove Proposition~\ref{H respects}}, which is the main technical result in this paper. 
In {Section~\ref{Induced moduli map}}, we {define} moduli maps $f_r^\sharp$ (resp. $f^\flat_r$) between $\mathcal D_r(X)$ (resp. $D_r(X)$) and $\mathcal D_{r'}(X')$ (resp. $D_{r'}(X')$) which are induced by a proper holomorphic map between $\Omega$ and $\Omega'$.
In Section~\ref{Rigidity of the induced moduli map}, we {show} that 
$f^\flat_r$ is a subgrassmannian respecting holomorphic map and extends to a standard holomorphic embedding for some $r$.
Finally in Section~\ref{Proof of Theorems}, we prove Theorem~\ref{main} and \ref{nonexistence}. {In Section~\ref{Appendix} we prove some results from the method of moving frames that have been used in the article.}

\medskip

{\bf Acknowledgement} The first author was supported by the Institute for Basic Science (IBS-R032-D1-2021-a00). The second author was supported by the GRF-grant { 17306523} of the HKRGC (Hong Kong Research Grants Council). 
The third author was partially supported by Basic Science Research Program through the National Research Foundation of Korea (NRF) funded by the Ministry of Education (NRF-2022R1F1A1063038).

\section{preliminaries}\label{preliminaries}
\subsection{Hermitian symmetric spaces}\label{Hermitian symmetric spaces}
{
Let $(X_o,g_o)$ be an irreducible Hermitian symmetric space of the noncompact type and denote by $G_o$ the identity component of its automorphism group of biholomorphic self-maps (which are necessarily isometries with respect to $g_o$), which is a K\"ahler-Einstein metric of negative Ricci curvature. Let $K \subset G_o$ be a maximal compact subgroup, so that $X_o = G_o/K$ as a homogeneous space and $K \subset G_o$ is the isotropy subgroup at $o := eK$, $e \in G_o$ being the identity element. { The structure of $X_o = G_o/K$ corresponds to a simple orthogonal symmetric Lie algebra $(\mathfrak{g},\mathfrak{k};\theta)$ (cf. Helgason \cite[Chapter IV, Proposition 3.5]{Hel78}) where $\theta \in {\rm End}(\mathfrak{g}_o)$ is a Lie algebra automorphism such that $\mathfrak{k} \subset \mathfrak{g}_0$ is precisely the subset of elements fixed by $\theta$ (which is the differential at $e \in \mathfrak{g}_o$ of an inner automorphism $\tau$ of $G_o$, $\tau(g) = s^{-1}gs$ for some element $s$ belonging to the center $Z(K) \cong \mathbb S^1$ of $K$, cf. e.g., Mok \cite[p.49]{M89}). }

In what follows, for a Lie group denoted by a Roman letter we denote the associated Lie algebra by the corresponding Gothic letter, and vice versa. Write $\mathfrak g_o = \mathfrak k \oplus \mathfrak m$ for the Cartan decomposition of $\mathfrak g_o$ at $o$ which is the eigenspace decomposition of $\theta$ on $\mathfrak g_o$ corresponding to the eigenvalues $1$ and $-1$ respectively. There is an element $z$ in the center $\mathfrak z$ of $\mathfrak k$ such that ${\rm ad}(z)|_{\mathfrak m}$ is the almost complex structure at $o$. We write $\mathfrak g$ for the complexification of $\mathfrak g_o$, and $G$ for the complexification of $G_o$ so that $G_o \hookrightarrow G$ canonically. { We have the Harish-Chandra decomposition $\mathfrak g = \mathfrak m^+ \oplus \mathfrak k^\mathbb C \oplus \mathfrak m^-$ (where for a real vector space $V$ we denote by $V^\mathbb{C}$ the complexification $V\otimes_{\mathbb R}\mathbb C$), which is the eigenspace decomposition for ${\rm ad}(z)$, extended by complex linearity as an element of ${\rm End}_{\mathbb C}(\mathfrak g)$, corresponding to the eigenvalues $\sqrt{-1}, 0$ and $-\sqrt{-1}$ respectively.} Writing $\mathfrak p := \mathfrak k^\mathbb C \oplus \mathfrak m^- \subset \mathfrak g$, $\mathfrak p \subset \mathfrak g$ is a parabolic subalgebra, and $G/P$ is the presentation as a complex homogeneous space of a Hermitian symmetric space $X$ of the compact type dual to $X_o$.  The canonical embedding $G_o \hookrightarrow G$ induces a holomorphic map $X_o = G_o/K \hookrightarrow G/P = X$, which is the Borel embedding realizing $X_o$ as an open subset of $X$. At $o = eK \in G_o/K \hookrightarrow G/P$ write $\mathfrak{m}^+ := T_o(X_o) = T_o(X)$.  The Harish-Chandra embedding theorem gives a holomorphic embedding $\tau: \mathfrak{m}^+ \rightarrow X$ onto a Zariski open subset in $X$ such that $\Omega := \tau^{-1}(X_o) \Subset \mathfrak{m}^+$ is a bounded symmetric domain on the Euclidean space $\mathfrak{m}^+ \cong \mathbb C^n$, where $n = \dim_{\mathbb C}(X)$ (cf.\cite{W72}).}

{
Write $\mathfrak g_c = \mathfrak{k} \oplus \sqrt{-1}\mathfrak{m} \subset \mathfrak{g}$.  Then, $\mathfrak g_c$ is the Lie algebra of a compact Lie subgroup $G_c \subset G$. The Lie groups $G_c$ and $G_o$ are respectively a compact real form and a noncompact real form of the simple complex Lie group $G$ (with a trivial center), such that $G_c \cap G_o = K$. $G_c$ acts transitively on $X$, and, extending $\theta \in {\rm End}(\mathfrak g_o)$ by complex linearity to $\mathfrak g$, $\mathfrak g_c$ is stable under $\theta$, and $(\mathfrak g_c,\mathfrak k;\theta)$ is a simple orthogonal Lie algebra underlying $X_c := G_c/K$ as an irreducible Hermitian symmetric space of the compact type.  There is a $G_c$-invariant K\"ahler-Einstein metric $g_c$ of positive Ricci curvature such that $((X_c,g_c),(X_o,g_o))$ is a dual pair of irreducible Hermitian symmetric spaces of the semisimple type. Moreover, $\mathfrak g_c = \mathfrak{k} \oplus \sqrt{-1}\mathfrak{m} \subset \mathfrak{g}$ is the Cartan decomposition of $\mathfrak g_c$. In what follows we will identify $X_c = G_c/K$ with $X = G/P$ via the biholomorphism induced from the inclusion $G_c \subset G$. 

%We have the Cartan decomposition $\mathfrak g_o = \mathfrak{k} \oplus \mathfrak{m}$ which is the eigenspace decomposition of $\theta$ corresponding respectively to the eigenvalues $1$ and $-1$. Likewise we have the Cartan decomposition $\mathfrak g_o = \mathfrak{k} \oplus \sqrt{-1}\mathfrak{m}$. We have the Harish-Chandra decomposition $\mathfrak{g} = \mathfrak{k}^{\mathbb C} \oplus \mathfrak{m}^{\mathbb C} = \mathfrak{k}^{\mathbb C} \oplus \mathfrak{m}^{+} \oplus \mathfrak{m}^{-}$ (where $\mathfrak{m}^{\mathbb C} = \mathfrak{m}^{+} \oplus \mathfrak{m}^{-}$ is the eigenspace decomposition of the almost complex structure $J$ on $T_o^{\mathbb C}(X)$). 

Let $\mathfrak{h} \subset \mathfrak{g}$ be a Cartan subalgebra lying inside $\mathfrak{k}^\mathbb{C}$, and $\Delta := \{\alpha_1,\cdots,\alpha_s\}$ be a full set of simple roots of $\mathfrak{g}$ with respect to $\mathfrak{h}$, and $\Phi$ be the set of all $\mathfrak{h}$-roots of $\mathfrak{g}$, so that $\mathfrak{g} = \mathfrak{h} \oplus \left(\bigoplus_{\varphi \in \Phi} \mathfrak{g}^{\varphi}\right)$, where $\mathfrak{g}^\varphi \subset \mathfrak g$ is the (complex 1-dimensional) root space associated to the root $\varphi \in \Phi$. (Here and in what follows by a root we always mean an $\mathfrak{h}$-root in $\mathfrak{g}$, i.e., an element of $\Phi$.) There are the standard notions of positive roots and negative roots, and of compact and noncompact roots in $\Phi$, so that, $\pi \in \Phi$ is called a positive root if and only if it is an integral combination $n_1\alpha_1+\cdots + n_s\alpha_s$, where each $\alpha_k$ is a nonnegative integer for $1 \le k \le s$ and $(\alpha_1,\cdots,\alpha_s) \neq 0$, and $\varphi \in \Phi$ is called a compact root if and only if $\mathfrak{g}^\varphi \subset \mathfrak{k}^{\mathbb C}$, otherwise $\varphi$ is called a noncompact root.  Two distinct roots $\varphi_1, \varphi_2 \in \Phi$ are said to be strongly orthogonal if and only if neither $\varphi_1 + \varphi_2$ nor $\varphi_1 - \varphi_2$ is a root. 

Note that the notation $\mathfrak{m}^{+}$ { has been} given two interpretations: (a) as the holomorphic tangent space $T_oX$ and (b) as a complex vector subspace of $\mathfrak{g}$ in the Harish-Chandra decomposition $\mathfrak{g} = \mathfrak{k}^{\mathbb C} \oplus \mathfrak{m}^{\mathbb C} = \mathfrak{k}^{\mathbb C} \oplus \mathfrak{m}^{+} \oplus \mathfrak{m}^{-}$.  Regard $\mathfrak{g}$ as the Lie algebra of holomorphic vector fields on $X = G/P$.  If we identify $u \in T_o(X) \cong \mathfrak{m}^{+}$ with the holomorphic vector field $u'$ in $\mathfrak{m}^{+} \subset \mathfrak{g}$ that it corresponds to as a result of the two interpretations of $\mathfrak{m}^{+}$ as given above, the holomorphic embedding $\tau: \mathfrak{m}^{+} \to X$ is given by $\tau(u) = {\rm exp}(u')(e)\ {\rm mod} \ P \in G/P = X$.}

%Owing to the lack of a single reference giving all we need for the background to the last paragraph, to put things in perspective we give here more details on ``characteristic subspaces'' and ``characteristic subdomains'' in the general setting of irreducible bounded symmetric domains embedded in their dual Hermitian symmetric spaces of the compact type by the Borel embedding, illustrating at the same time the role played by the totally geodesic holomorphic embeddings of $\Delta^{q-r} \times \Omega_0$ in $\Omega$ in the study of proper holomorphic maps.

{
%By definition we have $\tau(u) = {\rm exp}(u)(e)\ {\rm mod} P \in G/P = X$, in which $u \in \mathfrak{m}^+$ is identified   

%It is known that $K^{\mathbb C}$ and $G$ have the same rank as complex Lie groups. 
%We use Gothic letters to denote the Lie algebra of Lie groups, hence $\mathfrak{g}$ is the (complex) Lie algebra of $G$, and $\mathfrak{g}_o$ is the (real) Lie algebra of $G_o$, etc.  

Let $\Pi \subset \Phi$ be a maximal set of mutually strongly orthogonal positive noncompact roots. We have $|\Pi| = {\rm rank}(\Omega) = q$.  Let $\Lambda \subsetneq \Pi$ be nonempty, $1 \le r := |\Lambda| < q$. In \cite{MT92} the authors defined a characteristic subspace $X_{\Lambda,o} \subset X_o$ which is a totally geodesic complex submanifold in $(X_o,g_o)$ passing through $o \in X_o$ together with a characteristic subspace $X_{\Lambda} \subset X$ which is a totally geodesic complex submanifold in $(X,g_c)$ such that $X_{\Lambda,o} \subset X_{\Lambda}$.  $(X_{\Lambda,o}, X_{\Lambda})$ is a dual pair of irreducible Hermitian symmetric spaces, and the inclusion $X_{\Lambda,o} \subset X_{\Lambda}$ is the Borel embedding.  Moreover, $X_{\Lambda,o}$ corresponds under the holomorphic embedding $\tau: \mathfrak{m}^{+} \to X$ to $\tau^{-1}(X_{\Lambda,o}) =: \Omega_{\Lambda} \subset \Omega$. By \cite[Proposition 1.12]{MT92}, $\Omega_{\Lambda} \subset \Omega$ is of the form $\Omega_\Lambda = \mathfrak{m}_{\Lambda}^{+} \cap \Omega$ for the complex linear subspace $\mathfrak{m}_{\Lambda}^{+} \subset \mathfrak{m}^{+}$ identified with $T_o(X_\Lambda)$.  By a case-by-case checking each $X_{\Lambda}$ (and hence each $X_{\Lambda,o}$) is an irreducible Hermitian symmetric space. We have ${\rm rank}(X_{\Lambda}) = {\rm rank}(X_{\Lambda,o}) = |\Lambda| = r$.

From the Restricted Root Theorem it is well-known from Wolf \cite{W72} (as given in \cite[Proposition 1.4.1]{MT92}) that whenever $\Lambda_1$ and $\Lambda_2$ are of the same cardinality, there exists $k \in K$ such that $X_{\Lambda_2} = k(X_{\Lambda_1})$ (hence also $X_{\Lambda_2,o} = k(X_{\Lambda_1,o})$) . By a {\it characteristic subspace} of $(X,g_c)$ we mean $h(X_{\Lambda})$ for some $h \in G_c$ and for some $\Lambda$.  From \cite[Proposition 1.12]{MT92}  and \cite[Lemma 4.4]{Ts93} characteristic subspaces of $X$ are {\it invariantly geodesic} in the sense that $\gamma(X_{\Lambda})$ is totally geodesic in $(X,g_c)$ for any $\gamma\in G$ and an $\Lambda$. In particular, for a characteristic subspace $Y \subset X$ passing through $o =eP \in G/P \cong X$ and for $\gamma \in M^- = {\rm exp}(\mathfrak{m}^-)$, $\gamma Y \subset X$ is totally geodesic in $(X,g_c)$ while $T_o(\gamma Y) = T_o(Y)$ (since $\mathfrak{m}^-$ consists of holomorphic vector fields vanishing to the order 2, hence $d\gamma(o) = {\rm id}_{T_o(X)}$), and it follows from uniqueness properties of totally geodesic complex submanifolds that $\gamma Y = Y$.  Moreover, by \cite[{\it loc. cit.}]{MT92}, for any $\gamma \in G$ such that $A := \gamma(X_{\Lambda}) \cap \mathfrak{m}^{+} \neq \emptyset$, $A \subset \mathfrak{m}^{+} $ is a complex affine subspace.  

By a {\it characteristic subspace} of $(X_o,g_o)$ we mean $h(X_{\Lambda,o})$ for some $h \in G_o$ and for some $\Lambda$, and a {\it characteristic subdomain} $\Omega' \subset \Omega$ is simply $\tau^{-1}(Y_o)$ for some characteristic subspace $Y_o \subset X_o$. For $1 \le r < q$ we see from the Restricted Root Theorem that there is up to the action of $G$ (resp. $G_o$) only one isomorphism class of characteristic subspaces $Y \subset X$ (resp. $Y_o \subset X_o)$ of rank $r$, thus also only one isomorphism class of characteristic subdomains $\Omega' \subset \Omega$ of rank $r$ under the natural action of $G_o$ on $\Omega$. 

%For $1 \le r < q$ denote by $\mathfrak C_r$ the moduli space of characteristic subspaces of $X$ of rank $r$. From the definition $G_c$ acts transitively on $\mathfrak C_r$.  From the fact that characteristic subspaces of $X$ are invariantly geodesic $G$ acts also transitively on $\mathfrak C_r$, from which it follows readily that $\mathfrak C_r$ admits the structure of a complex homogeneous manifold $\mathfrak C_r = G/Q$ for some closed complex subgroup $Q \subset g$, which is parabolic since the compact Lie group $G_c$ acts transitively on it.  In other words, $\mathfrak C_r = G/Q$ admis the structure of a rational homogeneous manifolder.  Denote by $\mathfrak D_r$ the moduli space of characteristic subdomains $\mathcal D \subset \Omega$ of rank $r$. By the description $D = Y \cap \Omega$ for some characteristic subspace $Y \subset X$ of rank $r$ it follows that $\mathfrak D_r$ can be identified as an  open subset 

For $1 \le r < q$ the complex Lie group $G$ acts on the set $\mathfrak C_r$ of characteristic subspaces of $X$ of rank $r$, hence $\mathfrak C_r$ admits the structure of a complex homogeneous manifold. From the definition, $G_c \subset G$ already acts transitively on $\mathfrak C_r$. It follows that $\mathfrak C_r$ is compact, hence $\mathfrak C_r$ is a rational homogeneous manifold given by $\mathfrak C_r = G/Q$ for some parabolic subgroup $Q\subset G$. Denote by $\mathfrak D_r$ the moduli space of characteristic subdomains ${\Omega'_r} \subset \Omega$ of rank $r$. By the description ${\Omega_r'} = Y \cap \Omega$ for some characteristic subspace $Y \subset X$ of rank $r$ it follows that $\mathfrak D_r$ can be identified as an open subset (in the complex topology) of $\mathfrak C_r$, and it was proven in \cite{MT92} by Oka's characterization of domains of holomorphy that every meromorphic function on $\mathfrak D_r$ extends to a rational function on $\mathfrak C_r$, an intermediate result essential for both \cite{MT92} and \cite{Ts93}.

Now suppose $\Omega' \subset \Omega$ is a characteristic subdomain.  Write $\Omega' = h(\Omega_\Lambda)$ for some characteristic subdomain $\Omega_\Lambda = \tau^{-1}(X_{\Lambda,o})\Subset \mathfrak{m}_\Lambda^{+}$ and for some $h \in G_o$.  Recall that $\Pi = \{\pi_1,\cdots,\pi_r\}$ is a maximal set of mutually strongly orthogonal positive noncompact roots.  Consider $\Lambda_1 := \Pi - \{\psi_1\}$ for any element $\psi_1 \in \Pi$.  A nonzero vector $\alpha \in \mathfrak{g}^{\psi_1}$ is a minimal rational tangent in the sense that $\alpha \in T_{o}\ell_\alpha$ for some minimal rational curve $\ell_\alpha \subset X$ passing through $o$.  Note that $\Delta_\alpha := \ell_\alpha \cap \Omega$ is a minimal disk on $\Omega$. Write $L: = SU(1,1)/\{\pm I_2\}$. By \cite[Proposition 1.7]{MT92}, there exists an $(L \times G_{\Lambda_1,o})$-equivariant holomorphic totally geodesic embedding of $\Delta \times \Omega_{\Lambda_1}$ into $\Omega$, written here as $\beta_1: \Delta \times \Omega_{\Lambda_1} \to \Omega$, where $G_{\Lambda_1,o} \subset G_o$ is a noncompact real form of $G_{\Lambda_1} \subset G$, in which the Lie algebra $\mathfrak{g}_{\Lambda_1}$ of $G_{\Lambda_1}$ is the derived algebra of $\mathfrak h + \bigoplus_{\rho\perp \psi_1}\mathfrak{g}^\rho$, where for $\rho_1, \rho_2 \in \Phi$, $\rho_1\perp \rho_2$ if and only if $B(\rho_1,\overline{\rho_2}) = 0$ for the Killing form $B(\cdot,\cdot)$ of $\mathfrak{g}$. Noting that $\Omega_{\Lambda_1}$ is irreducible, by induction it follows that for any $\Lambda, \emptyset \neq \Lambda \subsetneq \Pi$, writing $\Psi = \{\psi_1,\cdots,\psi_{r}\}$ and $\Lambda = \Pi - \Psi$, there exists an $(L^{q-r}\times G_{\Lambda,o}$)-equivariant holomorphic totally geodesic embedding $\beta: \Delta^{q-r} \times \Omega_{\Lambda}$ into $\Omega$, where $G_{\Lambda,o} \subset G_o$ is a noncompact real form of $G_\Lambda \subset G$, in which the Lie algebra $\mathfrak{g}_\Lambda$ of $G_{\Lambda}$ is the derived algebra of $\mathfrak h + \bigoplus_{\rho\perp\Psi}\mathfrak{g}^\rho$, $\rho\perp\Psi$ meaning $\rho\perp\psi$ for all $\psi \in \Psi$.  Note that $\beta$ extends to a holomorphic embedding, still to be denoted $\beta$, of $(\mathbb P^1)^{q-r} \times X_\Lambda$ into $X$ (when we identify $\Omega$ with $X_o$).  In particular, $\beta$ is defined and continuous on $\overline{\Delta^{q-r} \times \Omega_{\Lambda}}$. 

{The topological boundary $\partial \Omega$ of $\Omega$ decomposes into a disjoint union $\bigcup_r S_r$ of $G_o$-orbits $S_r, ~r=0,\ldots, q-1$. {To emphasize $X$ or $\Omega$, we will occasionally write $S_r$ as $S_r(X)$ or $S_r(\Omega)$ in the future.} Each $S_r$ is foliated by maximal complex manifolds called {\it boundary components} of $\Omega$. For the definition of boundary components, see \cite[Part I, 5. Boundary Components]{W72}.}

The boundary components of $\Omega$ of rank $r$ lying on $\overline{\beta(\Delta^{q-r} \times \Omega_{\Lambda})}$ are of the form $\beta(\{a\}\times \Omega_\Lambda)$, where $a \in (\partial \Delta)^{q-r}$. The group $G_o$ acts transitively on the moduli space of boundary components of $\Omega$ of any fixed rank. Write $B_1 = \beta(\{(1,\cdots,1)\}\times \Omega_\Lambda)$.  Then, for any boundary component $B \subset \partial\Omega$ of rank $q-r$, there exist $\gamma \in G_o$ such that $B = \gamma(B_1)$.  Then $\gamma\circ\beta: (\Delta^{q-r} \times \Omega_{\Lambda}) \to \Omega$ is an $(L^{q-r}\times G_{\Lambda,o}$)-equivariant holomorphic totally geodesic embedding whose image contains $B$ in its topological closure, as described.

Write $\Sigma \subset \Omega$ for the image of $\gamma\circ\beta$. Then, $\Sigma \subset \Omega$ is a holomorphically embedded totally geodesic copy of $\Delta^{q-r} \times \Omega_{\Lambda}$ such that $B \subset \overline{\Sigma}$. We note that such a complex submanifold $\Sigma \subset \Omega$ is not unique.  In fact $\Sigma' = \nu(\Sigma)$ plays the same role as $\Sigma$ for any $\nu \in G_o$ belonging to the stabilizer subgroup $N \subset G_o$ of $B \subset S_r \subset \partial\Omega$. (Since the moduli space $\mathfrak C_r$ is compact and the moduli space of all $\Sigma = \delta(\Delta^{q-r} \times \Omega_{\Lambda}), \delta := \gamma\circ\beta$ as in the above, is easily checked to be noncompact, $N$ does not stabilize $\Sigma$.)  Noting that $K = {\rm Aut}_o(\Omega)$ acts transitively on the moduli space $\mathfrak{B}_r$ of boundary components of rank $r$ (cf. Wolf \cite[p. 287]{W72}), in the previous paragraph we may choose $\gamma = \kappa \in K$, in which case $\Sigma \subset \Omega$ passes through $o \in \Omega$. We observe that there is a unique $\Sigma = ({\kappa}\circ\beta)(\Delta^{q-r} \times \Omega_{\Lambda}) \subset \Omega$ such that $o \in \Sigma$ and $B\in \overline{\Sigma}$.  To see this, it suffices to note that the complex submanifold $\Sigma \subset \Omega$, being totally geodesic and passing through $o \in \Omega$, is uniquely determined by the holomorphic tangent space $T_o\Sigma \subset T_o\Omega$, which is in turn determined by $B \subset S_r$. However, it can readily be checked that the set of all ${\Omega'_r} := (\kappa\circ\beta)(\{z\} \times \Omega_\Lambda)$, $\kappa \in K$, $z \in \Delta^{q-r}$, thus obtained does not exhaust all characteristic subdomains of rank $r$ on $\Omega$.  The approach of studying proper holomorphic maps in \cite{MT92} and \cite{Ts93} was to deduce properties on the restriction of proper holomorphic maps to characteristic subdomains from properties of radial limits (thus the role of the polydisk factor $\Delta^{q-r}$) of proper holomorphic maps on boundary components, hence the necessity for introducing holomorphic embeddings of $\Delta^{q-r} \times \Omega_{\Lambda}$ accounting for all characteristic subdomains.} 

% (hence $\mathfrak{B}_q$ is compact and $N \subset G_o$ is parabolic0
%(The isotropy subgroup $K = {\rm Aut}_o(\Omega)$ acts transitively on the moduli space of boundary components of rank $q$, hence $G_o/Q$ is a compact manifold, and $N \subset G_o$ is a parabolic subgroup.). 

%If we denote by $\tau_{\Lambda}: \mathfrak{m}_{\Lambda,o}^{+} \to X_{\Lambda}$, $\Omega_{\Lambda} = \tau_{\Lambda}^{-1}(X_{\Lambda, o}$, from the definition of the Harish-Chandra embeddings it follows readily that the Harish-Chandra realization $\tau_{\Lambda}^{-1}: X_{\Lambda,o} \hookrightarrow \mathfrak{m}^{+}$ as a bounded domain is the restriction of $\tau^{-1}: X_o \overset{\cong}\longrightarrow \Omega$ to $X_{\Lambda,o}$. 

%Moreover, $\Omega_{\Lambda,0} \Subset  \Omega_0 = \mathfrak{m}_{\Lambda}}^{+}$ is the Harish-Chandra realization of $X_{\Lambda,o}$, i.e., the image of the Harish-Chandra embedding $\tau_{\Lambda}^{-1}: X_{\Lambda,o} \to \mathfrak{m}_{\Lambda}^{+}$, where $\tau_\Lambda: \mathfrak_{\Lambda}^{+} \to X_\Lambda$ is the Harish-Chandra embedding of $\mathfrak{m}_{\Lambda}^{+}$ onto a Zariski open subset of $X_\Lambda$, and $\tau_\Lambda$ is the restriction of the Harish-Chandra embedding $\tau|_{\\mathfrak{m}^{+}:  = \tau|_{X_{\Lambda,o}$.  
%\end{document} 
 
% In [MT92], $\Omega_o := \tau^{-1}(X_{\Lambda,o}) \subset \Omega$ is called a characteristic subdomain. 

%{The boundary components, characteristic subdomains and their embeddings are canonically defined Lie theoretically. For more detail, see \cite{MT92}, { \S1, p.94ff}.}

\subsection{Moduli spaces of Hermitian symmetric subspaces and characteristic subspaces}\label{moduli spaces}
{Each irreducible Hermitian symmetric space is associated with a { Dynkin diagram marked at a single node}, and any Hermitian symmetric subspace corresponding to a marked subdiagram of the marked Dynkin diagram is termed a subspace of subdiagram type.}
In this subsection we describe moduli spaces of {certain} Hermitian symmetric subspaces of subdiagram type and characteristic subspaces
in the irreducible Hermitian symmetric space of type I, II, III.
We refer the reader to \cite{W69} and \cite[Part III]{W72} for more details.

\begin{enumerate}\label{subgrassmannian}
\item\label{Grassmannian} 
Let $X$ be the complex Grassmannian $Gr(q,p)$ consisting of $q$-planes passing through the origin in $\mathbb C^{p+q}$.
Then {$G = SL(p+q, \mathbb C)/\mu_{p+q}I_{p+q}$, where $\mu_m$ stands for the group of $m$-th roots of unity, and $I_m$ stands for the $m$-by-$m$ identity matrix,} and for any $A\in SL(p+q,\mathbb C)$, 
$A$ acts on $\Lambda^q\left({\mathbb C}^{p+q}\right)$ by
\begin{equation}\label{action}
A(w_1\wedge \cdots \wedge w_q) = Aw_1 \wedge \cdots \wedge Aw_q,
\end{equation}
where $w_1, \ldots, w_q \in \mathbb C^{p+q}$.  Taking $w_1, \ldots, w_q$ to be linearly independent and identifying $Gr(q,p)$ with its image in $\mathbb P\left(\Lambda^q\left({\mathbb C}^{p+q}\right)\right)$ under the Pl\"ucker embedding, we have the induced action of $A \in SL(p+q,\mathbb C)$ on $Gr(q,p)$.
A subgrassmannian in $Gr(q,p)$ is the set of all elements $x\in Gr(q, p)$ such that 
\begin{equation}\label{subgra1}
V_1 \subset x \subset V_2
\end{equation}
for given complex vector subspaces $V_1, V_2 \subset \mathbb C^{p+q}$.  
Hence, for fixed positive integers $a \leq  b$, the moduli space of subgrassmannians with $\dim V_1 =a$, $\dim V_2 = b$ is the flag variety 
%of type $A_{p+q-1}$ given by 
\begin{equation}\label{flag}
\mathcal F(a,b ;\mathbb C^{p+q})=\{ (V_1, V_2) : \{0\} \subset V_1 \subset V_2\subset \mathbb C^{p+q}, \dim V_1 = a, \, \dim V_2 = b \}.
\end{equation}
Since $Gr(p,q)$ is biholomorphic to $Gr(q,p)$, without loss of generality we will assume from now on $q \le p$, so that $Gr(q,p)$ is of rank $q$.
For $(V_1, V_2)\in \mathcal F(a,b;\mathbb C^{p+q})$ we {denote} the corresponding
subgrassmannian by $X_{(V_1, V_2)}$.
{We denote the moduli space of subgrassmannians where} $\dim V_1 = q-r$, $\dim V_2 = p+r$ for $r=1,\ldots, q-1$ by $\mathcal D_r(X)$, i.e.,
\begin{equation}\label{css1}
\mathcal D_r(X)=\{ (V_1, V_2) : \{0\} \subset V_1 \subset V_2\subset \mathbb C^{p+q}, \dim V_1 = q-r, \, \dim V_2 = p+r \}.
\end{equation}

\item \label{Orthogonal}
Let $X$ be the orthogonal Grassmannian $OGr_n$ consisting of $n$-planes passing through the origin
in $\mathbb C^{2n}$ isotropic {with respect} to a nondegenerate symmetric 
bilinear form $S_n = \left( \begin{array}{cc}
0& I_n\\
I_n& 0
\end{array}\right)
$ on $\mathbb C^{2n}$. 
Note that $q:=\text{rank of } OGr_n = \left[ \frac{n}{2} \right]$. In this case 
{$G = SO(2n, \mathbb C)/\left\{\pm I_{2n}\right\}$} and it acts on 
$OGr_n$ by \eqref{action}.
Consider a subgrassmannian in $OGr_n$ which is the set of all elements $x\in OGr_n$ such that 
\begin{equation}\label{subgra2}
 V \subset x \subset V^\perp
\end{equation}
 for a given isotropic complex vector subspace $V\subset \mathbb C^{2n}$ {with respect to $S_n$}, 
 where $V^\perp$ denotes the {annihilator} of $V$ with respect to $S_n$.
Let $\mathcal D_r(X)$ {and $\mathcal D_{r, \frac12}(X)$} denote the moduli spaces of such subgrassmannians in $OGr_n$, i.e.,
for $r=1,\ldots, q-1$
{
\begin{equation}\label{css2}
\begin{aligned}
\mathcal D_r(X) &= \{ (V, V^\perp)\in \mathcal F(2(q-r), 2n-2(q-r);\mathbb C^{2n}) : S_n(V,V)=0\}\\
    \mathcal D_{r,\frac12}(X) &= \{ (V, V^\perp)\in \mathcal F(2(q-r)+1, 2n-2(q-r)-1;\mathbb C^{2n}) : S_n(V,V)=0\} 
    \end{aligned}
\end{equation}
}
For $ (V, V^\perp) \in \mathcal D_r(X)$ {or $\mathcal D_{r,\frac12}(X)$} we will denote the corresponding 
subgrassmannian by $X_{V}$.
\item \label{Lagrangian}
Let $X$ be the Lagrangian Grassmannian $LGr_n$ consisting of $n$-planes passing through the origin in $\mathbb C^{2n}$ {which are isotropic
with respect to the nondegenerate antisymmetric bilinear form} $J_n =\left( \begin{array}{cc}
0 & I_n \\
-I_n & 0
\end{array}\right)$ on $\mathbb C^{2n}$.
In this case {$G= Sp(n,\mathbb C)/\left\{\pm I_{2n}\right\}$} and it acts on $LGr_n$ by \eqref{action}.
Consider a subgrassmannian in $LGr_n$ which is the set of all elements $x\in LGr_n$ such that 
\begin{equation}\label{subgra3}
 V \subset x \subset V^\perp
\end{equation}
for a given isotropic complex vector subspace $V\subset \mathbb C^{2n}$ with $\dim V = n-r$, where $V^\perp$ denotes the {annihilator} with respect to $J_n$.
Let $\mathcal D_r(X)$ denote the moduli space of such subgrassmannians in $X=LGr_n$, i.e.,
for $r=1,\ldots, n-1$
\begin{equation}\label{css3}
\mathcal D_r(X) = \{ (V, V^\perp)\in \mathcal F(n-r, n+r;\mathbb C^{2n}) : J_n(V,V)=0\} 
\end{equation}
For $(V,V^\perp) \in \mathcal D_r(X)$ we will denote the corresponding 
subgrassmannian by $X_{V}$
\end{enumerate}
\medskip

Recall that for each boundary component ${B}\subset S_{r}$, there exists { a totally geodesic complex submanifold $\Sigma \subset \Omega$ passing through the origin $o \in \Omega$}, a polydisk $\Delta^{q-r}$, { and a totally geodesic holomorphic embedding $\epsilon: \Delta^{q-r} \times \Omega_0 \to \Omega$ such that $B = \epsilon(\{t\} \times \Omega_0)$ for some $t\in(\partial \Delta)^{q-r}$.}  For each point $z\in \Delta^{q-r}$,   $\epsilon(\{z\} \times \Omega_0) =: \Omega'\subset \Omega$ is a {characteristic subdomain} of $\Omega$.  In general each characteristic subdomain is a bounded symmetric domain { on a characteristic symmetric subspace $X'$ of $(X,g_c)$}. 

 { For the following description of the characteristic subdomains for each irreducible bounded symmetric domain of types I, II, and III, see \cite[Part III]{W72} for further information.}

%Next we will describe characteristic subspaces when $X=Gr(q,p)$, $OGr_n$, or $LGr_n$.

\begin{enumerate}
\item 
Characteristic subspaces of rank $r$ in $Gr(q,p)$ are
the subgrassmannians $X_{(V_1, V_2)}$ with $\dim V_1 = q-r$ and $\dim V_2 = p+r$ in \eqref{subgra1} and hence 
the moduli space of them is $\mathcal D_r(X)$ given by \eqref{css1}.

The bounded symmetric domain $D_{p,q}^I$ corresponding to $Gr(q,p)$ is the set of $q$-planes in $\mathbb C^{p+q}$ on which the nondegenerate Hermitian form $I_{p,q} = \left(\begin{array}{cc}
I_q& 0 \\
0& -I_p
\end{array}\right)$ is positive definite.
Write $M^{\mathbb C}(p,q)$ for the set of $p \times q$ matrices with coefficients in $\mathbb C$, and denote by $\{e_1, \ldots, e_{p+q}\}$ the standard basis of $\mathbb C^{p+q}$.  For $Z \in M^\mathbb C(p,q)$, denoting by $v_k$, $1 \le k \le q$, the $k$-th column vector of $Z$ as a vector in $\mathbb C^p = {\rm Span}_\mathbb C\{e_{1+q},\ldots,e_{p+q}\}$ we identify $Z$ with the $q$-plane in $\mathbb C^{p+q}$ spanned by $\{e_k + v_k: 1 \le k \le q\}$.
Then we have
\begin{equation}\label{typeI}
D_{p,q}^I= \left\{ Z\in M^{\mathbb C}(p,q): I_q -  Z^*Z >0 \right\}
\end{equation}
where $Z^*$ denotes the conjugate transpose of $Z$.
The characteristic subdomains of rank $r$ of $D_{p,q}^I$ are of the form $X_{(V_1, V_2)}\cap D_{p,q}^I$ with $(V_1, V_2)\in \mathcal D_r(X)$.

\item
Characteristic subspaces of $OGr_n$ of rank $r$ are the subgrassmannians of the form \eqref{subgra2}
with $\dim V= 2\left[ \frac{n}{2} \right]-2r$.
Hence the moduli space of {these subgrassmannians} is $\mathcal D_{r}(X)$.
%Note that 
%$\dim V_1^\perp = 2n-2\left[ \frac{n}{2} \right] +2r$.

The bounded symmetric domain corresponding to $OGr_n$ is the set of $n$-planes in $X$ on {which} $I_{n,n}$ is positive definite.
It is given by
\begin{equation}\nonumber
D_{n}^{II}= \left\{ Z\in M^{\mathbb C}(n,n) : I_n -  Z^*Z >0,\,
Z=-Z^t \right\}.
\end{equation}
The characteristic subdomains of $D_{n}^{II}$ are of the form $X_V\cap D_{n}^{II}$ with $(V,V^\perp)\in \mathcal D_{r}(X)$.

\item
Characteristic {subspaces} of $LGr_n$ of rank $r$ {are} of the form \eqref{subgra3}
with $\dim W = n-r$.
Hence the moduli space of {of these subgrassmannians} is $\mathcal D_r(X)$.
%Note that $\dim W^\perp = n+r$.

The bounded symmetric domain corresponding to $LGr_n$ is the set of $n$-planes in $LGr_n$ on {which} $I_{n,n}$ is positive definite.
It is given by
\begin{equation}\nonumber
D_{n}^{III}= \left\{ Z\in M^{\mathbb C}(n,n) : I_n -  Z^*Z >0,\,
Z=Z^t \right\}.
\end{equation}
The characteristic subdomains of $D_{n}^{III}$ are of the form $X_V\cap D_{n}^{III}$ with $(V, V^\perp)\in \mathcal D_r(X)$.
\end{enumerate}
\medskip

Define
$$
	\mathcal{D}_r(\Omega):=\{\sigma\in \mathcal{D}_r(X)\colon  \Omega_\sigma := X_\sigma\cap\Omega\neq \emptyset\},
$$
where $X_\sigma$ is the subgrassmannian of $X$ corresponding to $\sigma\in \mathcal D_r(X)$.
We may consider 
$\mathcal D_r(\Omega)$ as the moduli space 
of the characteristic subdomains of rank $r$.
%when $\Omega$ is the type I or III (resp. the type II).
For each boundary orbit $S_k$ with $k\geq r$, define
\begin{equation}\label{Dr(Sr)}
	\mathcal{D}_r(S_k):=\{\sigma\in \mathcal{D}_r(X)
	\colon \Omega_\sigma:= X_\sigma\cap S_k \text{ is a nonempty open set in }X_\sigma\}.
\end{equation}
{Similarly, we define $\mathcal D_{r, \frac12}(\Omega)$ and $\mathcal D_{r,\frac12}(S_k)$ for the type II domains. Then $\mathcal{D}_r(\Omega)$, $\mathcal{D}_r(S_k)$ and $\mathcal D_{r, \frac12}(\Omega)$, $\mathcal D_{r,\frac12}(S_k)$ are $G_o$-orbits in $\mathcal{D}_r(X)$ and  $\mathcal{D}_{r, \frac12}(X)$ such that $\mathcal{D}_r(S_k)\subset\partial 		\mathcal{D}_r(\Omega)$ and $\mathcal{D}_{r,\frac12}(S_k)\subset\partial \mathcal{D}_{r,\frac12}(\Omega)$, respectively.} 
For notational consistency, we define
$$	~\mathcal{D}_0(X)=X,~\mathcal{D}_0(\Omega)=\Omega,
	~\mathcal{D}_0(S_k)=S_k.
$$
{Whenever necessary}, we will denote by $S_r(X)$ the boundary orbits of $\Omega\subset X$ for {a} specific $X$. 
\medskip

By Section 10 of \cite{W72}, we obtain the following lemma.

\bl\label{null}
Let $X=Gr(q,p)$. Then,
$\mathcal{D}_r(S_r)$ is parametrized by $(q-r)$-dimensional subspaces of $\mathbb C^{p+q}$ isotropic with respect to $I_{p,q}$. More precisely, any $\sigma\in \mathcal D_r(S_r)$ is of the form
$\sigma=(V_1, V_2)$, where $V_1$ is a $(q-r)$-dimensional isotropic {subspace} of $I_{p,q}$, $V_2$ is the {annihilator} of $V_1$ with respect to $I_{p,q}$ and vice versa. 
\el

Since one can embed $OGr_n$ and $LGr_n$ into $Gr(n,n)$ as totally geodesic {complex} submanifolds, by Lemma~\ref{null}, we conclude that
$\mathcal D_{r}(S_r)$ is parametrized by $(2[n/2]-2r)$-dimensional isotropic spaces {with respect to} $I_{n,n}$ for $X=OGr_n$ and 
$(n-r)$-dimensional isotropic spaces {with respect to} $I_{n,n}$ for $X=LGr_n$.

\subsection{Associated characteristic bundles}
We refer the reader to \cite{M89} as a general reference for this subsection.
For each $\sigma\in \mathcal{D}_r(\Omega)$, there exists a polydisk $\Delta^{q-r}$ such that $\Delta^{q-r}\times \Omega_\sigma$ is a totally geodesic submanifold of $\Omega$.
{Let $Gr(q-r,T\Omega)$ be the Grassmannian bundle defined by $\bigcup_{p\in \Omega}Gr(q-r, T_p\Omega)$.} Define {$\mathscr C^{q-r}(\Omega)\subset Gr(q-r,T\Omega)$} to be the set of tangent spaces of such  $\Delta^{q-r}$'s. 
Define the {\it r-th associated characteristic bundle} $\mathcal{NS}_r(X)\subset Gr(n_r, TX)$ (resp. $\mathcal{NS}_r(\Omega)\subset Gr(n_r,T\Omega)$) to be the collection of all the holomorphic tangent 
spaces to $X_\sigma$ with $\sigma \in \mathcal{D}_r(X)$ (resp. $X_\sigma$ with $\sigma \in \mathcal{D}_r(\Omega)$), 
which is a holomorphic fiber bundle over $X$,
where $n_r=\dim(X_\sigma)$ for $\sigma\in \mathcal{D}_r(X).$
By \cite{MT92}, we obtain $\mathcal{NS}_r(\Omega) = \mathcal{NS}_r(\Omega)\big|_0 \times \Omega$.
From {\cite[p.249ff.]{M89}}, $\mathcal {NS}_{q-1}(\Omega)|_0$ is a Hermitian symmetric space of the compact type.  More generally we have { the}
following statement.

Here in the proof, for clarity we denote by $[\,\cdots]$ the point in a classifying space corresponding to the object inside the square bracket.

\bl \label{NS}
$\mathcal{NS}_{r}(X)\big|_0$ is a Hermitian symmetric space of the compact type.
\el

{
\begin{proof}
\begin{enumerate}
\item {\bf $Gr(q,p)$ :}\
For a point $[V] \in X = Gr(q,p) = Gr(q,\mathbb C^{p+q})$ we have $T_{[V]}(Gr(q,p)) = V^*\otimes \mathbb C^{p+q}/V$.
Fix the base point $0 = [V_0]\in Gr(q,p)$ and identify $T_0X$ with $M^{\mathbb C}(p,q)$. Denote by $K^\mathbb C$ the image of $GL(q,\mathbb C) \times GL(p,\mathbb C)$ in {$GL(V_0^*\otimes \mathbb C^{p+q})/V_0$} where $(A,B) \in GL(q,\mathbb C) \times GL(p,\mathbb C)$ acts on $Z \in M^{\mathbb C}(p,q)$ by $(A,B)(Z) = BZA^{-1}$, which descends to the isotropy action of $K^\mathbb C$ on $T_0X$. By definition $\mathcal{NS}_r(\Omega)|_0 \subset Gr(n_r,T_0(\Omega))$, $n_r = r(p-q+r)$. The isotropy action of $K^\mathbb C$ on $T_0X$ induces a $K^\mathbb C$-action on $Gr(r(p-q+r),T_0(\Omega))$, and $K^\mathbb C$ acts transitively on $\mathcal{NS}_r(\Omega)|_0$.  When $\sigma \in \mathcal D_r(X)$ corresponds to $X_\sigma \subset  Gr(q,p)$ and $X_\sigma$ passes through 0, we have $[T_0(X_\sigma] := [E_{r}\otimes F_{p-q+r}] \in Gr(n_r,T_0(\Omega)$, where $E_{r}$ (resp. $F_{p-q+r}$) is a vector subspace in $V_0^* \cong \mathbb C^q$ (resp. {in} $\mathbb C^{p+q}/V_0 \cong \mathbb C^p$) of dimension $r$ (resp. $p-q+r$).  The action of $K^\mathbb C$ on $Gr(r(p-q+r),T_0(\Omega))$ descends from $(A,B)[E_{r}\otimes F_{p-q+r}] = [(AE_r)\otimes (BF_{p-q+r})]$.  As a $K^\mathbb C$-orbit, $\mathcal{NS}_r(\Omega)\big|_0 \cong Gr(r, q-r)\times Gr(p-q+r,q-r)$.   

\item {\bf $OGr_n$ :} \ Recall that $X = OGr_n$ consists of isotropic $n$-planes in $(\mathbb C^{2n}, S)$, $S$ being a nondegenerate symmetric bilinear form. For $[V] \in OGr_n \subset Gr(n,n)$ we have $T_{[V]}(Gr(n,n)) = V^*\otimes \mathbb C^{2n}/V$.  Under the isomorphism $\mathbb C^{2n}/V \cong V^*$ induced by $S$, { we have $T_{[V]}(Gr(n,n)) \cong V^*\otimes V^*$}, and $T_{[V]}(OGr_n) = \Lambda^2V^*$.  At the base point $0 = [V_0] \in OGr_n$ identify $T_0X$ with $\Lambda^2V_0^* \cong \Lambda^2(\mathbb C^n) \cong M_a^{\mathbb C}(n,n)$. Here, $M_a^{\mathbb C}(n,n)$ denotes the set of anti-symmetric $n\times n$ matrices with complex entries. Take $\mathbb C^n$ to consist of column vectors $w$, on which $GL(n,\mathbb C)$ acts by $A(w) = Aw$.  Let $K^\mathbb C$ be the image of $GL(n,\mathbb C)$ in $GL(\Lambda^2\mathbb C^n)$ by the action $A(Z) = AZA^t$ for $Z \in M_a^\mathbb C(n,n)$.  By definition $\mathcal{NS}_r(\Omega)|_0 \subset Gr(n_r,T_0(\Omega))$, $n_r := r(2r-1)$.  When $\sigma \in \mathcal D_r(X)$ corresponds to $X_\sigma \subset OGr_n$, and $X_\sigma$ passes through 0, we have $[T_0(X_\sigma)] := [\Lambda^2(E_{2r})] \in Gr(r(2r-1),T_0(\Omega))$, $E_{2r} \subset \mathbb C^n$ being a $(2r)$-plane. The action of $K^\mathbb C$ on $\mathcal{NS}_r(\Omega)|_0$ descends from $A(\Lambda^2(E_{2r})) = \Lambda^2(A(E_{2r}))$ for $A \in GL(n,\mathbb C)$ and $[E_{2r}] \in Gr(2r,n-2r)$. As a $K^\mathbb C$-orbit, $\mathcal{NS}_r(\Omega)|_0$ is the image of $Gr(2r,n-2r)$ in $Gr\left(r(2r-1),T_0(\Omega)\right)$ under the holomorphic embedding $\lambda: Gr(2r,n-2r) \to Gr\left(r(2r-1),\Lambda^2(\mathbb C^n)\right)$ defined by $\lambda([E_{2r}]) = [\Lambda^2(E_{2r})]$ for any $(2r)$-plane $E_{2r} \subset \mathbb C^n$.  

\item {\bf $LGr_n$ :} \
\ Recall that $X = LGr_n$ consists of isotropic $n$-planes in $(\mathbb C^{2n}, J_n)$, $J_n$ being a symplectic form. For $[V] \in LGr_n \subset Gr(n,n)$ we have $T_{[V]}(Gr(n,n)) = V^*\otimes \mathbb C^{2n}/V \cong V^*\otimes V^*$ induced by $J_n$, and $T_{[V]}(LGr_n) = S^2V^*$. At the base point $0 = [V_0] \in OGr_n$ identify $T_0X$ with $S^2V_0^* \cong S^2(\mathbb C^n) \cong { M_s^{\mathbb C}(n,n)}$. Here, $M_s^{\mathbb C}(n,n)$ denotes the set of symmetric $n\times n$ matrices with complex entries. Let $K^\mathbb C$ be the image of $GL(n,\mathbb C)$ in $GL(S^2\mathbb C^n)$ by the action $A(Z) = AZA^t$ for $Z \in M_s^\mathbb C(n,n)$.  By definition $\mathcal{NS}_r(\Omega)|_0 \subset Gr(n_r,T_0(\Omega))$, $n_r := \frac{r(r+1)}{2}$.  When $\sigma \in \mathcal D_r(X)$ corresponds to $X_\sigma \subset LGr_n$, and $X_\sigma$ passes through 0, we have $[T_0(X_\sigma)] := [S^2(E_{r})] \in Gr\left(\frac{r(r+1)}{2},T_0(\Omega)\right)$, $E_{r} \subset \mathbb C^n$ being an $r$-plane. The action of $K^\mathbb C$ on $\mathcal{NS}_r(\Omega)|_0$ descends from $A(S^2(E_r)) = S^2(A(E_r))$ for $A \in GL(n,\mathbb C)$ and $[E_r] \in Gr(r,n-r)$. As a $K^\mathbb C$-orbit, $\mathcal{NS}_r(\Omega)|_0$ is the image of $Gr(r,n-r) $ in $Gr\left(\frac{r(r+1)}{2},T_0(\Omega)\right)$ under the holomorphic embedding $\nu: Gr(r,n-r) \to Gr\left(\frac{r(r+1)}{2},S^2(\mathbb C^n)\right)$ defined by $\nu([E_r]) = [S^2(E_r)]$, $E_r \subset \mathbb C^n$ being an $r$-plane. 

%Denote by $M_s^{\mathbb C}(n,n)$ the set of $n\times n$ symmetric matrices.Identifying $T_0X$ with $M_s^{\mathbb C}(n,n)$, the holomorphic tangent space to $X_\sigma$ with $\sigma\in \mathcal{D}_r(X)$ is given by $ E_{r}\otimes E_{r}$, where $E_{r}$ is a vector subspaces in $\mathbb C^{n}$ of the dimension $r$. Since $\mathcal{NS}_r(\Omega)\big|_0$ is a $K^{\mathbb C} = U(n)^{\mathbb C}= GL(n,\mathbb C)$-orbit under the action $A(E_r\otimes E_r ) = (AE_{r})\otimes (A^{-1}E_{r} )$, one has $\mathcal{NS}_r(\Omega)\big|_0 = Gr(r, q-r)$.
\end{enumerate}
\end{proof}}

\section{Subgrassmannians in the moduli spaces}\label{Subgrassmannians in the moduli spaces}
%\subsection{Subgrassmannians in $\mathcal{D}_r(X)$}
%For $\tau\in \mathcal D_s(X)$, the corresponding subgrassmannians in $X$ will be denoted by $X_{\tau}$.

\bd\label{mathcal ZQ}
\begin{enumerate}
\item 
For $\tau \in \mathcal{D}_{s}(X)$ {or $\tau\in \mathcal D_{s,\frac12}$} with $s<r$, define
{
$$
\mathcal{Z}_{\tau}^{r}:=\{\sigma\in \mathcal{D}_r(X)\colon X_\tau\subset X_\sigma\}
$$ 
and
$$
\mathcal{Z}_{\tau}^{r,\frac12}:=\{\sigma\in \mathcal{D}_{r,\frac12}(X)\colon X_\tau\subset X_\sigma\}.
$$
}
\item
For $\mu\in \mathcal{D}_{s}(X)$ {or $\mu\in \mathcal D_{s,\frac12}(X)$} with $s>r$, define
$$
{ \mathcal{Q}_\mu^r}:=\{\sigma\in \mathcal{D}_r(X) \colon X_\sigma\subset X_{\mu}\}
$$
and {
$$
\mathcal{Q}_\mu^{r,\frac12}:=\{\sigma\in \mathcal{D}_{r,\frac12}(X) \colon X_\sigma\subset X_{\mu}\}.
$$}
\end{enumerate}
\ed
{From the definitions}, we obtain the following
for $X_\tau = X_{(V_1, V_2)}$ {or $X_\mu = X_{(V_1, V_2)}$}  
$$\mathcal{Z}_{\tau}^r
=\{(W_1, W_2)\in \mathcal D_r(X) : W_1\subset V_1,~ V_2\subset W_2  \}$$
and
%for $X_\mu = X_{(V_1, V_2)},$
$$\mathcal{Q}_{\mu}^r=\{(W_1, W_2)\in \mathcal D_r(X) : V_1\subset W_1,~ W_2\subset V_2  \}.$$
For a given $r$, we will omit { the} superscript $r$ if there is no confusion.
\medskip

%\bl\label{rank one HSS}
%For $\tau\in \mathcal{D}_{s}(X)$ with $s<r$,
%$\mathcal{Z}_{\tau}$ is the following:
%\begin{enumerate}
%\item $X=Gr(q,p)$ and $X_\tau = X_{(V_1, V_2)}:$ 
%$$\mathcal{Z}_{\tau}
%=\{(W_1, W_2)\in \mathcal D_r(X) : W_1\subset V_1, V_2\subset W_2  \}$$
%\item $X=OGr_n$ and $X_\tau = X_{V}:$ 
%$$\mathcal{Z}_{\tau}=\{(W, W^\perp)\in \mathcal D_r(X):W\subset V \}$$
%\item $X= LGr_n$ and $X_\tau = X_{V}:$
%$$\mathcal{Z}_{\tau}=\{(W, W^\perp)\in \mathcal D_r(X):W\subset V \}$$ 
%\end{enumerate}
%For $\mu\in \mathcal{D}_{s}(X)$ with $s>r$, $\mathcal{Q}_{\mu}$ is the following: 
%\begin{enumerate}
%\item $X= Gr(q,p)$ and $X_\mu = X_{(V_1, V_2)}:$ 
%$$\mathcal{Q}_{\mu}=\{(W_1, W_2)\in \mathcal D_r(X) : V_1\subset W_1, W_2\subset V_2  \}$$
%\item $X= OGr_n$
%and $X_\mu = X_{V}:$ 
%$$\mathcal{Q}_{\mu}=\{(W, W^\perp)\in \mathcal D_r(X):V\subset W \}$$
%\item $X= LGr_n$ and $X_\mu=X_V:$ 
%$$\mathcal{Q}_{\mu}= \{(W, W^\perp)\in \mathcal D_r(X):V\subset W \}.$$
%\end{enumerate}
%\el

%We have observed in Section \ref{moduli spaces} that $\mathcal{D}_r(X)$ parametrizes characteristic symmetric subspaces in $X$ of the same type if $X$ is of type I or III(resp. II).
%Let 
%$$
%	V_X = 
%	\begin{cases}
%	\mathbb C^{p+q} & \text{ if } X = Gr(q,p), \\
%	\mathbb C^{2n}  & \text{ if } X = OGr_n \text{ or } LGr_n,
%	\end{cases}.	
%$$
Let 
$
pr\colon\mathcal{F}(a,b;V_X)\to Gr(a, V_X)
$
be the projection defined by 
$$
	pr(V_1, V_2)=V_1,
$$
where $V_X=\mathbb C^{p+q}$, if $X=Gr(q,p)$ and $\mathbb C^{2n}$, if $X=OGr_n$ or $ LGr_n$. 
\begin{Def}\label{DZQ}
For a given $r$, define
$$
	D_r(X): = pr(\mathcal{D}_r(X)),\quad
 Z_{\tau}:=pr(\mathcal Z_{\tau}),\quad
	Q_{\mu}:=pr(\mathcal Q_{\mu}),
$$
{and
$$
	D_{r,\frac12}(X): = pr(\mathcal{D}_{r,\frac12}(X)),\quad
 Z_{\tau}^{\frac12}:=pr(\mathcal Z_{\tau}^{\frac12}),\quad
	Q_{\mu}^{\frac12}:=pr(\mathcal Q_{\mu}^{\frac12}).
$$}
\end{Def}

$D_r(X)$ is a submanifold of $Gr(a, V_X)$, where $a=q-r$ if $X$ is of type~I or III and $a=2(q-r)$ if $X$ is of type~II and $Z_{\tau}$, $Q_{\mu}$ are subgrassmannians of $D_r(X)$.

In the case $X = Gr(q,p)$, $Q_\mu$ is the image of the holomorphic embedding $\imath: Gr(1,V_2/V_1) \to Gr(a+1,V_2)$, $a := \dim(V_1)$, defined by setting, for any $1$-dimensional complex vector subspace $\ell \subset V_2/V_1$, $\imath(\ell) = W_{2,\ell}$ where $W_{2,\ell} \subset V_2$ is the unique $(a+1)$-dimensional complex vector subspace in $V_2$ such that $W_{2,\ell} \supset V_1$ and such that $W_{2,\ell}/V_1 = \ell$. The description of $Q_\mu$ for $X = OGr_n$ and $X = LGr_n$ are similar. 
More precisely, {for $r$ fixed and for $\tau \in \mathcal D_s(X), s < r$ and for $\mu \in \mathcal D_s(X), s > r$},  we have Table \ref{table 1}.
\begin{table}[h]
\caption{Subgrassmannians}
\label{table 1}
\begin{tabular}{c|c|c|c}
$X$& $D_r(X)$& $Z_{\tau}~(X_{\tau} = X_{(V_1, V_2)})$& $Q_{\mu} ~(X_{\mu} = X_{(V_1, V_2)})$  \\\hline
$Gr(q,p)$&$Gr(q-r, \mathbb{C}^{p+q}) $& $Gr(q-r, V_1)$&$\{V\in Gr(q-r,V_2)\colon V_1\subset V\}$ \\\hline
$OGr_n$&$OGr(2[ n/2]-2r, \mathbb{C}^{2n})$& $Gr(2[ n/2]-2r, V_1)$&$\{V\in OGr(2[ n/2]-2r, V_1^\perp)\colon V_1\subset V\}$ \\\hline
$LGr_n$&$SGr(n-r, \mathbb{C}^{2n})$& $Gr(n-r, V_1)$ &$\{V\in SGr(n-r, V_1^\perp)\colon V_1\subset V\}$
\end{tabular}
\end{table}\\
In particular, if $\tau\in \mathcal{D}_{r-1}(X)$ and $\mu\in \mathcal{D}_{r+1}(X)$, we have Table \ref{table 2}: 
\begin{table}[h]
\caption{{ Subgrassmannians} when the rank difference { $|s-r|$ equals $1$}}
\label{table 2}
\begin{tabular}{c|c|c}
$X$ & $Z_{\tau}~(X_{\tau} = X_{(V_1, V_2)}) $ & $Q_{\mu}~(X_{\mu} = X_{(V_1, V_2)})$ \\\hline
$Gr(q,p)$ &$Gr(q-r, V_1) { \cong} Gr(1, V_1^*)$  &$V_1\oplus Gr(1, V_2/V_1)$\\\hline
$OGr_n$ &$Gr(2[n/2]-2r, V_1){ \cong} Gr(2, V_1^*)$  &$V_1\oplus OGr(2, V_1^\perp/V_1)$ \\\hline
$LGr_n$  &$Gr({n}-r, V_1){ \cong} Gr(1, V_1^*) $&$V_1\oplus Gr(1, V_1^\perp/V_1)$
\end{tabular}
\end{table}

\noindent
Table 1 above gives in particular for comparison the pairs $(D_r(X),Z_\tau^r)$, where $\tau \in \mathcal D_s(X)$ and the pairs $(D_r(X),Q^r_\mu)$, where $\mu \in \mathcal D_s(X)$, and Table 2 gives the special cases where the gap $|s-r|$ is equal to 1.  In the case of type-II Grassmannians we need to consider in addition $D_{r,\frac{1}{2}}(X)$, $Z^{r,\frac{1}{2}}_\tau$ where $\tau \in \mathcal D_r(X)$, $Z^r_\tau$, where $\tau \in \mathcal D_{r-1,\frac{1}{2}}(X)$, and $Q_{\mu}^{r, \frac{1}{2}}$, where $\mu \in \mathcal D_r(X)$.  If we label $D_t(X)$ as being of level $t$, $D_{t,\frac{1}{2}}(X)$ as being of level $t+\frac{1}{2}$, $Z_\tau^r$ for $\tau \in D_t(X)$ as being of level $t$, $Z^r_\tau$, $\tau \in \mathcal D_{t,\frac{1}{2}}(X)$ as being of level $t + \frac{1}{2}$, and  $Q_{\mu}^{r,\frac{1}{2}}$, where $\mu \in \mathcal D_t(X)$, as being of level $t + \frac{1}{2}$, then we will need to consider for comparison the pairs $(D_{r,\frac{1}{2}}(X),Z^{r,\frac{1}{2}}_\tau)$, where $\tau \in D_r(X)$, the pairs $(D_r(X),{Z^r_\tau})$, where $\tau \in \mathcal D_{r-1,\frac{1}{2}}(X)$, and the pairs $({D_{r,\frac{1}{2}}(X)},Q^{r,\frac{1}{2}}_\tau)$, where $\tau \in \mathcal D_r(X))$.  These are pairs $(A,B)$, where the gap of the levels of $A$ and $B$ are equal to $\frac{1}{\phantom{.}2\phantom{.}}$.  For this purpose we have the data given by following Table 3, noting that for type-II Grassmannians we have $\mathcal D_{r,\frac{1}{2}}(X) = Gr(2[\frac{n}{2}]-2r-1, \mathbb C^{2n})$.  To be consistent with the other tables, we drop the reference to $r$ in the table. 

\begin{table}[h]
\caption{{Subgrassmannians when $X$ is of type II} and the gap is
%% $\frac{1}{\phantom{.}2\phantom{.}}$}
$\frac{1}{\,2\,}$}
\label{table 3}
\begin{tabular}{c|c|c|c}
$X$ 
& $Z_{\tau}^{\frac12}(X_{\tau} = X_{(V_1, V_2)},\tau\in D_r(X)) $ 
& $Z_\tau$ $(X_\tau = X_{(V_1, V_2)}, \tau\in D_{r-1, \frac12}(X))$
& $Q_{\mu}^{\frac12}~(X_{\mu} = X_{(V_1, V_2)})$ \\\hline
$OGr_n$ 
&$Gr(2[\frac{n}{2}]-2r-1, V_1){ \cong} Gr(1, V_1^*)$  
&$Gr(2[\frac{n}{2}]-2r, V_1)\cong Gr(1,V_1^*)$
&$V_1\oplus OGr(1, V_1^\perp/V_1)$ \\
\end{tabular}
\end{table}

%By definition, we obtain
%\bl\label{minimal disc}
%If $\tau\in \mathcal{D}_{r-1}(X)$ and $\mu\in \mathcal{D}_{r+1}(X)$, then  
%$Q_{\mu}\cap Z_{\tau}$
%is a minimal rational curve in $Z_{\tau}$. Conversely, every minimal rational curve in $Z_{\tau}$ is of this form.
%\el

Let $X=G/P$, where $G$ is one of the complex simple Lie groups  
$SL(q+p, \mathbb C){/ \mu_{p+q}I_{p+q}}$, ${SO(2n, \mathbb C) /\{\pm I_{2n}\}}$ or $Sp(n,\mathbb C){ /\{\pm I_{2n}\}}$ {according to} the type of $X$ and $P$ is a maximal parabolic subgroup of $G$. 
Then $\mathcal D_r(X)$ and $D_r(X)$ are biholomorphic to $G/P', G/P''$ with parabolic subgroups $P', P''$ of $G$ and 
{their} automorphism groups are exactly $G$ if $r\neq 0$ (see Section 3.3 in \cite{Akhiezer}).
In particular, $\mathcal D_r(X)$ and $D_r(X)$ are rational homogeneous {spaces}.

{We say that $D_r(X)$ is {\it connected by chains} of $Z_\tau$ with $\tau\in \mathcal D_{r-1}$ if, for any two points $A$, $B$ in $D_r(X)$, there exist $\tau_1,\ldots, \tau_k\in\mathcal D_{r-1}(X)$ for some $k$, 
such that $Z_{\tau_i}\cap Z_{\tau_{i+1}}\neq \emptyset$ for all $i=1,\ldots, k-1$ and $A\in Z_{\tau_1}$, $B\in Z_{\tau_k}$.
A similar definition can be applied to chains of $Q_\mu$ with $\mu\in \mathcal D_{r+1}(X)$ and chains of $Z_\tau^{\frac{1}{2}}$ with $\tau\in \mathcal D_{r-1,\frac{1}{2}}(X)$.}

\bl
${D}_r(X)$ is connected by chains of $Z_{\tau}$ with $\tau\in \mathcal{D}_{r-1}(X)$ {and} chains of $Q_{\mu}$ with $\mu\in \mathcal{D}_{r+1}(X)$.
If $X$ is of type II,
${D}_{r,\frac12}(X)$ is connected by chains of $Z_{\tau}^{\frac12}$ with $\tau\in \mathcal{D}_{r}(X)$ and $D_{r}(X)$ is connected by chains of $Z_\tau$ with $\tau\in \mathcal D_{r-1,\frac12}(X)$

\el

\begin{proof} 
We will prove the lemma when $X$ is of the type I. The same argument can be applied to other cases.
Let  $X=Gr(q,p)$.
Then $\mathcal D_r(X) = \mathcal F (q-r, p+r;\mathbb C^{p+q})$ and $D_r(X) = Gr(q-r, \mathbb C^{p+q})$ by Table~\ref{table 1}.
For two distinct points $x_0, x_1 \in {D_r(X)}$, choose a sequence $V_0,\ldots,V_m\in D_{r}(X)$ such that
$$ x_0= V_0,~ x_1= V_m, ~ \dim(V_{i-1}\cap V_{i})=q-r-1, ~i=1,\ldots m.$$
Define
$$ W_i=V_i+V_{i+1},\quad i=0,\ldots, m-1.$$
Then $x_0$ and $x_1$ are connected by the chain of {$Z_{\tau_i}=Gr(q-r, W_i)$, $1 \le i \le m$}, and by the chain of {$Q_{\mu_i}=(V_i\cap V_{i+1})\oplus Gr(1, W_i/(V_i\cap V_{i+1}))$, $0 \le i \le m-1$}.
\end{proof}

Define
$$
	\Sigma_r:=
	 	pr(\mathcal{D}_r(S_r)) .
%	 {  \text{ and } \quad
%	 \Sigma_{r,\frac12} := pr(\mathcal{D}_{r,\frac12}(S_r)).}
$$
By Lemma~\ref{null}, we obtain
$$\Sigma_r =D_r(X)\cap \{ V\in Gr(a, V_X) :   I_{p,q}|_{V} =0 \}$$
for some suitable $a$ and $I_{p,q}$.

\bl\label{one to one}
The closed submanifold $\Sigma_r \subset D_r(X)$ inherits from $D_r(X)$ the structure of a {Levi-}nondegenerate homogeneous CR manifold  {whose Levi form has eigenvalues of both signs} such that $pr: \mathcal D_r(S_r) \to \Sigma_r$ is a CR diffeomorphism.
\el

\begin{proof}
{In \cite[Section 2]{K21}, It was shown that $\Sigma_r$ has the structure of a Levi-nondegenerate CR manifold whose Levi form has eigenvalues of both signs.} 
We only need to show that $pr$ is one to one since it is smooth and regular.
Let $\sigma\in \mathcal D_r(S_r)$. Then $\sigma$ is expressed by the set of $q$-planes $x$ satisfying $$
{\rm Span}_\mathbb C\{e_{p+1} \wedge \cdots \wedge e_{p+q-r}\} \subset x \subset
{\rm Span}_\mathbb C\{e_{q-r} \wedge \cdots \wedge e_{p+q}\}.
$$
Therefore $\sigma$ is determined uniquely by the $I_{p,q}$-isotropic space $\mathbb C e_{p+1}\wedge \cdots \wedge e_{p+q-r}$ by Lemma~\ref{null}. 
\end{proof}

\bl\label{Z-tau-rank one}
Let $s<r$ and let $\tau\in \mathcal D_s(X)$. Then $\tau\in \mathcal{D}_{s}(S_{s})$ if and only if $Z_{\tau}\subset \Sigma_r$. 
\el

\begin{proof}
We only consider the case {where} $X=Gr(q,p)$. The same argument can be applied to $X=OGr_n$ or $X=LGr_n$.
Let $\tau\in \mathcal D_s(S_s)$. We may express $X_\tau$ as $X_{(W_1, W_2)}$ with $I_{p,q}$-isotropic $(q-s)$-dimensional subspace 
$W_1$ and $(p+s)$-dimensional subspace $W_2$. Then any element $X_{(V_1, V_2)}$ in $\mathcal Z_{\tau}$ satisfies
$V_1\subset W_1$. Hence we obtain $Z_{\tau} \subset \Sigma_r$. Conversely, $W_1$ is spanned by $\{V_1:V_1\subset W_1\}$ and if $W_1$ is not a null space of $I_{p,q}$, then there exists $V_1\subset W_1$ of dimension $(q-r)$ such that $I_{p,q}\big|_{V_1}\neq 0$, i.e., 
$pr({\sigma})\not\in \Sigma_r$ for $pr(\sigma)=V_1\in Z_{\tau}$. 
\end{proof}

\bl\label{Q-Sigma}
Let $X=Gr(q,p)$ or $LGr_n$.
If $\mu\in \mathcal{D}_{r+1}(S_{r+1})$, then $Q_{\mu}\cap \Sigma_r$ is a real hyperquadric in $Q_{\mu}$.
\el

\begin{proof}
If $X=Gr(q,p)$, we may express $X_\mu$ as $X_{(W_1, W_2)}$ with $I_{p,q}$-isotropic $(q-r-1)$-dimensional subspace 
$W_1$ and $(p+r+1)$-dimensional subspace $W_2$. Hence any element in $Q_{\mu}\cap \Sigma_r$ can be 
represented by a vector $w\in W_2 /W_1$ satisfying $I_{p,q}|_{W_1 \wedge w} = 0$.
We can apply the same argument to the case {where} $X=LGr_n$.
\end{proof}

Let $r$ be fixed. Since a maximal integral manifold of the CR bundle $T^{1,0}\Sigma_r$ is a maximal complex submanifold of $\Sigma_r$, by Section~3 in \cite{K21}, we obtain that
$Z_{\tau},~\tau\in \mathcal{D}_0(S_0)$, is a maximal complex {manifold} in $\Sigma_r$ and vice versa.

\bl\label{HSS in Sigma}
Let $X=Gr(q,p)$. Then $\Sigma_r$ is covered by {Grassmannians} of rank 
$\min(r, q-r)$.
\el
\bpf
Choose a point $x\in \Sigma_r$. Then there exists a $(q-r)$-dimensional $I_{p,q}$-isotropic vector space $V_x$ representing $x$. Choose a $q$-dimensional $I_{p,q}$-isotropic space $W_x$ that contains $V_x$.
Then $Gr(q-r, W_x)$ is a subgrassmannian of rank $\min(r, q-r)$ in $\Sigma_r$ passing through $x$. 
\epf

{ Let $X=Gr(q,p)$ so that $D_r(X)=Gr(q-r, \mathbb C^{p+q})$. In Harish-Chandra coordinates $\{(x;y;z)$; { $x,\in M^{\mathbb C}(r, q-r), ~y\in M^{\mathbb C}(q-r,q-r)$, $z\in M^{\mathbb C}(p-q+r, q-r)\}$ on a big Schubert cell of $Gr(q-r,\mathbb C^{p+q})$, 
$\Sigma_r$ is defined by
\begin{equation*}
	I_{q-r}+x^*x-y^*y-z^*z=0,
\end{equation*}
where $x^*=\bar x^t$ and so on. At $P=(0;I_{q-r};0)\in \Sigma_r$, the real tangent space $T^{\mathbb R}_P\Sigma_r$
is defined by
$$dy_i^{~j}+d \bar y_j^{~i}=0,\quad i,j=1,\ldots,q-r$$
and the complex tangent space
$T^{1,0}_P\Sigma_r$ is defined by
$$dy_i^{~j}=0,\quad i,j=1,\ldots,q-r.$$
Therefore the real dimension of $\Sigma_r$ is $2(q-r)(p+r)-(q-r)^2$ and the CR dimension of $\Sigma_r$ is $(q-r)(p+r)-(q-r)^2$. Furthermore, for {the} complex structure $J$ of $D_r(X)$, we obtain
$$J(T^{\mathbb R}_P\Sigma_r)=\{dy_i^{~j}-d\bar y_{j}^{~i}=0\},$$
and hence $\Sigma_r$ is a generic CR manifold in $ D_r(X)$.
A maximal complex manifold $M$ in $\Sigma_r$ passing through $P$ should satisfy the system
$$dx^*\wedge dx-dy^*\wedge dy-dz^*\wedge dz=0.$$
Therefore on $T_PM$, we obtain
$$ dy=0$$
and
$$dx^*\wedge dx-dz^*\wedge dz=0.$$ 
Hence maximal complex manifolds in $\Sigma_r$ passing through $P$ 
are locally equivalent to 
\beq\label{max com}
\left \{ ( x;I_{q-r};Ax):x\in M^{\mathbb C}_{r, q-r}\right\}
\eeq
for a $(p-q+r)\times r$ matrix $A$ such that $A^*A=I_r$.
\medskip

{The} contact form $\theta$ on a CR manifold $S$ is a matrix{-}valued $\mathbb C$-linear one{-}form on the complexified tangent bundle of $S$ such that 
$$ker(\theta)=T^{1,0}S+T^{0,1}S,$$
where $T^{1,0}S$ is the CR bundle and $T^{0,1}S=\overline {T^{1,0}S}$.
}}

\bl\label{bracket gen} 
The CR structure of $\Sigma_r$ is Levi-nondegenerate. Furthermore, the 
CR structure of $\Sigma_r$ is bracket generating in the sense that for any nonzero real tangent vector $v$, there exist two $(1,0)$ vectors $w_1$, $w_2$ such that 
$\theta\wedge d\theta(v, w_1, \overline w_2)\neq 0$, {where $\theta$ is a contact form {on} $\Sigma_r$.}
\el

\bpf
For the CR structure of $\Sigma_r$ when $X$ is of type~I, see \cite{K21}. In the proof, we only consider $X=LGr_n$. The same argument can be applied for $X=OGr_n$.
Let $X=LGr_n$ and hence $D_r(X)=SGr(n-r, \mathbb{C}^{2n})$. We regard $D_r(X)$ as a submanifold in $Gr(n-r, \mathbb C^{2n})$.
Since everything is purely local, we can choose {Harish-Chandra} coordinates $(x;y;z)$; { $x, z\in M^{\mathbb C}(r, n-r), ~y\in M^{\mathbb C}(n-r,n-r)$, on a big Schubert cell $\mathcal W \subset Gr(n-r,\mathbb C^{2n})$, where $\mathcal W$ is identified with $M^\mathbb C(n+r,n-r) =  M^\mathbb C(r,n-r) \oplus M^\mathbb C(n-r,n-r) \oplus M^\mathbb C(r,n-r)$;
 and $\mathcal W \cap SGr(n-r,\mathbb{C}^{2n})$ is defined by}
\beq\label{1}
	y-y^t+x^tz-z^tx=0
\eeq
{since an $(n-r)$-plane in $\mathcal W$ lies in $SGr(n-r, \mathbb{C}^{2n})$ if and only if it is isotropic with respect to the symplectic form $J_n$ on $\mathbb C^{2n}$, and $\mathcal W \cap \Sigma_r$ is defined by \eqref{1} and}
\beq\label{2}
	I_{n-r}+x^*x-y^*y-z^*z=0,
\eeq
where $x^*=\bar x^t$ and so on, {since $\mathcal W \cap \Sigma_r \subset \mathcal W \cap SGr(n-r,\mathbb C^{2n})$ and it consists precisely of $(n-r)$-planes therein isotropic with respect to the indefinite Hermitian bilinear form $I_{n,n}$ on $\mathbb C^{2n}$}. Fix $P = (0;I_{n-r};0)$.  Then,
$$
 T_P D_r(X)=\{dy-dy^t=0\}
$$
and
$$
T_P\Sigma_r=\{dy-dy^t=dy+dy^*=0\}.
$$ 
Therefore we obtain
\beq\label{generic cr}
T_PD_r(X)=T_P\Sigma_r+J(T_P\Sigma_r),
\eeq
where $J$ is the complex structure of $D_r(X)$. Since $\Sigma_r$ is homogeneous, \eqref{generic cr} holds for any $P\in \Sigma_r$, i.e., $\Sigma_r$ is a generic CR manifold in $D_r(X)$.

Now choose $\tau\in \mathcal{D}_0(S_0)$ such that $P\in Z_{\tau}$. By Lemma~\ref{Z-tau-rank one}, we obtain $Z_\tau\subset \Sigma_r$ and hence 
$$T_PZ_\tau\subset T^{1,0}_P\Sigma_r=\{dy=0\}.$$	
On the other hand, at $P=(0;I_{n-r};0)$, subgrassmannians of the form  
$\left\{(x;I_{n-r};Ax) : x\in M_{r,n-r}^{\mathbb{C}}\right \}$ or $\left\{(Az;I_{n-r};z) : z\in M_{ r, n-r}^{\mathbb{C}}\right \}$ with $r\times r$ symmetric matrices $A$ are contained in $\Sigma_r$, which implies 
$$
	{\rm Span}_\mathbb C\left\{ \bigcup_{\tau}T_PZ_\tau\right\}=\{dy=0\},
$$
where the union is taken over all $\tau\in \mathcal D_0(S_0)$ such that $Z_\tau\ni P$.

Let 
$$
	\theta  := x^*d x - y^*dy - z^*dz,
	\quad
	\text{and}
	\quad
	\tilde \theta:=dy+x^tdz-z^tdx.
$$
Then, $\theta$ is a skew-Hermitian contact form on $\{ I_{n-r}+x^*x-y^*y-z^*z=0\}$ and $\tilde\theta$ is a symmetric one form on $D_r(X)$ (by equation \eqref{1}). Moreover, since $J_n=0$ on $\tau$ and $P\in Z_\tau$ if and only if $P\subset \tau$ as subspaces of $V_X$, by differentiating 
$$J_n(v,w)=0,\quad v\subset P,~w\subset \tau,$$
we obtain
$$T_PZ_\tau\subset \{\tilde\theta=0\}$$
for all $Z_\tau,~\tau\in \mathcal D_0(\Sigma)$ with $P\in Z_\tau.$
Hence, by the same argument as above, we can show that $\theta$ and $\tilde\theta$
{together} define the CR structure on $\Sigma_r$.
Notice that at $P=(0;I_{n-r};0)$, 
$$
	\tilde\theta\wedge d\tilde\theta=dy\wedge(dx^t\wedge dz-dz^t\wedge dx)
$$
on $T_PD_r(X)$ and hence {the proof is completed}.
\epf

%%%%%%%%%%%%%%%%%%%%%%%%%%%%%%%%%%%%%%%%%%%%%%%%%%%%%%%%%%
\section{Rigidity of the pair $(SGr(q,\mathbb C^{2n}), Gr(q, \mathbb C^{2n})$)}\label{Rigidity of the pair}
%%%%%%%%%%%%%%%%%%%%%%%%%%%%%%%%%%%%%%%%%%%%%%%%%%%%%%%%%%
We consider the question of rigidity for mappings for the pair $(X,X')$, where $X$ is {the} symplectic Grassmannian  $SGr(q,\mathbb C^{2n})$, $2 \le q \le n$, {$X'$ is the Grassmannian $Gr(q,\mathbb C^{2n})$}, {and $X$ is} identified with its image inside $X'$ by a standard embedding in the obvious way. 

{ The framework for formulating the rigidity problem above is the geometric theory of uniruled projective manifolds $X$ based on the study of varieties of minimal rational tangents (cf. \cite{HM98}, \cite{HM99}, \cite{M08b}, \cite{M16}, \cite{MZ19}).    From Mori theory there exists on $X$ a non-constant parametrized rational curve $f_0: \mathbb P^1 \to X$ which is free (i.e., $f_0^*TX \ge 0$ in the sense that $f_0^*TX$ decomposes into a direct sum of holomorphic line bundles of degree $\ge 0$ on $\mathbb P^1$) such that deformations of the cycle $[f_0(\mathbb P^1)]$ cannot split into two irreducible components at a general point $x \in X$. The space consisting of $f_0$ and its deformations $f$ as free rational curves, modulo the natural action by ${\rm Aut}(\mathbb P^1) \cong \mathbb PSL(2,\mathbb C)$, given by $(f,\varphi) \mapsto f\circ\varphi$ for $\varphi \in {\rm Aut}(\mathbb P^1)$, defines a minimal rational component $\mathcal K$, and a member $[f] \in \mathcal K$ is called a minimal rational curve.  We specialize to the case where $X$ is of Picard number 1, in which case $X$ is necessary Fano. In what follows when we speak of minimal rational curves and minimal rational components we will make the more restrictive assumption that ${\rm deg}(f_0^*TX)$ is minimal among all free parametrized rational curves on $X$. 

There is a smallest subvariety $B \subsetneq X$ such that for $x \in X-B$ the space $\mathcal K_x \subset \mathcal K$ of minimal rational curves  passing through $x$ is compact. We call $B \subset X$ the bad set of $(X,\mathcal K)$. For a general point $x \in X$ by the variety of minimal rational tangents $\mathscr C_x(X)$ we mean the Zariski closure (equivalently topological closure) of the set of all tangents $[df(0)(T{\mathbb P^1})] \in \mathbb PT_x(X)$ of (parametrized) minimal rational curves belonging to $\mathcal K$ such that $f(0) = x$ and $f$ is immersed at 0.  By Kebekus~\cite{Ke02}, at a general point $x \in X$ every minimal rational curve belonging to $\mathcal K$ and passing through $x$ is immersed, at (each branch passing through) the point $x$, so that it is not necessary to take Zariski closure in the definition of $\mathscr C_x(X)$.}   

{ The rigidity results in this section and in Section~\ref{Rigidity of subgrassmannian respecting holomorphic maps} will be used to show the rigidity of the induced moduli map $f^\flat_r$ (or its analogue) in Section~\ref{Rigidity of the induced moduli map}.}  Here by rigidity of the pair $(SGr(q,\mathbb C^{2n}), Gr(q, \mathbb C^{2n}))$ we will mean a form of rigidity weaker than the notion of rigidity of an admissible pair $(X,X')$ as was defined in \cite{MZ19} but which is nonetheless sufficient for our purpose {(cf. Proposition~\ref{standard})}. In a nutshell the support $S \subset X'$ of the sub-VMRT structure we consider comes from a holomorphic embedding ${H: U \to X'}$ on some nonempty connected open subset $U \subset X$, which, owing to the specific way that {$H$} is defined starting with a proper holomorphic map $f:\Omega \to \Omega'$, can be proven by means of CR geometry to transform any connected open subset of a minimal rational curve into a minimal rational curve {(as is given in the proof of Lemma~\ref{one jet of H} for the case of $(X,X')$), from which it follows that $H$ admits a rational extension by the proof of \cite[Theorem 1.1]{HoM10} of non-equidimensional Cartan-Fubini extension (cf. proof of Proposition~\ref{immersion}). One may say that we are proving more precisely rigidity of the triple $(SGr(q,\mathbb C^{2n}), Gr(q, \mathbb C^{2n}); H)$.}

{The main result of this section is Proposition~\ref{standard} proving that for a VMRT-respecting holomorphic map $H: U \to X'$ defined on a nonempty connected open subset $U \subset X$ modeled on the pair {$(X,X')$} of rational homogeneous manifolds of Picard number 1, {i.e., $H^1(X,\mathcal O^*) \cong \mathbb Z, H^1(X',\mathcal O^*) \cong \mathbb Z$}, which is { known} to extend to a rational map {$H: X \dashrightarrow X'$} (where by abuse of notation we use the same symbol $H$ to denote both the originally defined map on $U$ and its rational extension to $X$), the extended map is actually a standard holomorphic embedding {$H: X \to Y$ of $X$ onto some complex submanifold $Y \subset X'$}, i.e., it is the obvious embedding $\imath: SGr(q,\mathbb C^{2n}) \to Gr(q,\mathbb C^{2n})$ up to automorphisms of both the domain and the target manifolds.

The problem for the case of the pair $(SGr(q,\mathbb C^{2n}),Gr(q,\mathbb C^{2n})), n \ge 2, SGr(n,\mathbb C^{2n}) = LGr_n$, the Lagrangian Grassmannian of rank $n$, has been settled in \cite{M19} in which it was proven that the admissible pair of compact Hermitian symmetric spaces $(LGr_n,Gr(n,n))$, which is of non-subdiagram type, is rigid in the sense of the geometric theory of sub-VMRT structures.  Here for the purpose of our application to Theorem 1.2, the map $H$ arises from a proper holomorphic map $f: D^{\rm III}_n \to D^I_{n,n}$, and we will be able to establish that $H$ extends to a holomorphic map from $SGr(q,\mathbb C^{2n})$ into $Gr(q,\mathbb C^{2n})$, and we deal in this section with the question whether $H: SGr(q,\mathbb C^{2n}) \to Gr(q,\mathbb C^{2n})$} is a standard embedding.  

For the purpose of showing that $H$ is a standard embedding, we generalize certain arguments in \cite{M19} for the pair $(LGr_n,Gr(n,n))$ to our situation.
%(the first half being that of the nondegeneracy of sub-VMRT structures modelled on the pair, which is straightward but unnecessary for us, and which we put aside). 
Here we will recall some basic notions from the theory of sub-VMRT structures in order to be able to apply the argument of {\it parallel transport} along minimal rational curves as in \cite{M19}.  As opposed to the Lagrangian Grassmannian, the problem for parallel transport on symplectic Grassmannian ${X} = SGr(q,\mathbb C^{2n})$ for $2 \le q < n$ exhibit new difficulties. 

{
The problem of rigidity of an admissible pair $(X,X')$ is first of all related to the Recognition Problem of $X$. To put things in perspective, let us recall the Recognition Problem for a rational homogeneous space $X = G/P$ of Picard number 1.  Let $\mathcal K$ be the unique minimal rational component on $X = G/P$.  The VMRTs $\mathscr C_x(X)$ at all points $x\in X$ are equivalent to each other in the following sense. Take $0 = eP \in X$ as a reference point.  Then, for every point $x \in X$, the inclusion $\mathscr C_x(X) \subset \mathbb PT_x(X)$ is projectively equivalent to the inclusion $\mathscr C_0(X) \subset \mathbb PT_0(X)$ in the sense that there exists a projective linear isomorphism $\Lambda: \mathbb PT_0(X) \to \mathbb PT_x(X)$ such that $\Lambda(\mathscr C_0(X)) = \mathscr C_x(X)$. We say that the Recognition Problem for $X$ is solved in the affirmative if and only if the following statement $(\dagger)$ holds true:
    $(\dagger)$  {\it Let $(Y,\mathcal H)$ be a Fano manifold of Picard number 1 equipped with a minimal rational component $\mathcal H$, and denote by $\mathscr C_y(Y)$ the VMRT of $(Y,\mathcal H)$ at a general point $y \in Y$.  Suppose for a general point $y \in Y$ the inclusion $\mathscr C_y(Y) \subset \mathbb PT_y(Y)$ is projectively equivalent to $\mathscr C_0(X) \subset \mathscr C_0(X).$ Then, $Y$ is biholomorphically equivalent to $X$.} 
We note that although the Recognition Problem is stated here for the case where $Y$ of Picard number 1, the known (partially) affirmative solutions (cf. Theorem~\ref{hwangli}) apply even without the Picard number 1 condition on $Y$ to give an open VMRT-respecting embedding into $X$ of some sufficiently small neighborhood $\mathcal U$ (in the complex topology) of a general minimal rational curve $\ell \subset Y$.  It turns out that, coupled with the extension theorem for { sub-VMRT} structures (from \cite[Main Theorem 2]{MZ19}) and the Thickening Lemma (Theorem~\ref{thickening} here), this is enough for our application to solve in the affirmative the rigidity problem of the pair $(SGr(q,\mathbb C^{2n}), Gr(q, \mathbb C^{2n})), n \ge 2$.

The parabolic subgroup $P \subset G$ is determined by the marking of a single node $\alpha$ of the Dynkin diagram $\mathfrak D(G)$ of $G$. When the node $\alpha$ is a long root (resp. short root), we will call $X = G/P$ a rational homogeneous manifold of Picard number 1 associated to a long root (resp. short root).  For instance, when $G$ is of $A$, $D$ or $E$ type, all simple roots are of the same length, { hence} $X = G/P$ is always associated to a long root.  We call the Recognition Problem for $X = G/P$ the long-root case (resp. short-root case) when $X = G/P$ is associated to a long root (resp. short root).  The long-root case of the Recognition Problem was solved in the affirmative in the cases where $X = G/P$ is Hermitian symmetric or contact homogeneous by Mok \cite{M08d} and by Hong-Hwang \cite{HH08} for the rest of the long-root cases.  
  
{ We return now to our situation of the pair $(SGr(q,\mathbb C^{2n}),Gr(q,\mathbb C^{2n}))$, $n >2$, where we need first of all to deal with the Recognition Problem for $X = SGr(n,\mathbb C^{2n})$, $n > 2$.  Here $X = G/P$ where $G$ is the automorphism group of $(\mathbb C^{2n},\sigma)$, where $\sigma$ is a (complex) symplectic form on $\mathbb C^{2n}$, in other words the complex Lie group ${\rm Sp}(n,\mathbb C)$ of symplectic transformations.  
The Dynkin diagram of its Lie algebra $\mathfrak{sp}(n,\mathbb C)$ is $C_n$, consisting of $n$ simple roots 
$\alpha_1,\alpha_2,\cdots,\alpha_{n-1},\alpha_n$, where $\alpha_1$ and $\alpha_n$ are long roots, and $\alpha_2,\cdots,\alpha_{n-1}$ are short roots. We have $SGr(q,\mathbb C^{2n}) = G/P$, where $\mathfrak{p} \subset \mathfrak{g}$ is the parabolic subalgebra corresponding to the $q$-th node $\alpha_q$, which is a short root since by definition $2 \le q < n$.}
  
First of all, $X$ is marked at a short root, and the Recognition Problem for $X$ is much harder than the long-root case.
Fortunately, the Recognition Problem has recently been settled by \cite{HL21}, which, together with the Thickening Lemma, allows us to analytically continue $H$ along certain minimal rational curves. (It should be { noted that}, as will be explained later, the Recognition Problem is not solved in the affirmative as stated above, but an additional invariant needs to be determined in order for us to assert that $Y$ is biholomorphically equivalent to $X$ in the notation of the third last paragraph.) }Secondly, the moduli space of minimal rational curves on {$X$} is no longer homogeneous, and for our purpose arguments by parallel transport along minimal rational curves can only be carried out for general minimal rational curves, but we show that it is nonetheless sufficient to prove that the extended { rational} map {$H: X \dashrightarrow X'$} has no indeterminacies and is in fact a holomorphic immersion.  

Local calculations in terms of Harish-Chandra coordinates to be deferred to Section~\ref{Rigidity of subgrassmannian respecting holomorphic maps} allow us to show that {$H: X \to X'$} can be dilated via $\mathbb C^*$-action to a standard embedding, and the homotopy and cohomological arguments (involving volume forms) as in \cite{M19} allows us to recover $H$ as the obvious embedding up to {automorphisms} of the domain and target manifolds.

We now consider the pair ${(X,X')} = (SGr(q,\mathbb C^{2n}), Gr(q,\mathbb C^{2n}))$, $2 \le q \le n$ from the perspective of the geometric theory of sub-VMRT structures.
The obvious inclusion map $\imath\colon X \hookrightarrow X'$ sends minimal rational curves onto minimal rational curves, and we have $\imath_*\colon H_2(X,\mathbb Z) \overset{\cong}\longrightarrow H_2(X',\mathbb Z)$. We identify $X'$ as a projective submanifold by means of the Pl\"ucker embedding $\nu\colon Gr(q, \mathbb C^{2n}) \hookrightarrow \mathbb P^N$, $N+1 = \dim_\mathbb C \bigwedge^q\big(\mathbb C^{2n}\big) = \frac{(2n)!}{q!\cdot (2n-q)!}$.  To relate to the theory of sub-VMRT structures as given in \cite{MZ19} and \cite{M19} we have first of all  

\begin{Lem}\label{admissible}
In the notation above $(X,X')$ is an admissible pair of rational homogeneous manifolds of Picard number $1$ in the sense of \cite{MZ19} which is of non-subdiagram type.
\end{Lem}
\begin{proof}
To prove that the pair $(X,X')$ is an admissible pair of rational homogeneous manifolds of Picard number 1 in the sense of \cite{MZ19}, it suffices to show that $X$ is a linear section of $X' \subset \mathbb P^N$.

Denote by {$J_n$} the underlying symplectic form on $\mathbb C^{2n}$.  
For $q \ge 2$ let $\lambda\colon \bigotimes^{q}\big(\mathbb C^{2n}\big) \to \bigotimes^{{q-2}}\big(\mathbb C^{2n}\big)$ be the linear map uniquely determined by $\lambda(u_1 \otimes\cdots\otimes u_q) = {J_n}(u_1,u_2)(u_3\otimes\cdots\otimes u_q)$, and denote by ${\mu}\colon \bigwedge^{q}\big(\mathbb C^{2n}\big) \to \bigotimes^{q-2}\big(\mathbb C^{2n}\big)$ its skew-symmetrization.  We have readily ${\mu}\colon  \bigwedge^{q}\big(\mathbb C^{2n}\big) \to \bigwedge^{q-2}\big(\mathbb C^{2n}\big)$, where $\bigwedge^0\mathbb C^{2n} := \mathbb C$.  Now, for $\Pi \in Gr(q,\mathbb C^{2n}) = X'$ spanned by $u_1,\cdots,u_q$, $u_{\chi(1)}\wedge\cdots \wedge u_{\chi(q-2)}$ are linearly independent as $\chi\colon \{1,\cdots,q-2\} \to \{1,\cdots,q\}$ ranges over all injective maps, hence ${\mu}(u_1\wedge\cdots\wedge u_q) = 0$ if and only if ${J_n}(u_{s(1)},u_{s(2)}) = 0$ for any permutation $s$ of $\{1,\cdots,q\}$.  Thus ${\mu}(u_1\wedge\cdots\wedge u_q) = 0$ if and only if $\Pi$ is isotropic in $(\mathbb C^{2n},{J_n})$.  In other words, $X \subset X'$ is the linear section defined by the vanishing of the vector-valued linear map {$\mu$} on $\bigwedge^q\big(\mathbb C^{2n}\big)$. 

Since any rational homogeneous manifold determined by a subdiagram of the marked Dynkin diagram for a Grassmannian must itself necessarily be a Grassmannian, the admissible pair $(X,X')$ is of non-subdiagram type.
\end{proof}

Note that in the case where $q=n$, $X$ is the Lagrangian Grassmannian $LGr_n$, and the rigidity phenomenon for substructures for the admissible pair $(X,X')$ has been demonstrated in \cite{M19}, which is stronger than the rigidity phenomenon for mappings for the same pair $(X,X')$.  Thus, in what follows our focus is in the case $2\leq q<n$, although in the statement of results for the purpose of uniformity we will include the case where $X$ is a Lagrangian Grassmannian as a special case.  We refer the reader to \cite{HM05} and \cite{HL21} for descriptions of the VMRT on a symplectic Grassmannian, and to \cite{MZ19} for basics concerning sub-VMRT structures. For simplicity, we will consider sub-VMRT structures $\varpi\colon \mathscr C(S) \to S$ on some locally closed complex submanifolds modeled on the admissible pair $(X,X')$ which are already known to extend to a projective subvariety ${Y} \subset X'$, since for the application to complete the proof of the Theorem~\ref{main} in the case of proper holomorphic maps from type~III to type~I domains we {will be} led to a VMRT-respecting map $h\colon U \overset{\cong}\longrightarrow S \subset X'$ which is known to extend to a rational map $H\colon X \dashrightarrow X'$ (cf. Proposition~\ref{immersion}).  

We summarize in what follows information about the VMRT of a symplectic Grassmannian taken from \cite{HM05} which is of relevance for our further discussion on the rational map $H$. With respect to the standard labeling of nodes in Dynkin diagrams as for instance found in \cite{Y93}, the symplectic Grassmannian $SGr(q,\mathbb C^{2n})$, $2\le q \le n$ (denoted as $S_{q,n}$ in \cite{HM05}) is of type $(\mathfrak{sp}_n,\alpha_q)$. Fix a Cartan subalgebra $\mathfrak h \subset \mathfrak{sp}_n$.  For $2 \le q < n$ the symplectic Grassmannian $X := SGr(q,\mathbb C^{2n})$ is a rational homogeneous space of Picard number 1 associated to a graded complex Lie algebra of depth 2, $\mathfrak{sp}_n =:\mathfrak g = \mathfrak g_{-2}\oplus \mathfrak g_{-1}\oplus \mathfrak g_{0} \oplus \mathfrak g_{1} \oplus \mathfrak g_{2}$, where for $k \neq 0$ the vector space $\mathfrak g_k$ is spanned by root spaces $\mathfrak g^\rho$ for roots $\rho$ with coefficient equal to $k$ in the simple root $\alpha_q$, and 
$\mathfrak g_o = \mathfrak h \oplus \mathfrak t$, where $\mathfrak t$ is spanned by root spaces $\mathfrak g^\rho$ for roots $\rho$ with vanishing coefficient in the simple root $\alpha_q$.  We have $[\mathfrak g_k,\mathfrak g_\ell] \subset \mathfrak g_{k+\ell}$, setting $\mathfrak g_p := 0$ whenever $p { \notin} \{-2, -1, 0, 1, 2\}$.  The parabolic subalgebra $\mathfrak p$ is given by $\mathfrak p = \mathfrak g_{-2}\oplus \mathfrak g_{-1}\oplus \mathfrak g_{0}$.  Writing $G ={Sp(n,\mathbb C)}$ and $P \subset G$ for the parabolic subgroup corresponding to the parabolic subalgebra $\mathfrak p \subset \mathfrak g$, we have $X = G/P$ and the identification $T_0(G/P) = \mathfrak g_1 \oplus \mathfrak g_2$.  The vector subspace $\mathfrak g_o \subset \mathfrak g$ is a reductive Lie algebra corresponding to a Levi factor $L := G_o \subset P$, { which} has a one-dimensional center $\mathfrak z$ and we have a direct sum decomposition of Lie algebras $\mathfrak g_o = \mathfrak z \oplus \mathfrak{sl}_{q}\oplus\mathfrak{sp}_{n-q}$ (the semisimple part corresponding to the Dynkin subdiagram obtained by removing $\alpha_q$).  $L$ acts irreducibly on $\mathfrak g_1$ and $\mathfrak g_2$.  The isotropy action of $P$ on $\mathfrak g_1$ defines the minimal $G$-invariant holomorphic distribution $D \subset TX$.  We have $D \cong U^*\otimes Q$, where $U$ is the universal rank-$q$ holomorphic vector bundle inherited from the Grassmannian $X' = Gr(q,\mathbb C^{2n}) \supset SGr(q,\mathbb C^{2n}) = X$, and $Q$ is a rank $2(n-q)$ holomorphic vector bundle. At $0 \in G/P$ the direct factor {up to isogeny} $SL(q,\mathbb C)$ of $L$ acts nontrivially on $U^*_0$ while the direct factor {up to isogeny} $Sp(n-q,\mathbb C)$ acts nontrivially on $Q_0$. The isotropy action of $P$ on $\mathfrak g_2$ defines a holomorphic vector bundle {{$R$} on $X$ which is isomorphic to $TX/D$. { We have} {$R \cong S^2U^*$}. 

A point $x \in SGr(q,\mathbb C^{2n})$ corresponds to {a} $q$-dimensional complex vector subspace $V$ in $(\mathbb C^{2n},{J_n})$.  Denoting by $V^{\perp} \subset \mathbb C^{2n}$ the {annihilator} of $V$ with respect to ${J_n}$, by hypothesis we have $V \subset V^{\perp}$. (We have $Q_0 = V^\perp/V$ equipped with a symplectic form induced from ${J_n}$.) A minimal rational curve $\Lambda$ on $X$ containing $x \in X$ is determined by the choice of complex vector subspaces $A, B \subset \mathbb C^{2n}$, $\dim_{\mathbb C}A = q-1$, $\dim_{\mathbb C}B = q+1$, such that $A \subset V \subset B$. We say that the minimal rational curve $\Lambda \subset X$ is special if and only if $B$ is isotropic in $(\mathbb C^{2n},{J_n})$, otherwise $\Lambda$ is referred to as a ``general minimal rational curve'' on $X$.   Then, the set of vectors tangent to special minimal rational curves on $X$ span a proper {holomorphic} distribution which is precisely $D \subsetneq TX$. For a special rational curve $\Lambda$ passing through $x \in X$, $T_x(\Lambda) =: \mathbb C\alpha$, we will refer to $[\alpha] \in \mathscr C_x(X)$ as a special rational tangent.  

The VMRT $\mathscr C_0(X) \subset \mathbb PT_0(X)$ can be described explicitly as follows.

\begin{Lem}\label{symplectic-vmrt}
The highest weight orbit $\mathscr S_0(X) = \mathbb PU^*_0\otimes \mathbb PQ_0 \hookrightarrow \mathbb P(U_0^*\otimes Q_0)$ of the $L$-representation in $\mathbb PT_0(D) \cong \mathbb P\mathfrak g_1$ is the variety of special rational tangents at $0 \in X$,  $\mathscr S_0(X) \subset \mathscr C_0(X)$, the VMRT at $0 \in X$. Writing $\mathcal W_0$ for the highest weight orbit of the $L$-representation in $\mathbb P\mathfrak g_2$, which is the image of $\mathbb PU_0^*$ in $\mathbb P\mathfrak g_2$ under the Veronese embedding, we have $\mathcal W_0 \subset \mathscr C_0(X)$.  Let $N \subset P$ be the nilpotent Lie subgroup corresponding to the nilpotent Lie subalgebra $\mathfrak n := \mathfrak g_{-2}\oplus\mathfrak g_{-1} \subset \mathfrak p$, then the orbit of {$[\lambda_0\odot\lambda_0]$}, $0 \neq \lambda_0 \in U_0^*$, under $N$ is given by $N[\lambda_0{\odot}\lambda_0] = \{[\lambda_0\otimes \mu + \lambda_0{\odot}\lambda_0] \in \mathbb P(\mathfrak g_1 \oplus\mathfrak g_2): \mu \in Q_0\} \subset \mathscr C_0(X)$.  Moreover the VMRT $\mathscr C_0(X)$ is precisely the union of $\mathscr S_0(X)$ and the $N$-orbits {$N[\lambda\odot\lambda]$} as $\lambda$ ranges over non-zero vectors in $U_0^*$.  As a consequence $\mathscr C_0(X)$ is the union of $\mathscr S_0(X)$, the unique closed $P$-orbit in $\mathscr C_0(X)$, and the unique open $P$-orbit $\mathcal O := \mathscr C_0(X)\!-\mathscr S_0(X)$.  Thus, $\mathscr C_0(X) = \{[\lambda{\otimes}\mu +  \lambda{\odot}\lambda]: 0 \neq \lambda \in U^*_0, \mu \in Q_0\}$.
\end{Lem} 
\begin{proof}
Since {$SL(q,\mathbb C)$} acts transitively on $\mathcal W_0$, and $N$ acts transitively on $N[\lambda{\otimes}\lambda]$ by definition, $P$ acts transitively on $\mathcal O = \mathscr C_0(X)\!-\!\mathscr S_0(X)$. Clearly $\mathcal O \subset \mathscr C_0(X)$ is the unique (Zariski) open $P$-orbit.  All other statements are implicitly in \cite[Chapter 2]{HM05}.
\end{proof}
From the explicit description of the VMRT $\mathscr C_0(X)$ on the symplectic Grassmannian $X$, by a straightforward determination of the projective second fundamental form of $\mathscr C_0(X) \subset \mathbb PT_0(X)$ as a projective submanifold we have readily the following characterization of $\mathscr S_0(X) \subset \mathscr C_0(X)$ and $\mathcal O \subset \mathscr C_0(X)$ in terms of projective geometry.

\begin{Lem}[Lemma 6.6 in \cite{HL21}] \label{special-tangents}
Denote by $\zeta\colon S^2T{\mathscr C_0(X)} \to N_{\mathscr C_0(X)|\mathbb PT_0(X)}$ the projective second fundamental form as a holomorphic bundle map.  Then $\zeta$ is surjective at $[\alpha] \in \mathscr C_0(X)$ if and only if $[\alpha] \in \mathcal O$.  
\end{Lem} 

In Proposition~\ref{immersion} we will prove that $H$ is a holomorphic immersion.  The proof will rely on 
%the non-equidimensional Cartan-Fubini extension results of \cite{HoM10}, 
the theory of geometric substructures of \cite{MZ19}, especially the Thickening Lemma, and the characterization results of symplectic Grassmannians of Hwang-Li~\cite{HL21}.  Here it should be noted that according to \cite{HL21}, strictly speaking a symplectic Grassmannian other than a Lagrangian Grassmannian cannot be recognized among projective manifolds of Picard number $1$ solely by the VMRT at a general point.  In its place it has been shown in \cite{HL21} that in these cases the symplectic Grassmannians are characterized by the VMRT at a general point together with the nondegeneracy of {the} Frobenius form associated to a proper distribution determined by the VMRT.  We observe that this condition is automatically satisfied in the problem at hand, when the geometric substructure arises from a germ of VMRT-respecting holomorphic map.  

Given a uniruled projective manifold $(M,\mathcal K_M)$ and a locally closed complex submanifold $S$ of $M$, for $x \in S$ we define $\mathscr C(S) := \mathscr C(M) \cap \mathbb PT(S)$, $\mathscr C_x(S) := \mathscr C_x(M) \cap \mathbb PT_x(S)$. Writing $\mu\colon T_x(M) - \{0\} \to \mathbb PT_x(M)$ for the canonical projection, for a subset $E \subset \mathbb PT_x(M)$ we write $\widetilde{E} := \mu^{-1}(E) \subset T_x(M) - \{0\}$ for the affinization of $E$.  Write $\varpi := \pi|_{\mathscr C(S)}\colon \mathscr C(S) \to S$. { The following definitions and Lemma are taken from \cite{MZ19}.}

\begin{Def}\label{subVMRT}
We say that $\varpi :=\pi|_{\mathscr{C}(S)}\colon \mathscr{C}(S) \to S$ is a sub-VMRT structure on $(M,\mathcal K_M)$ if and only if 

\begin{enumerate}
\item[(a)]
{\it the restriction of $\varpi$ to each irreducible component of $\mathscr{C}(S)$ is surjective,
and }
\item[(b)]
{\it at a general point $x \in S$ and for any irreducible component $\Gamma_x$ of $\,\mathscr C_x(S)$,
we have $\Gamma_x \not\subset\text{\rm Sing}(\, \mathscr C_x(M))$.} 
\end{enumerate}
\end{Def}

\begin{Def}\label{condition-t}
Let $(M,\mathcal K_M)$ be a uniruled projective manifold $M$ equipped with a minimal rational component $\mathcal K_M$. 
Let $\varpi\colon \mathscr C(S) \to S$, $\mathscr C(S): = \mathscr C(M) \cap \mathbb PT(S)$, be a sub-VMRT structure on a locally closed submanifold $S$ of $M$.  For a point $x \in S$, and $[\alpha] \in \text{\rm Reg}(\mathscr C_x(S)) \cap \text{\rm Reg}(\mathscr C_x(M))$, we say that $(\mathscr C_x(S),[\alpha])$, or equivalently $(\widetilde{\mathscr C_x}(S),\alpha)$, satisfies \text{\,\rm Condition (T)} (with respect to the sub-VMRT structure $\varpi\colon \mathscr C(S) \to S$ on $(M,\mathcal K_M)$) if and only if $\,T_\alpha(\widetilde{\mathscr C_x}(S)) = T_\alpha(\widetilde{\mathscr C_x}(M)) \cap T_x(S)$. 
%We say that $\varpi: \mathscr C(S) \to S$ satisfies \text{\,\rm Condition (T)} at $x$ if and only if $\,(\widetilde{\mathscr C_x}(S),[\alpha])$ satisfies \text{\,\rm Condition} \text{\rm (T)} for a general point $[\alpha]$ of each irreducible component of $\text{\rm Reg}(\mathscr C_x(S)) \cap \text{\rm Reg}(\mathscr C_x(X))$. We say that $\varpi: \mathscr C(S) \to S$ satisfies \text{\,\rm Condition (T)} if and only if it satisfies the condition at a general point $x \in S$.
\end{Def} 

Concerning Condition (T) we have the following lemma { on} linear sections ${Y}$ of a projective submanifold $M$ uniruled by projective lines which is a special case of \cite[Lemma 5.5]{MZ19} {in which} ${Y}$ is {further} assumed nonsingular (and uniruled by projective lines).  

\begin{Lem}\label{condition-t} 
Let $(M, \mathcal K_M)$, $M \subset \mathbb P^N$, be a uniruled projective manifold endowed with a minimal rational component consisting of projective lines, and denote by $\pi\colon \mathscr C(M) \to M$ the VMRT structure on ${M}$. Let ${Y} \subset M$ be a smooth linear section of ${M}$ and write $\mathscr C({Y}) = \mathscr C(M) \cap \mathbb PT({Y})$, the sub-VMRT structure on ${Y}$. Then, for a general point $z \in {Y}$ and a general smooth point $[\alpha] \in \mathscr C_z({Y})$, $(\mathscr C_z({Y}),[\alpha])$ satisfies $\text{\rm Condition (T)}$.
\end{Lem}

For the study of rational curves on a projective variety it is essential to find free rational curves lying on the smooth locus of the variety.  From the perspective of the theory of sub-VMRT structures the following result, which is a simplified version of the Thickening Lemma in \cite[Proposition 6.1]{MZ19}, gives a sufficient condition for finding an open neighborhood of some rational curve which is an immersed complex submanifold.

\begin{Thm}\label{thickening}
Let $(M,\mathcal K_M)$ be a uniruled projective manifold endowed with a minimal rational
component, $\dim_\mathbb C M =: n$, and $\varpi\colon \mathscr C(S) \to S$ be a sub-VMRT structure.
$\dim_\mathbb C S =: s$, and assume that there exists a projective subvariety ${Y} \subset M$ such that $\dim_\mathbb C {Y} = s$ and $S \subset {Y}$.  
Let $[\alpha] \in \mathscr C(S)$ be a { smooth} point of both $\mathscr C(S)$ and $\mathscr C(M)$ such that 
$\varpi\colon \mathscr C(S) \to S$ is a submersion at $\,[\alpha]$,
$\varpi([\alpha]) =: x$, $[\ell] \in \mathcal K_M$ be the minimal rational curve $($which 
is smooth at $x)$ such that $T_x(\ell) = \mathbb C\alpha$, and $\varphi\colon \text{\bf P}_\ell \to \ell$ be the  
normalization of $\ell$, $\bf P_\ell \cong \mathbb P^1$.  Suppose $(\mathscr C_x(S),[\alpha])$ satisfies $\text{\rm Condition (T)}$. Then, there exists an $s$-dimensional complex manifold 
{${\bf E}(\ell)$}, ${\bf P}_\ell  \subset {{\bf E}(\ell)}$, and a holomorphic immersion $\Phi\colon {{\bf E}(\ell)} \to M$ such that $\Phi|_{\bf P_\ell} \equiv \varphi$ and such that $\Phi({{\bf E}(\ell)})$ contains a neighborhood of $\,x$ on $S$.
\end{Thm}
 
{Crucial to our arguments is the following solution \cite{HL21} of Hwang-Li giving a solution to the Recognition Problem for the symplectic Grassmannian.} 

\begin{Thm}[\cite{HL21}]\label{hwangli}
Let $X$ be a symplectic Grassmannian $SGr(q,\mathbb C^{2n})$, $0< q \leq n$.  Let $Y$ be a uniruled projective variety containing a smooth standard rational curve $\ell_0 \subset {\rm Reg}(Y)$ in its smooth locus.  { Denote} by $\mathcal K^0_Y$ the normalized moduli space of $($unparametrized$)$ free rational curves $\ell \subset {\rm Reg}(Y)$ which are deformations of $\ell_0$ inside ${\rm Reg}(Y)$.  Denote by $\mathscr C^0_y(Y) \subset \mathbb PT_y(Y)$ the variety of $\mathcal K^0_Y$-rational tangents at a general point $y$ on ${\rm Reg}(Y)$ and denote by $\mathscr C_y(Y)$ the topological closure of $\mathscr C^0_y(Y)$ in $\mathbb PT_y(Y)$.   Assume that there exists a nonempty Euclidean open subset $O \subset Y$ such that for any $y \in O$, $\mathscr C_y(Y)  \subset \mathbb PT_y(Y)$ is projectively equivalent to $\mathscr C_0(X)  \subset \mathbb PT_0(X)$ for a $($ and hence any$)$ reference point $0 \in X$.  Then, given any member $[\ell] \in \mathcal K^0_Y$ such that $\ell$ is a standard rational curve, some Euclidean neighborhood of $\ell$ is biholomorphic to a Euclidean neighborhood of a general line in one of the presymplectic Grassmannians corresponding to $(\mathbb C^{2n},{\mu})$, where ${\mu}$ denotes a skew-symmetric complex bilinear form on $\mathbb C^{2n}.$
\end{Thm}

For the meaning of presymplectic Grassmannians and that of a general line on such a space we refer the reader to \cite{HL21}.  

By the hypothesis in Theorem~\ref{hwangli}, for any $y \in Y$, $\mathscr C_y(Y)  \subset \mathbb PT_y(Y)$ is projectively equivalent to $\mathscr C_0(X)  \subset \mathbb PT_0(X)$ for a reference point $0 \in X$.   Assuming $r \ge 1$ we have thus on $O$ a uniquely determined  $E \subsetneq T_O$ {corresponding to the subspace $\frak g_1$}.  We have the Frobenius form $\varphi\colon E {\otimes E} \to T_O/E$ defined as follows.  Let $y \in {O}$ and $v, w \in E_y$.  Shrinking the neighborhood ${U(y)}$ of $y$ if necessary let $\widetilde v, \widetilde w$ by $E$-valued holomorphic vector fields on ${U(y)}$ such that $\widetilde v(y) = v$ and $\widetilde w(y) = w$, then $\varphi(v,w) := [\widetilde v,\widetilde w](y)/E_y \in T_y(Y)/E_y$ is uniquely determined independent of the holomorphic extensions $\widetilde v, \widetilde w \in \Gamma({U(y)},E)$, and the Frobenius form $\varphi\colon E {\otimes} E \to T_O/E$ is defined at the arbitrary point $y \in O$ by $\varphi(v\otimes w) = \varphi(v,w)$ and extended to $E\otimes E$ by complex linearity.  Since the Lie bracket is skew-symmetric we may regard the Frobenius form as $\varphi\colon \bigwedge^2E \to T_O/E$.
 
\begin{Cor}\label{symplectic}
In the notation of the preceding paragraph and Theorem~\ref{hwangli}, assuming that the Frobenius form $\varphi\colon \bigwedge^2E \to T_O/E$ is nondegenerate in the sense that for any $y \in O$ and for any nonzero vector $v \in E_y$, there exists $w \in E_y$ such that $\varphi(v\wedge w) \neq 0$.  Then, in the concluding statement of Theorem~\ref{hwangli}, there exists some Euclidean neighborhood of $\ell$ in $Y$ which is biholomorphic to a Euclidean neighborhood of a general minimal rational curve on $X$.
\end{Cor}

In what follows we consider holomorphic embeddings defined on some nonempty connected open subset $U \subset X$.  Shrinking $U$ if necessary, we may assume that $\Lambda \cap S$ is either empty or a nonempty connected open set for any minimal rational curve $\Lambda$ on $X$. (For example, composing the minimal projective embedding of $X$ with a local affine linear projection in inhomogeneous coordinates, we may choose {an open subset $U \subset X$ which is identified by means of local holomorphic coordinates with a convex open subset $U' \subset \mathbb C^s$}, so that $\Lambda \cap U$ is an open subset of an affine line whenever $\Lambda \cap U \neq \emptyset$.)   

\begin{Pro}\label{immersion} 
Write $X := SGr(q,\mathbb C^{2n})$ and $X' := Gr(q,\mathbb C^{2n})$, $2 \le q \le n$.  Suppose there exists a nonempty connected open subset $U \subset X$ and a holomorphic embedding $H\colon U \to X'$ onto a locally closed complex submanifold $S \subset X'$ such that for any $x \in U$, writing $\mathscr C_{H(x)}(S) := \mathscr C_{H(x)}(X') \cap \mathbb PT_{H(x)}(S)$, the inclusion $\mathscr C_y(S) \subset \mathbb PT_y(S)$, for any $y \in S$, is projectively equivalent to $\mathscr C_0(X) \subset \mathbb PT_0(X)$ for a reference point $0 \in X$, and such that
{ the following statement $(*)$ holds true. $(*)$ For any minimal rational curve $\Lambda$ on $X$ such that $\Lambda \cap U \neq \emptyset$, $H(\Lambda \cap U)$ is an open subset of some projective line on $X'$.  Then, $H\colon U \overset{\cong}\longrightarrow S$ extends to a { rational map} $H\colon X \dashrightarrow X'$.  Furthermore,  $H\colon X {\to }{X'}$ is { in fact} a holomorphic immersion onto a projective subvariety ${Y} \subset X'$.}
\end{Pro} 

\noindent
\begin{proof} 
Since the case of $q = n$ has been established in \cite{M19}, in what follows we assume that $2\leq  q < n$ so that $X$ is a symplectic Grassmannian other than a Lagrangian Grassmannian. { In what follows we apply results of \cite{HoM10} to the case of $H\colon U \overset{\cong}\longrightarrow S \subset X'$, which is VMRT-respecting (i.e., $H_*(\mathscr C_x(X)) = \mathscr C_{H(x)}(X') \cap \mathbb P(dH(T_x(U))$ for $x \in U$).). 

It follows from the hypothesis $(*)$ that $H$ admits a rational extension, by the proof of \cite[Theorem 1.1]{HoM10} of non-equidimensional Cartan-Fubini extension.  More precisely, for a VMRT-respecting holomorphic embedding $\varphi: U \to X'$ in general, assuming that $\varphi$ satisfies some non-degeneracy condition (i.e., 2(b) in the Definition preceding \cite[ Proposition 2.1]{HoM10}) concerning the second fundamental form of VMRTs as projective subvarieties (noting that 2(a) concerning the bad locus $(X',\mathcal K')$ in the cited proposition is vacuous because $X'$ is a rational homogeneous manifold) consists of two steps.  First of all, as given in \cite[\it loc. cit.]{HoM10}, it is proven that the map $\varphi$ transforms a connected open subset of a minimal rational curve into a minimal rational curve as a consequence of the said non-degeneracy condition. Secondly, rational extendibility of $\varphi$ is deduced from Hartogs extension by means of parametrized analytic continuation along minimal rational curves as done in \cite[Proposition 4.3]{HM01}. Without the first step the arguments of the second step are still valid provided that we know {\it a priori} that $(*)$ holds true for $H: U \to X'$.  Here $(*)$ is taken as a hypothesis, hence $H: U \to X'$ extends rationally to $H: X \dashrightarrow X'$.  (It will be checked in Section~\ref{Rigidity of the induced moduli map} that $(*)$ is valid for $H$ being some moduli map $f_r^\flat$ (or its analogue) to be defined in Section~\ref{Induced moduli map}.})  For the proof of Proposition~\ref{immersion} it remains to establish the last statement that $H\colon X {\to }{X'}$ is in fact a holomorphic immersion, which we proceed now to do.  

There is a subvariety $A \subsetneq X$ such that the meromorphic map $H\colon X \dashrightarrow X'$ is holomorphic and of maximal rank on $X-A$.  Write ${Y} \subset X'$ for the Zariski closure of $H(X\!-\!A)$. We apply Theorem~\ref{thickening}, the Thickening Lemma adapted to our situation, to the meromorphic map $H\colon X \dashrightarrow {Y}$ in order {to find an open neighborhood of $\ell$ in $Y$ which is an immersed complex submanifold where $\ell$ is a certain projective line lying on $Y$.}  
%For the application of the Thickening Lemma we need to check Condition (T).  
By the hypothesis, for every point $s \in S$, and for $\mathscr C_s(S) := \mathscr C_s(X') \cap \mathbb PT_s(S)$, the inclusion $\mathscr C_s(S) \subset \mathbb PT_s(S)$ is projectively equivalent to the inclusion $\mathscr C_0(X) \subset \mathbb PT_0(X)$, $0 \in X$. For $x \in X\!-\!A$, writing $z:= H(x)$, $H$ maps some connected open neighborhood $U(x)$ of $x$ on $X\!-\!A$ onto a locally closed complex submanifold $S(z) \subset X'$.  Define $\mathscr C_z(S(z)) := \mathscr C_z(X') \cap \mathbb PT_z(S(z))$. 

Consider the subset $W \subset X\!-\!A$ such that the inclusion $\mathscr C_z(S) \subset \mathbb PT_z(S(z))$ is projectively equivalent to the inclusion $\mathscr C_0(X) \subset \mathbb PT_0(X)$. Then, $W$ contains the nonempty connected open subset $U \subset X\!-\!A$ (in the Euclidean topology).  We claim that $W$ contains a nonempty Zariski open subset $\mathcal W \subset X\!-\!A$.  To see this let $\chi\colon \mathscr P \to X'$ be the Grassmann bundle whose fiber over $w \in X'$ consists of $s$-planes $\Pi \subset T_w(X')$. Denote by $\mathscr S \subset \mathscr P$ the fiber subbundle whose fiber over $w \in X'$ consists of $s$-planes $\Pi$ such that the inclusion $\mathbb P\Pi\cap \mathscr C_w(X') \subset \mathbb P\Pi$ is projectively equivalent to the inclusion  $\mathscr C_0(X) \subset \mathbb PT_0(X)$.  $\mathscr S \subset \mathscr P$ is a constructible subset.  Hence, at every point $w \in X'$, the topological closure $\mathscr Q_w := \overline{\mathscr S_w} \subset \mathscr P_w$ is a Zariski closed subset of $\mathscr P_w$, and $\mathscr S_w$ contains a nonempty Zariski open subset. Since Zariski open subsets are closed under taking unions, there is a biggest (nonempty) Zariski open subset in $\mathscr S_w$, to be denoted by $\mathscr S^0_w \subset \mathscr S_w$.  Let $G'$ be the identity component of the automorphism group of $X'$.  $G' \cong \mathbb PGL(2n,\mathbb C)$ is a connected complex algebraic group.  For $w \in X'$ write $P'_w \subset G'$ for the parabolic subgroup which is the isotropic subgroup of $G'$ at $w$, so that $X \cong G'/P'_w$.  By the maximality of $\mathscr S^0_w \subset \mathscr S_w$ it follows that $\mathscr S^0_w$ is invariant under the isotropy action of $P'_w$, and it follows that by varying $w$ over $X'$ we have an algebraic fiber bundle $\mathscr S^0$ over $X'$ whose fiber at $w \in X'$ is given by $\mathscr S_w^0$.

By assumption, over the connected open subset $U \subset X$ the holomorphic map $h\colon U \overset{\cong}\longrightarrow S \subset X'$ induces a holomorphic map $\theta\colon U \to \mathscr S$, which is the composition $\zeta\circ  h$, here $\zeta$ is a holomorphic section of $\mathscr S^0$ over $S$. The meromorphic map $H\colon X \dashrightarrow {Y}$ induces a meromorphic map $\Theta\colon X \dashrightarrow \mathscr Q|_{Y}$ $(= \overline{\mathscr S}|_{Y}$) such that $\Theta$ is holomorphic on $U$ and $\Theta|_U \equiv \theta$.  Hence there exists some Zariski open subset $\mathcal W \subset X-A$ containing $U$ such that $H$ is holomorphic and of maximal rank on $\mathcal W$ and such that the induced holomorphic map ${\Theta}$ takes values in the Zariski open subset $\mathscr S^0|_{Y}$ of $\mathscr S|_{Y}$, as claimed.

Write $\mathcal W = X\!-\!\mathcal A$, $\mathcal W \subset W$, $\mathcal A \supset A$. Let now $x \in {\rm Reg}(\mathcal A)$.  Since the VMRT $\mathscr C_x(X) \subset \mathbb PT_x(X)$ is projectively nondegenerate (cf. \cite{HM05}),
there exists some $[\alpha] \in \mathscr C_x(X)$ such that $\alpha \notin T_x(\mathcal A)$. Since the condition imposed on $[\alpha]$ is an open condition on $\mathscr C_x(X)$ without loss of generality we may assume that $[\alpha]$ is tangent to a general minimal rational curve (in the sense of the paragraph immediately following Theorem~\ref{hwangli}). Let now $\Lambda$ be the (unique) minimal rational curve on $X$ passing through $x$ such that $T_x(\Lambda) = \mathbb C\alpha$.  For any point $y \in \Lambda \cap \mathcal W$, $H$ is a holomorphic immersion at $y$ and $\mathscr C_{w}(X') \cap \mathbb PT_{w}({Y}')$, $w := H(y)$, is projectively equivalent to $\mathscr C_0(X) \subset \mathbb PT_0(X)$, where we may take ${Y}' = {Y}$ if ${Y}$ is smooth at $H(y)$, and in general we take ${Y}'$ to be a nonsingular irreducible branch of ${Y} \cap V$ for some neighborhood $V$ of $H(y)$ on $X'$ such that $H$ (being a holomorphic immersion at $y$) is a biholomorphism of some neighborhood ${U(y)}$ of $y$ onto ${Y}'$. From the hypothesis that $H\colon U \overset{\cong}\longrightarrow S$ maps open subsets of minimal rational curves onto open subsets of minimal rational curves of $X'$ lying on $S \subset X'$, by analytic continuation it follows that over $\mathcal W \subset X\!-\!A$, the map $H$ is a holomorphic immersion and it maps any germ of minimal rational curve onto a germ of minimal rational curve.  Thus $H$ maps the germ of $\Lambda$ at $y$ to the germ of a (unique) minimal rational curve $\ell$ of $X'$ at $w$. 

By the choice of ${\Lambda}$, ${\Lambda} \cap \mathcal W$ is the complement in ${\Lambda}$ of a finite number of points. Let now $y \in \Lambda \cap \mathcal W$ (so that in particular $H$ is an immersion at ${y}$) and such that $H({y}) \in {\rm Reg}({Y})$.  
We will apply Theorem~\ref{thickening} (the Thickening Lemma) to the minimal rational curve $\ell \subset X'$ which lies on ${Y}$. For this purpose we have to check the validity of Condition (T) {on} the pair $(\mathscr C_w({Y'}),[T_w(\ell)])$ for the germ of sub-VMRT structure $\varpi\colon \mathscr C({Y}') \to {Y}'$ for a smooth neighborhood ${Y}'$ of ${H(y)}$ on ${Y}$, $\mathscr C({Y}'):= \mathscr C(X') \cap \mathbb PT({Y}')$.  Recall that, writing $T_w(\ell) = \mathbb C\beta$, by Definition~\ref{condition-t}, $(\mathscr C_w({Y'}),[\beta])$ satisfies Condition (T) for the sub-VMRT structure $\varpi\colon \mathscr C({Y}') \to {Y}'$ on ${Y'}$ if and only if  
$$
(\dagger) \quad
T_\beta(\widetilde{\mathscr C_w}({Y}')) = T_\beta(\widetilde{\mathscr C_w}(X')) \cap T_w({Y}').
$$ 

By hypothesis the inclusion $\mathscr C_w({Y}') \subset {\mathbb PT_w(Y')}$ is projectively equivalent to the inclusion $\mathscr C_0(X) \subset {\mathbb PT_0(X)}$, hence the statement ($\dagger)$ is equivalent to the statement 
$$
(\dagger\dagger)\quad T_\gamma(\widetilde{\mathscr C_0}(X)) = T_\gamma(\widetilde{\mathscr C_0}(X')) \cap T_0(X)
$$ for $\gamma \in \widetilde{\mathscr C_0}(X)$ being a vector tangent to a general minimal rational curve on $X'$ passing through $0$. Writing $G$ resp. $G'$ for the identity component of ${\rm Aut}(X)$ resp. ${\rm Aut}(X')$, and $P \subset G$ resp. $P' \subset G'$ for the isotropy (parabolic) subgroups at $0 \in X$ resp. $0 \in X'$, we have the standard inclusions $G \subset G'$ and $P = P'\cap G \subset P'$, $X = G/P \subset G'/P' = X'$, which defines the standard embedding $\imath\colon X \hookrightarrow X'$.  Now $P'$ acts transitively on the VMRT $\mathscr C_0(X')$ for the Grassmannian $X = Gr(q,\mathbb C^{2n})$ (which is an irreducible Hermitian symmetric space of the compact type), while by Lemma \ref{special-tangents} the VMRT $\mathscr C_0(X)$ of the symplectic Grassmannian $X = SGr(q,\mathbb C^{2n})$ is almost homogeneous under the action of $P$, with a unique open $P$-orbit $\mathcal O$ consisting of projectivizations of non-zero vectors $\gamma$ tangent to general minimal rational curves passing through 0, and a unique closed ${P}$-orbit $\mathcal F = \mathscr C_0(X)-\mathcal O$ consisting of projectivizations of those $\gamma$ tangent to special minimal rational curves passing through 0.

By Proposition~\ref{admissible} $X \subset, X' \subset \mathbb P^N$ is a linear section when the Grassmannian $X'$ is identified as a projective submanifold by the Pl\"ucker embedding.  By Proposition~\ref{condition-t}, at a general point $x \in X$ and a general point $[\xi] \in \mathscr C_x(X)$, $(\mathscr C_x(X),[{\xi}])$ satisfies Condition (T) (with respect to the sub-VMRT structure $\varpi\colon\mathscr C(X) \to X$ on $X'$).  In our case by homogeneity the conclusion holds actually at any point $x \in X$ (in place of requiring $x$ to be a general point).  Thus, we may take $x = 0$, and conclude that $(\mathscr C_0(X),[\xi])$ satisfies Condition (T) for a general point $[\xi] \in \mathscr C_0(X)$.  Since the statement that Condition (T) holds for $(\mathscr C_0(X),[\xi])$ is invariant under the action of $P$ it follows that ($\dagger\dagger$) must hold everywhere on the unique open $P$-orbit $\mathcal O \subset \mathscr C_0(X)$,
hence Condition (T) holds for $(\mathscr C_0(X),[{\xi}])$ whenever ${[\xi]} \in \mathcal O$.  As a consequence Condition (T) holds for  $(\mathscr C_w({Y}'),[\beta])$, $T_w(\ell) = \mathbb C\beta$ for the sub-VMRT structure $\varpi\colon \mathscr C({Y}') \to {Y}'$ on $Y'$.

On the Grassmannian ${X'}$ the minimal rational curve $\ell \subset {X'}$ is smooth, and the normalization $\varphi\colon {\bf P}_\ell \to \ell$ is just a biholomorphism.  It follows by Theorem~\ref{thickening} that there exists some complex manifold {${\bf E}(\ell)$} containing ${\bf P}_\ell$ and a biholomorphism $\Phi\colon {\bf E}_\ell \to { \Phi({\bf E}_\ell)} \subset X'$ such that $\Phi({\bf E}_\ell) =: { Y}_\ell \subset { Y}$. We compare now the two germs of complex manifolds along rational curves given by $(X;\Lambda)$ on $X$ and $({ Y}_\ell;\ell)$ on ${ Y}$.  From our choices there is a point $y \in \Lambda$ and an open neighborhood ${U(y)}$ of $y$ on $X$, such that $H$ is holomorphic on ${U(y)}$, $H$ maps ${U(y)}$ onto a neighborhood ${ Y}'$ of $w = H(y)$ on ${ Y}_\ell$ and $\Lambda \cap {U(y)}$ onto $\ell \cap { Y}'$ { such that $H$ is VMRT-respecting on ${U(y)}$ and such that, for $u \in U(y)$, $\mathscr C_u(X) \subset \mathbb P T_u (X)$ is projectively equivalent to $\mathscr C_v({ Y}_\ell) \subset \mathbb PT_v(X')$.} On ${ Y}'$ we have {by Lemma~\ref{symplectic-vmrt} a holomorphic} distribution $E$ which is spanned at every point $v = H(u)$ by the affinization of the subset $\mathcal F_u \subset \mathscr C_u({ Y}')$ consisting of points where the projective second fundamental form of $\mathscr C_u({ Y}') \subset \mathbb PT_u({ Y}')$ fails to be surjective.  Since the latter property in projective geometry is obviously preserved by $[dH]$, it follows that $\mathcal F_v = [dH](\mathbb PD_u \cap \mathscr C_u(U(y)))$ for every point $u \in U(y)$ where $D \subset TX$ is the minimal holomorphic distribution spanned by special rational tangents.  Since the Frobenius form $\varphi_D\colon\bigwedge^2 D \to TX/D$ associated to $D \subsetneq T(X)$ is nondegenerate (in the sense as described in Corollary \ref{symplectic}) everywhere on $X$, and we have $[dH(\xi), dH(\eta)] = dH([\xi,\eta])$ for holomorphic $D$-valued vector fields on ${U(y)}$, it follows that the Frobenius form $\varphi_E\colon \bigwedge^2 E \to T_{{ Y}'}/E$ associated to the holomorphic distribution $E \subsetneq T({ Y}')$ is also everywhere nondegenerate on ${ Y}' \subset { Y(\ell)}$.  It follows by Theorem~\ref{hwangli} that, shrinking ${ Y}_\ell$ if necessary, there exists some neighborhood $\mathcal U_0$ of $\Lambda \subset X$  and a biholomorphism $\Theta_0\colon { Y}_\ell \overset{\cong}\longrightarrow\mathcal U_0$ such that $\Theta_0|_\ell\colon \ell \overset{\cong}\longrightarrow \Lambda$, {and moreover by the statement of Theorem 7.12 in \cite{HL21} $\Theta_0$ preserves VMRTs.}

{\it A priori} ${\Theta_0}$ is unrelated to $H$.  However, using ${\Theta_0}$ we may now identify ${Y(\ell)}$ as an open subset of a copy $X_1$ of $X$, and consider $H|_{{U(y)}} \colon U(y) \overset{\cong}\longrightarrow {Y}'$ as a VMRT-preserving biholomorphism between the connected open subset ${U(y)} \subset X$ and ${Y}' \subset {Y(\ell)} \subset X_1$.  It follows by the Cartan-Fubini extension theorem of \cite{HM01} that $H|_{{U(y)}}$ extends to a biholomorphism $\Psi\colon X \overset{\cong}\longrightarrow X_1$. Thus, shrinking ${Y(\ell)}$ (as a complex manifold containing $\ell$) if necessary, there exists a neighborhood $\mathcal U$ of $\Lambda$ on $X$ and a biholomorphism $\Theta\colon \mathcal U \overset{\cong}\longrightarrow {Y(\ell)}$ such that $\Theta|_{{U(y)}} \cong H|_{{U(y)}} \colon{U(y)} \overset{\cong}\longrightarrow {Y}'$, $\Theta|_\Lambda\colon \Lambda \overset{\cong}\longrightarrow \ell$.  In particular, we have proven that $H\colon X \dashrightarrow { Y} \subset X'$ is holomorphic and in fact a local biholomorphism at $x \in {\rm Reg}(\mathcal A)$.  Since $x \in \mathcal A$ is arbitrary, we conclude that $H$ is a local biholomorphism at every point $x \in X\!-\!{\rm Sing}(\mathcal A)$.  Replacing now $\mathcal A$ by ${\rm Sing}(\mathcal A)$ and repeating the argument a finite number of times we conclude that actually $H$ is everywhere holomorphic and of maximal rank on $X$, and hence $H\colon X \to { Y} \subset X'$ is a holomorphic immersion onto the projective subvariety ${ Y} \subset X'$.  Since the only possible singularities of ${ Y}$ arise from intersection of locally closed complex submanifolds, denoting by $\nu\colon \widetilde { Y} \to { Y}$ the normalization of ${ Y}$, $\widetilde { Y}$ is a projective manifold, and $H\colon X \to { Y} \subset X'$ lifts to a holomorphic covering map $H^\sharp\colon X \to \widetilde {Y}$ such that $H = \nu\circ  H^\sharp$.  As $X$ is simply connected, we conclude that $H^\sharp\colon X \to \widetilde { Y}$ is a biholomorphism, hence $H\colon X \to { Y}\subset X'$ is a birational holomorphic immersion onto ${ Y}$, as asserted.  The proof of Proposition~\ref{immersion} is complete. 
\end{proof}

{
\begin{Rem}\label{remark_section4}
\begin{enumerate}
\item[(a)] The proof in \cite{M19} that for $n \ge 3$, $(LGr_n,Gr(n,\mathbb C^{2n}))$, is a rigid pair of admissible rational homogeneous manifolds of Picard number 1 in the sense of the geometric theory of sub-VMRT structures of Mok-Zhang \cite{MZ19} can be adapted to yield the same statement for $(SGr(q,\mathbb C^{2n}),Gr(q,\mathbb C^{2n}))$ for $n \ge 3$ and $2 \le q < n$, by checking the nondegeneracy condition for substructures as given in \cite[Definition 3.1]{MZ19}, which is a modification of the nondegeneracy condition for mappings given in \cite[Proposition 2.1]{HoM10}.  (There is a second step requiring the consideration of the restriction map of global holomorphic vector fields from $Gr(q,\mathbb C^{2n})$ to $SGr(q,\mathbb C^{2n})$, which will in any event be needed and checked in the proof of Proposition~\ref{standard}.) Here we have only proven the rigidity statement only for the triple $(SGr(q,\mathbb C^{2n}),Gr(q,\mathbb C^{2n});H)$.

\item[(b)]
Writing $X = SGr(q,\mathbb C^{2n}$), $X' = Gr(q,\mathbb C^{2n}$ as in the proposition, note that we have not proven that $H$ is everywhere VMRT-respecting in the sense explained in the first paragraph of the proof of the proposition.  The latter is not clear since the VMRT-respecting property is not {\it a priori} a closed property as we vary on $X$.  Nonetheless, the stronger statement that $H\colon X \to {X}'$ is everywhere VMRT-respecting is not needed for the proof of rigidity of $(X,X';H)$.  

\item[(c)]
It will be proven in Section~\ref{Induced moduli map} that from a proper holomorphic map $f:D^{III}_q \to D^{I}_{r,s}$ satisfying $2 \le q' < 2q-1$, where $q' = {\rm min}(r,s)$, one can derive a certain moduli map $H: U \to X'$ for some connected open subset $U \subset X$, $X = SGr(q,\mathbb C^{2n})$, $X' = Gr(q,\mathbb C^{2n})$ and prove in Section~\ref{Rigidity of the induced moduli map} that it respects subgrassmannians in the sense of Definition~\ref{Def_respect subgrassmannians}, so that in particular the hypothesis $(*)$ in Proposition~\ref{immersion} is satisfied for $H: U \to X'$. The proofs in Section~\ref{Rigidity of the induced moduli map} will rely on CR geometry. 
\end{enumerate}
\end{Rem}}
 
{ 
As will be proven in Lemma~\ref{Hs}, from} the VMRT-respecting mapping $h\colon U \overset{\cong}\longrightarrow S \subset X'$, {by using $\mathbb C^*$-action on $X'$ which preserves $X$, one can obtain a holomorphic one-parameter family of VMRT-respecting holomorphic embeddings $h_s\colon U \overset{\cong}\longrightarrow S_s \subset X'$, $s\in \mathbb C^*$, $H_0 = H$. Moreover, if $H:U\to S$ extends to a holomorphic immersion $H\colon X\to X'$, then $_s$ extends to a holomorphic immersion $H_s \colon X\to X'$ such that $H_s$ restricted to a big Schubert cell converges to the standard embedding uniformly on compact subsets as $s$ tends to $0$ ({cf.} Lemma~\ref{Hs} for details).}

Recall that the holomorphic immersion $H:X\to X'$ in Proposition~\ref{immersion} restricted to a general minimal rational curve in $X$ is a biholomorphism onto a projective line in $X'$ and therefore preserves the volume of projective lines with respect to the standard metric. Due to the construction, the same is true for $H_s$, $s\in \mathbb C^*$.

%and such that for any $s$ satisfying $0 < s \le 1$, $H_s = \Phi_s\cdot H_1$ for some $\Phi_s \in G'$
\begin{Pro}\label{deformation}
Let $H\colon X \to { Y}$ be a birational holomorphic immersion onto ${ Y} \subset X'$ such that $H_*\colon H_2(X,\mathbb Z) \overset{\cong}\longrightarrow H_2(X',\mathbb Z) \cong \mathbb Z$.  Then, there exists a {one-parameter family} of birational holomorphic immersions $H_s\colon X \to { Y}_s$ onto ${ Y}_s \subset X'$, { $s\in \mathbb C^*$} such that $H_1 = H$, and such that the reduced irreducible cycles $[{ Y}_s] \in {\rm Chow}(X')$ converge as cycles to $[{ Y}_0] \in  {\rm Chow}(X')$, ${ Y}_0 \subset X'$ { being} the image of a standard embedding $H_0\colon X \overset{\cong}\longrightarrow { Y}_0 \subset X'$.
\end{Pro}

\begin{proof}
Let $\omega$ resp. $\omega'$ be a K\"ahler form on $X$ resp. $X'$ such that minimal rational curves on $X$ resp. $X'$ are of area equal to 1. For $ s\in \mathbb C^*$, since $H_{s*}\colon H_2(X,\mathbb Z) \overset{\cong}\longrightarrow H_2(X',\mathbb Z) \cong \mathbb Z$, hence $H_s^*\colon H^2(X',\mathbb Z) \overset{\cong}\longrightarrow H^2(X,\mathbb Z) \cong \mathbb Z$, the K\"ahler forms $\omega$ and $H_s^*\omega'$ must be cohomologous, and we have
$$
{\rm Volume}({ Y}_s, \omega') = {\rm Volume}(X,\omega).
$$
On the other hand, for the standard embedding $\imath\colon X \hookrightarrow X'$ we also have $\imath^*\colon H^2(X',\mathbb Z) \overset{\cong}\longrightarrow H^2(X,\mathbb Z) \cong \mathbb Z$, so that we also have ${\rm Volume}({ Y}_0, \omega') = {\rm Volume}(X,\omega)$.  Now $H_s$ converges uniformly on compact subsets of a big Schubert cell $\mathscr S \subset { X}$ to the standard embedding $H_0\colon \mathscr S \to \mathscr S' \subset X'$, $\mathscr S' \subset X'$ being a big Schubert cell. Write $m: = \dim_\mathbb C X$.  It follows that as $m$-cycles, the reduced $m$-cycles $[{ Y}_s]$ must subconverge to the sum of the reduced $m$-cycle $[{ Y}_0]$ and some cycle $R$ with ${\rm Supp}(R) \subset X'\!-\!\mathscr S'$.  Finally, knowing that for {$s\in\mathbb C^*$}, ${\rm Volume}({Y}_s,\omega') = {\rm Volume}({ Y}_0,\omega') = {\rm Volume}(X,\omega)$ it follows that 
${\rm Volume}(R,\omega') = 0$, { which implies} that $R = \emptyset$, hence $[{Y}_s]$ converges to $[{Y}_0]$ as reduced cycles, as asserted.      
\end{proof}

{We remark that since $Y_0$ in Proposition~\ref{deformation} is a smooth variety, by the same argument in \cite{M19} $H\colon X\to X'$ is a holomorphic embedding.
Define a family $\mathcal Y := \{ (s,y) : s\in { \mathbb C}, y\in  Y_s\}$ which is a complex analytic subvariety $\mathcal Y\subset { \mathbb C}\times X'$. Since all fibers of $\mathcal Y\to { \mathbb C}$ are equidimensional smooth and reduced subvarieties of { $\mathbb C\times X'$}, $\mathcal Y \to { \mathbb C}$ is a regular family of projective submanifolds.

}

\begin{Pro}\label{standard}
The birational holomorphic immersion $H\colon X \to { Y} \subset X'$ in Proposition~\ref{deformation} is actually a standard embedding $H\colon X \overset{\cong}\longrightarrow { Y} \subset X'$ onto a complex submanifold ${ Y} \subset X'$.  In other words, regarding $X \subset X'$ by means of the standard inclusion $\imath\colon X \hookrightarrow X'$ of the symplectic Grassmannian $X = SGr(n-r,\mathbb C^{2n})$ as a subset of the Grassmannian $X' = Gr(n-r,\mathbb C^{2n})$, there exists some $\Xi \in G' = {\rm Aut_0}(X')$ such that $\Xi|_X = H$, ${ Y} = \Xi(X)$.
\end{Pro}

\begin{proof}
By Proposition~\ref{deformation} and the remark above, there exists a one-parameter family of biholomorphism $H_s\colon X \to { Y}_s$ onto ${ Y}_s \subset X'$, {$s\in{ \mathbb C}$} such that $H_1 = H$ and $[{ Y}_s]$ converges to the reduced cycle $[{ Y}_0]$ of the image of a standard embedding $H_0$ of $X$ into $X'$.  We may take $H_0$ to be $\imath\colon X \hookrightarrow X'$ so that ${ Y}_0 = X$.  We assert that $X \subset X'$ is infinitesimally rigid as a complex submanifold.

By Lemma~5.1 in \cite{M19}, it suffices to check that the restriction map $r\colon \Gamma(X',T{X'}) \to \Gamma(X,T{X'}|_X)$ is surjective.
{Moreover by the scheme of Section~6 of \cite{M19}, it is enough to show that $\Gamma(X,N_{X|X'})$ is an irreducible representation of $\text{Aut}(X)$.
Since $N_{X|X'}$ is a homogeneous vector bundle with the fiber $\Lambda^2 U^*$ which is an irreducible homogeneous {vector} bundle over $SGr(n-r, \mathbb C^{2n})$, by the Bott-Borel-Weil Theorem, $X\subset X'$ is infinitesimally rigid.}

{Since $X$ is infinitesimally rigid, there exists $\epsilon > 0$ such that for any $s \in \mathbb C$ satisfying $|s| < \epsilon$, ${Y}_s$ must be the image $\Xi_s(X)$ for some automorphism $\Xi_s \in G'$.  Fix a complex number $s_0$ such that $|s_0| < \epsilon$. Since $Y_{s_0} = \Phi_{s_0}(Y)$ for some $\Phi_{s_0} \in G'$, we conclude that $Y = \Phi_{s_0}^{-1}({ Y}_{s_0}) = \Phi_{s_0}^{-1}(\Xi_{s_0}(X)) = \Theta(X)$ for $\Theta := \Phi_{s_0}^{-1}\circ \Xi_{s_0} \in G'$, as desired.}
\end{proof}

%{ In what follows we will sometimes not make any notational distinction between a local holomorphic map and its rational extension to be established, hence we may write $H: U \to X'$ in place of $H: U \to X'$ in the context of Proposition 4.10.}

\section{Rigidity of subgrassmannian respecting holomorphic maps}\label{Rigidity of subgrassmannian respecting holomorphic maps} 

This section is devoted to prove {the} main technical result (Proposition~\ref{H respects}) that will be used to show the rigidity of induced moduli maps. From now on, we denote by $G$ and $G'$  the groups of automorphisms of $D_r(X)$ and $D_{r'}(X')$, respectively for $r, r'>0$.

{We restate the definition of subgrassmannian respecting holomorphic maps as given in Definition 1.4 in a local form. 

\begin{Def}\label{Def_respect subgrassmannians}
{Let $U \subset D_r(X)$ be non-empty connected open subset.}  A holomorphic map $H\colon U\to D_{r'}(X')$ is said to \emph{respect subgrassmannians} if and only if for any $Z_{\tau}\subset D_r(X)$ such that
{$U \cap Z_\tau \neq \emptyset$ and for each irreducible component $W_\tau^\alpha$ of $U \cap Z_\tau$, $\alpha \in A$, there exists $Z_{\tau'(\alpha)}\subset D_{r'}(X')$ such that
\begin{enumerate}
\item
 $H(W_\tau^\alpha) \subset Z_{\tau'(\alpha)}$ and 
\item $H|_{W_\tau^\alpha}$ extends to a standard embedding from $Z_{\tau}$ to  $Z_{\tau'(\alpha)}$. 
\end{enumerate}}
\end{Def}

\begin{Def}
A holomorphic map $H\colon Gr(a, W_1)\to Gr(b, W_2)$ is called a trivial embedding
if there exist a subspace $W_0\subset W_2$ of dimension $b-a$ and a linear embedding $\imath\colon W_1\to W_2$ such that $H(V)=W_0\oplus \imath(V)$. 
Let $N\subset Gr(a, W_1)$ be a complex submanifold {of some connected open subset $U \subset Gr(a,W_1)$.} A holomorphic map $H\colon N\to Gr(b, W_2)$ is called a trivial embedding if $H$ extends to $Gr(a, W_1)$ as a trivial embedding. 
\end{Def}

\bp\label{H respects}
Let $P\in \Sigma_r(X),~P'\in \Sigma_{r'}(X')$ and let $H\colon (D_r(X),P)\to (D_{r'}(X'),P')$ be a germ of a subgrassmannian respecting holomorphic map such that 
$$H(\Sigma_r(X))\subset \Sigma_{r'}(X')$$
and
$$
 H_*(T_P D_r(X))\not\subset T_{P'}\Sigma_{r'}(X').
$$
Suppose that the rank of $Z_{\tau}, ~\tau\in \mathcal D_0(X)$, is greater than or equal to $2$, then $H$ is a trivial embedding. 
\ep

The proof will be given in several steps. First, we will show that {the} $1$-jet of $H$ coincides with a trivial embedding and $H$ maps projective lines to projective lines. To be precise,  we will prove Lemma~\ref{one jet of H}. Note that if $X$ is of type~I or type~II, then for any projective line $L\subset D_r(X)$, there exists a subgrassmannian $Z_{\tau}$ such that $L\subset Z_{\tau}$. Since $H$ respects subgrassmannians, $H$ sends projective lines to projective lines. 
For {the} type~III case, we need the following lemma {which concerns real hyperquadrics with mixed Levi signature in Euclidean spaces and holomorphic maps which transform germs of complex lines on such real hyperquadrics to one another.   The lemma will lead to line-preserving rational maps between projective spaces. For a rational map $F: V \dashrightarrow W$ between two projective manifolds, writing $A \subset V$ for the set of indeterminacies (which is of codimension $\ge 2$), we will write ${\bf F}(V) := \overline{F(V-A)}$ for the strict transform of $V$ under $F$. We have} 

\bl\label{hyperquadric}
Let $\Sigma\subset \mathbb{C}^n$, $n\geq 3$, be a Levi nondegenerate real hyperquadric with mixed Levi signature passing through $0$ and let $H\colon(\mathbb{C}^n,0)\to (\mathbb{C}^N,0)$ be a germ of {immersive} holomorphic map which maps {connected} open pieces of complex lines in $\Sigma$ into complex lines. {Then, $H$ extends to a projective linear embedding $\widetilde H: \mathbb P^n \to \mathbb P^N$.}
\el

\bpf
{ We will prove the lemma in two steps. First we will show that $H$ maps any ({connected} open pieces of) complex lines in $\mathbb C^n$ into complex lines. Then{,} using this property we will show that $H$ extends to a projective linear embedding.}

For a point $P\in \Sigma$, let $\mathscr C_P(\Sigma)$ be the set of all complex lines in $\Sigma$ passing through $P$. We regard $\mathscr C_P(\Sigma)$ as a subset of the projectivised complex tangent space $\mathbb{P}T^{1,0}_P \Sigma$, {of complex dimension $= n-2 \ge 1$ since $n\ge 3$ by hypothesis}, by identifying a complex line $L\in \mathscr C_P(\Sigma)$ with $[T_PL]$. Since $\Sigma$ has mixed Levi signature, $\mathscr C_P(\Sigma)$ is a nondegenerate real hyperquadric in $\mathbb{P}T^{1,0}_P\Sigma$. Choose a representative of $H$ denoted again by $H$ and let ${{\rm Dom}(H)}$ be its domain {of definition}. Let $P\in \Sigma\cap {{\rm Dom}}(H)$.
By the assumption on $H$, for any $L\in \mathscr C_P( \Sigma)$, $H(L\cap {{\rm Dom}}(H))$ is contained in a complex line. Hence for any $k\geq 1$, 
	$${{\rm Span}_\mathbb C}\{j_P^k \left(H_{L}\right)\}:=
	{{\rm Span_\mathbb C}}\left\{\left (\frac{d ^{j} H^1_L}{d\zeta^j}(0),\cdots,\frac{d^{j}H^N_L}{d \zeta^j}(0)\right),~
	1\leq j\leq k\right\}
	$$ 
is of dimension $\leq 1$, where $H_L(\zeta):=H(P+\zeta v)$ for $0\neq v\in T_PL$ and $\zeta\in \mathbb C$. Since the space ${{\rm Span_\mathbb C}}\{j_P^k \left(H_{L}\right)\}$ depends meromorphically on $L \in \mathbb PT^{1,0}_P\Sigma$ and $\mathscr C_P(\Sigma)$ is a nondegenerate real hypersurface in $\mathbb PT^{1,0}_P\Sigma$, {for each integer $k \ge 1$} the dimension of ${{\rm Span_\mathbb C}}\{j_P^k \left(H_{L}\right)\}$ is less {than} or equal to $1$ for all $L\in \mathbb PT^{1,0}_P\Sigma$.
Hence for all $P\in \Sigma\cap { {\rm Dom}(H)}$ and for all $L\in \mathbb PT^{1,0}_P\Sigma$, $H$ maps $L$ into a complex line. 

Now let $T$ be a germ of a nonvanishing holomorphic vector field at $0\in \mathbb C^n$ such that ${{\bf Re}}(T)$ generates a one parameter family of CR {translations} on $\Sigma$ and let $\{\xi_\eps,~\eps\in \mathbb C\}$ be its flow for {a} sufficiently small {complex number} $\eps$. Let $P\in \Sigma$. Since $\xi_t$ for sufficiently small $t\in \mathbb R$ is a CR automorphism of $\Sigma$, for all $L\in \mathbb PT^{1,0}_P\Sigma$, $H$ maps $\xi_t(L)$ into a complex line. Since the map
$$ t\in \mathbb C\to {{\rm Span_\mathbb C}}\left\{j^k_{\xi_t(P)}\left(H_{\xi_t(L)}\right)\right\}$$
is meromorphic and $\mathbb R\subset\mathbb C$ is a maximal totally real submanifold, we obtain 
$$\dim  {{\rm Span_\mathbb C}}\left\{j^k_{\xi_t(P)}\left(H_{\xi_t(L)}\right)\right\}\leq 1,\quad \forall k\geq 1.$$
Therefore for sufficiently small $t\in \mathbb C$, $H$ maps $\xi_t(L)$ into a complex line. 

%$\Omega^{+}, \Omega^-$ be connected components of $\mathbb{C}^n-\Sigma$ and let $\mathcal M(\Omega^+)$, $\mathcal M(\Omega^-)$ be the sets of complex lines contained in $\Omega^+$, $ \Omega^-$, respectively. We may regard $\mathcal M(\Omega^\pm)$ as subsets of the set 
Let $\mathcal M(\mathbb P^n)$ be the set of all projective lines in $\mathbb P^n$. Then $\mathcal M(\mathbb P^n)$ is a finite dimensional complex manifold. 
{We claim that $\{ \xi_\eps(L): P\in \Sigma, ~L\in \mathscr C(T^{1,0}_P\Sigma), ~\eps\in \mathbb C\}$ is an open set in $\mathcal M(\mathbb P^n)$, where $\mathscr C(T_P^{1,0}\Sigma)$ is the set of all projective lines passing through $P$ and tangent to $\Sigma$ at $P$. Let $\mathscr C_P$ be the set of all projective lines in $\mathbb P^n$ passing through $P\in \mathbb P^n$. Then 
$$\mathscr C:=\bigcup_{P\in \mathbb P^n}\mathscr C_P$$
becomes a complex manifold with double fibration over $\mathcal M(\mathbb P^n)$ and $\mathbb P^n$. Let $\pi:\mathscr C\to \mathcal M(\mathbb P^n)$ be the natural projection. 
Since $T$ is transversal to $T^{1,0}_0\Sigma$, $\pi^{-1}(\{ \xi_\eps(L): P\in \Sigma, ~L\in \mathscr C(T^{1,0}_P\Sigma), ~\eps\in \mathbb C\})$ is a smooth fiber bundle over an open neighborhood of $0\in \mathbb C^n$ with respect to the natural projection to $\mathbb P^n$. Hence{,} to prove the claim 
it is enough to show that $\pi^{-1}\left(\{ \xi_\eps(L): P\in \Sigma, ~L\in \mathscr C(T^{1,0}_P\Sigma), ~\eps\in \mathbb C\}\right)\cap \mathscr C_0$ is open in $\mathscr C_0$. We may assume that on a neighborhood of $0\in \mathbb C^n$, $\Sigma$ is locally defined by
$$ Im~w=\sum_{j=1}^\ell |z_j|^2-\sum_{j=\ell+1}^{n-1} |z_j|^2=:\langle z, z\rangle_\ell$$
and 
$$T=\frac{\partial}{\partial w}$$
so that 
$$\xi_\eps(z,w)=(z,w+\eps).$$
As {in the} above, we identify $\mathscr C_0$ with $\mathbb P^{n-1}$. Choose a point $P=(z_0, \sqrt{-1} \langle z_0, z_0\rangle_\ell)\in \Sigma$ away from $0$.
Then $T_P^{1, 0}\Sigma$ is defined by
$$\partial w=2\sqrt{-1}\langle \partial z, z_0\rangle_\ell.$$
Choose a complex line $L$ given by
$$P+\zeta(z_0, 2\sqrt{-1}\langle z_0, z_0\rangle_\ell),\quad \zeta\in \mathbb C$$
passing through $P$ and tangent to $\Sigma$ at $P$. 
Then the parallel translation $L_1$ of $L$ given by
$$L+(0, \sqrt{-1}\langle z_0, z_0\rangle_\ell)$$
passes through $0\in \mathbb C^n$ and
$$T_0L_1=\mathbb C (z_0, 2\sqrt{-1}\langle z_0, z_0\rangle_\ell).$$
Now consider a one{-}parameter family of points $P_t:=(tz_0, \sqrt{-1}\langle tz_0, tz_0\rangle_\ell)\in \Sigma, ~t\in \mathbb C.$
Then by the same argument, the family $\{P_t\}$ generates a family of lines $\{L_t\}$ passing through $0$ such that
$$T_0 L_t=\mathbb C(z_0, 2\sqrt{-1}\langle z_0, t z_0\rangle_\ell).$$
Since $z_0$ and $t$ are arbitrary, the conclusion follows. Since $H$ maps ({connected} open pieces of) projective lines in $\{ \xi_\eps(L): P\in \Sigma, ~L\in \mathscr C(T^{1,0}_P\Sigma), ~\eps\in \mathbb C\}$ into projective lines, as a consequence, $H$ maps any ({connected} open piece of) complex line in $\mathbb P^n$ into {a} projective line. }

{
Next, we will show that $H$ extends to a projective linear embedding. Since $H$ is locally immersive at $0\in \mathbb C^n$, we may assume that 
$$H_*(T_0\mathbb C^n)=\{(x, 0) \in \mathbb C^n\times\mathbb C^{N-n}\}\subset T_0\mathbb C^N\equiv \mathbb C^N.$$
Since $H$ maps complex lines into complex lines, this implies
$$H(U)\subset\{(x, 0)\in \mathbb C^n\times\mathbb C^{N-n}\}.$$
Then we can apply Proposition 2.3.3 in \cite[(2.3)]{M99} to show that $H$ extends rationally to $\mathbb P^n$, {and the extended rational map will still be denoted by $H\colon \mathbb{P}^n\dashrightarrow \mathbb{P}^N$}.} {Denote by $E \subset \mathbb P^n$ the set of indeterminacies of $H$, and by $R^0 \subset \mathbb P^n-E$ the subvariety consisting of all points $y \in \mathbb P^n -E$ such that $\dim(dH(y)) < n$.  Then, $R:= \overline{R^0} \subset \mathbb P^n$ is a subvariety.  
Write $B:= R \cup E \subset \mathbb P^n$ and
pick $x_0 \in \mathbb P^n - B$. Let $Q \subset \mathbb P^N$ be the projective linear subspace such that $T_{H(x)}(Q) = dH(T_{x_0}(\mathbb P^n)) \cong \mathbb C^m$.  Since $H$ maps the germ $(\ell;x_0)$ of a projective line $\ell$ at $x$ to the germ $(\Lambda;H(x_0))$ of a projective line $\Lambda \subset \mathbb P^N$ at $H(x_0)$, $H(\mathbb P^n-B)$ is an open subset of $Q$ containing $H(x_0)$, {$Q = {\bf H}(\mathbb P^n)$.} For the proof of Lemma~\ref{hyperquadric}, we may take $Q = \mathbb P^n \subset \mathbb P^N$, $n = N \ge 3$.  (We note that the rest of the arguments work also for $n = N = 2$.)  }

For a line-preserving surjective rational map $H: \mathbb P^n \dashrightarrow \mathbb P^n$, $R^0 \subset \mathbb P^n-E$ is the ramification divisor of $H|_{\mathbb P^n-E}$. We call $R = \overline{R^0} \subset \mathbb P^n$ the ramification divisor of $H$. The rational map $H$ being the meromorphic extension of a line-preserving biholomorphism $H: U \overset{\cong}\longrightarrow V$ between certain connected open subsets $U, V \subset \mathbb P^n$, we can apply the same argument to $h^{-1}: U \overset{\cong}\longrightarrow V$ and conclude that $H: \mathbb P^n \dashrightarrow \mathbb P^n$ is birational.  Hence, for any rational curve $\ell$ such that $\ell \cap (X-B) \neq \emptyset$, the holomorphic map $H|_{\ell - B}$ extends to a biholomorphism from $\ell$ onto a projective line $\Lambda \subset \mathbb P^n$.
Hence, by \cite[Proposition 2.4.1]{M99} and its proof, $R = \emptyset$ and $H:\mathbb P^n \overset{\cong}\longrightarrow \mathbb P^n$ is a biholomorphism.  This completes the proof of Lemma~\ref{hyperquadric}.
\epf

\bl\label{one jet of H}
For $r>1$, let $H\colon U\subset D_r(X)\to D_{r'}(X')$ be a subgrassmannian respecting holomorphic map defined on a connected open set $U$ such that $U\cap\Sigma_r(X)\neq \emptyset$. If  
$$H(U\cap \Sigma_r(X))\subset \Sigma_{r'}(X')$$
and
\beq\label{cr_trans_3} 
 H(U)\not\subset \Sigma_{r'}(X'),
\eeq
then for each $P\in U$, there exists a trivial embedding $\widetilde H=\widetilde H_P\colon D_r(X)\to D_{r'}(X')$ 
such that 
	$$ H_*(T_P D_r(X))=\widetilde H_*(T_P D_{r}(X)).$$
Moreover, $H$ maps complex lines to complex lines. 
\el

%Fix a reference point $Q\in N$. For a submanifold $S\subset M$ and a point $P\in S$, we denote by 
%$T_PS\approx T_Q N$ if there exists $H\in \Aut(M)$ such that 
%$H(Q)=P$ and 
%$$ T_PS=T_P H(N).$$

\bpf
In the proof, we only consider the case when 
$X=LGr_n$ and $X'=Gr(q',p')$ so that $D_r(X)=SGr(n-r, \mathbb C^{2n})$ and $D_{r'}(X')=Gr(q'-r', \mathbb C^{p'+q'})$. The same argument can be applied to other cases. 

For a {Lagrangian} subspace $V_0$ {in $(\mathbb C^{2n},J_n)$}, choose a basis $\{ e_1, \ldots, e_{2n}\}$ of $\mathbb C^{2n}$ 
such that $\{e_1 + e_{n+1}, \ldots, e_n + e_{2n}\}$ is a basis of $V_0$ and $\tau\in \mathcal D_0(X)$ such that
$$Z_{\tau}=Gr(n-r, V_0)\subset D_r(X).$$
At a point  ${ {\rm Span}_{\mathbb C}\{e_1+e_{n+1}, \cdots, e_{n-r} + e_{2n-r}\}} \in Z_{\tau}$, we may take a local coordinate system 
of $Z_\tau$ such that 
$Z_{\tau}$ is locally given by $\{(x):x\in { M^{\mathbb C}(r, n-r)}\}.$
Since $H$ respects subgrassmannians, 
$H$ restricted to ${Z_{\tau}}$ is a standard embedding. 
Hence 
we may assume that
%$$ H(0;I_{n-r};0)= \left(0;I_{q'-r'};0\right)$$
%and 
\begin{equation}\label{H}
H\big|_{Z_\tau}(x) =W_0\oplus  \left(x\right) \subset W_0\oplus Gr(n-r, W_1)
\end{equation}
or
\begin{equation}\label{H2}
H\big|_{Z_\tau}(x) =W_0\oplus  \left(x^t\right) \subset W_0\oplus Gr(r, W_1)
\end{equation}
for some subspaces $W_0$ and $W_1$.

Suppose \eqref{H} holds. Choose $V\in X$ such that $\dim V_0\cap V=n-1>n-r$.
Let  
$$Z_{\rho}=Gr(n-r, V).$$
Without loss of generality, we may assume 
$$V_0\cap V = { {\rm Span}_\mathbb C} \{e_1+e_{n+1},\ldots, e_{n-1} + e_{2n-1}\}.$$ 
Since $Z_{\tau}\cap Z_{\rho} = Gr(n-r, V_0\cap V)$ and $H$ restricted to ${Z_{\rho}}$ is also a standard embedding, by \eqref{H} with $x = \left( \begin{array}{c}
x' \\
0 \end{array}\right)$, $x'\in { M^{\mathbb C}(r-1, n-r)}$, we obtain 
$$H(Z_{\rho})\subset W_0 \oplus Gr(n-r,W)$$
for some $W$ such that $W_1\cap W$ is of codimension one in $W_1$ and $W$.
Since $D_r(X)$ is connected by chain of $Z_\rho$'s with rank $Z_{\rho}\geq 2$, we obtain
$$H(D_r(X))\subset W_0\oplus Gr(n-r, W_0^\perp).$$
Let 
$$
	X'':=Gr(q'', W_0^\perp),
$$
where $q''=q'-\dim W_0$. Then we obtain
\beq\label{H3}
	H(D_r(X))\subset W_0\oplus D_{r''}(X''),
\eeq
where $r''$ satisfies
$$ 
	D_{r''}(X'')=Gr(n-r, W^\perp_0)\cong Gr(n-r, \mathbb C^m),~m=\dim W^\perp_0.
$$
We will replace $X'$ and $r'$ with $X''$ and $r''$, still using the same notation.

Suppose \eqref{H2} holds. Similarly, 
for each $Z_\tau$, there exist $U_\tau, V_\tau\subset W_1'$ with $\dim V_\tau=n$ such that
\beq\label{in}
H(Z_\tau)=U_\tau\oplus Gr(r, V_\tau)
\eeq
and there exists an $(n-r)$ dimensional vector space $L$ independent of $\tau$ such that any $Gr(r, V_\tau)$ contains a projective space of the form $Gr(1, L+e_\tau)$ for some vector $e_\tau$.  
Let 
$$U_0=\bigcap_\tau U_\tau.$$
Since $H(Z_\tau)\in \Sigma_{r'}(X')$ for all $Z_\tau{ \subset} \Sigma_r(X)$, $U_0\oplus L$ is $I_{p',q'}$-isotropic. Choose the minimal vector space $V_0$ that contains $\bigcup_{\tau}U_\tau\oplus V_\tau.$ Write
$$ V_0=U_0\oplus L\oplus V_1,$$
where $V_1$ is orthogonal to $U_0\oplus L$ with respect to $I_{p',q'}$.
Then
$$
	H(D_r(X))\subset U_0\oplus Gr(r'', L\oplus V_1) {\cong} Gr(n-r, \mathbb C^m), \quad r''=\dim V_1,$$
{where $\cong$ in a big Schubert cell is given by $(x)\to (x^t)$.}
On the other hand, since $H(P)\in \Sigma_{r'}(X')$ for $P\in \Sigma_r(X)$, $V_1$ should be $I_{p',q'}$-isotropic. 
Therefore $H(D_r(X))\subset \Sigma_{r'}(X')$, contradicting the assumption on $H$.

From now on we assume \eqref{H} and \eqref{H3} hold. Choose {local coordinates} $(x;y;z)$ of $Gr(n-r, \mathbb{C}^{2n})$ and $(X;Y;Z)$ of $Gr(n-r, \mathbb C^m)$ such that
$\Sigma_r(X)$ is defined by \eqref{1}, \eqref{2} and 
$\Sigma_{r'}(X')$ is defined locally by
\beq\label{locally}
	-I_{n-r}-X^*X+Y^*Y+Z^*Z=0,
\eeq
where $Y\in { M^{\mathbb C}(n-r, n-r)}$, $X\in { M^{\mathbb C}(a, n-r)}$, $Z\in { 
M^{\mathbb C}(b, n-r)}$  for some $a\leq b$.
For $i,j=1,\ldots,n-r$, define
$$
\theta_i^{~j}:=\sum_{k=1}^{n-r}\overline{ y_k^i} dy_k^{~j}-\sum_{\ell=1}^{r}\overline{x_\ell^{~i}}dx_\ell^{~j}
+\sum_{\ell=1}^r\overline{z_\ell^{~i}}dz_\ell^{~j}
$$
and
$$
\Theta_i^{~j}:=\sum_{k=1}^{n-r}\overline{ Y_k^i} dY_k^{~j}-\sum_{L=1}^{a}\overline{X_L^{~i}}dX_L^{~j}+\sum_{L=1}^{b}\overline{Z_L^{~i}}dZ_L^{~j}.
$$
Then by Section~\ref{Subgrassmannians in the moduli spaces}, $\theta$ and $\Theta$ are contact forms of $\Sigma_r(X)$ and $\Sigma_{r'}(X')$, respectively which define their CR structures. 
Since the CR bundles over $\Sigma_r(X)$ and $\Sigma_{r'}(X')$ are defined by 
$$\theta_i^{~j}=0,\quad \Theta_i^{~j}=0,\quad i,j=1,\ldots,n-r,$$
we may assume that for a fixed reference point $P_0=(0; I_{n-r};0)\in \Sigma_r(X)$, 
\beq\label{tangent cr} 
T^{1,0}_{P_0}\Sigma_r(X)={\left\{dy_i^{~j}=0 \right\}},\quad T^{1,0}_{H(P_0)}\Sigma_{r'}(X')= { \left\{dY_i^{~i}=0\quad i, j=1,\ldots,n-r\right\}}.\eeq

Since $H$ preserves the CR structure, we obtain
\beq\label{basic}
H^*(\Theta_i^{~j})=0\mod \theta.
\eeq
We will omit $H^*$ in \eqref{basic} and the following equations if there is no confusion.
Let
$$
	\Theta_1^{~1}=\sum_{j,k} u_j^{~k}\theta_k^{~j}.
$$
By differentiation, we obtain
\beq\label{embed-eds}
	\sum_k d\overline{Y_k^{~1}}\wedge dY_k^{~1}-\sum_L\left(d \overline{X_L^{~1}}\wedge dX_L^{~1}-d\overline{Z_L^{~1}}\wedge dZ_L^{~1}\right)
	=\sum_{j,k,\ell,m}u_j^{~k}\left(d\overline{y_m^{~k}}\wedge dy_m^{~j}-d\overline{x_\ell^{~k}}\wedge dx_\ell^{~j}+d\overline{z_\ell^{~k}}\wedge dz_\ell^{~j}\right)
\eeq
modulo $\theta$. Choose a maximal subgrassmannian $N\subset \Sigma_r(X)$ passing through $P_0\in \Sigma_r(X)$. 
{By {\eqref{max com},} we may assume that 
$$N=\left\{(x;I_{n-r};x):x\in { M^\mathbb C(r, n-r)}\right\}.$$ 
Since $H$ maps $\Sigma_r(X)$ into $\Sigma_{r'}(X')$, $H$ maps $N$ into a maximal complex manifold in $\Sigma_{r'}(X')$. Then by \eqref{max com} and \eqref{locally}, we may assume
$$N':=H(N)\subset \left\{\left(
X ;I_{n-r};\left(\begin{array}{c}
X \\
0 \end{array}\right)\right):X\in { M^{\mathbb C}(a, n-r)}\right\}.$$
Since $H$ respects subgrassmannians, by \eqref{H}, 
$$H(x;I_{n-r};x)=\left(\left(\begin{array}{c}
x \\
0 \end{array}\right);I_{n-r};\left(\begin{array}{c}
x \\
0 \end{array}\right)\right)$$
up to $\Aut(\Sigma_r(X))$ and $\Aut(\Sigma_{r'}(X'))$.}
Define
$$
	\psi_\ell^{~j}=dz_\ell^{~j}-dx_\ell^{~j},
	\quad \ell=1,\ldots,r
$$
and
$$
	\Psi_L^{~j}=dZ_L^{~j}-dX_L^{~j},\quad L=1,\ldots,a,
$$
$$
	\Psi_L^{~j}=dZ_L^{~j},
	\quad L=a+1,\ldots, b.
$$
Since $N$ and $N'$ are integral manifolds of $\psi=0$ and $\Psi=0$, respectively and $H\colon N\to N'$ is the identity map,
we obtain
\beq\label{psi-Psi}
	\Psi=0\mod\theta, \psi
\eeq
and for $j=1,\ldots, n-r$,
$$dX_\ell^{~j}=dx_\ell^{~j}\mod \theta, \psi,\quad \ell=1,\ldots, r,$$
$$ dX_L^{~j}=0\mod \theta, \psi,\quad L>r.$$
Then on $T_{P_0}^{1,0}\Sigma_r(X)$, \eqref{embed-eds} can be written as
\beq\label{bracket-gen}
	\sum_{\ell=1}^r \overline{\Psi_\ell^{~1}}\wedge dx_\ell^{~1}+ \overline{dx_\ell^{~1}}\wedge \Psi_\ell^{~1}
	=
	\sum_{j,k,\ell,m} u_j^{~k}\left( \overline{\psi_\ell^{~k}}\wedge dx_\ell^{~j}+ \overline{dx_\ell^{~k}}\wedge \psi_\ell^{~j}\right),
	\mod \overline{\psi}\wedge\psi.
\eeq
{ 
Since the CR structure of $\Sigma_r(X)$ is bracket generating, the right{-}hand side of \eqref{bracket-gen} contains 
$dx_\ell^{~j}$ for $j>1$ and $\overline{dx_\ell^{~k}}$ for $k>1$ unless $u_j^{~k}\neq 0.$ 
Therefore we obtain }
$$
\Theta_1^{~1}=u\theta_1^{~1},
$$
where $u=u_1^{~1}$ and together with \eqref{psi-Psi} and Cartan's lemma,
$$
\Psi_\ell^{~1}=u\psi_\ell^{~1} \mod\theta,\quad \ell=1,\ldots,r.
$$

Suppose $u\equiv 0$, i.e., 
$$\Phi_1^{~1}\equiv 0.$$ 
Since $j=1$ is an arbitrary choice, we may assume 
$$\Theta_j^{~j}\equiv 0,\quad j=1,\ldots, n-r.$$
Then we obtain
$$ \Psi_\ell^{~j}=0,\mod \theta,\quad \forall j,\ell$$
and by differentiating
$$
\Theta_j^{~i}=0\mod \theta$$
and substituting $\Psi_\ell^{~j}=0$ modulo $\theta,$
we obtain
$$ \Theta_j^{~i}\equiv 0,\quad i,j=1,\ldots,n-r.$$
%implying that 
%$$H_*(T_P \Sigma_r(X))\subset T^{1,0}_{H(P)}\Sigma_{r'}(X')+T^{0,1}_{H(P)}\Sigma_{r'}(X'),\quad \forall P\in \Sigma_{r}(X)\cap U.$$
In particular, $H(\Sigma_r(X)\cap U)$ is an integral manifold of $\Theta\equiv 0$. Hence there exists a maximal complex manifold $M\subset \Sigma_{r'}(X')$ that contains $H(\Sigma_r(X)\cap U)$.
Since $H$ is holomorphic, by Lemma~\ref{bracket gen}, we obtain
$$H_*(T_P D_r(X))=H_*(T_P\Sigma_r(X))+JH_*(T_P\Sigma_r(X))\subset T_{H(P)}M,\quad \forall P\in \Sigma_{r}(X)\cap U.$$
Hence we obtain
$$H(U)\subset M\subset \Sigma_{r'}(X'),$$
contradicting \eqref{cr_trans_3}. 
Therefore we obtain $u\not\equiv 0$ and after dilation (See {Appendix}), we may assume that $u\equiv 1$ on an open set.
Since $\theta$ and $\Theta$ are Hermitian symmetric and $H:N\to N'$ is the identity map, by continuing the process, we obtain
$$
\Theta_i^{~j}=\theta_i^{~j},\quad i,j=1,\ldots,n-r
$$
and
\beq\label{Psi^1}
	\Psi_\ell^{~j}=\psi_\ell^{~j}\mod \theta, \quad \ell=1,\ldots, r.
\eeq

Fix $j=1$. Then after rotation (See {Appendix}), we may assume that
\beq\label{equation-0}
dX_\ell^{~1}-dx_\ell^{~1}=dZ_\ell^{~1}-dz_\ell^{~1}=0\mod \theta, \quad \ell=1,\ldots, r,
\eeq
\beq\label{equation-2}
dX_L^{~1}=dZ_L^{~1}=0\mod \theta,\{dx_\ell^{~k}, dz_\ell^{~k}: k>1\}, \quad L>r. 
\eeq
Since $H$ respects subgrassmannians, by restricting $H$ to subgrassmannians of the form $\{(x;I_{n-r};Ux):x\in { M^{\mathbb C}(r,n-r)}\}$, where $U$ is an $r\times r$ symmetric matrix, \eqref{equation-0} implies that for all $j=1,\ldots,n-r$,
\beq\label{equation-1}
dX_\ell^{~j}-dx_\ell^{~j}=dZ_\ell^{~j}-dz_\ell^{~j}=0\mod \theta, \quad \ell=1,\ldots, r.
\eeq
Moreover, since $H$ sends all rank one vectors in subgrassmannians to rank one vectors, \eqref{equation-1} applied to  \eqref{equation-2} implies
$$ dX_L^{~1}=dZ_L^{~1}=0\mod \theta, \quad L>r. $$
Since $H$ respects subgrassmannian {distributions}, this implies that for all $j=1,\ldots,n-r,$
\beq\label{dx-dz} 
dX_L^{~j}=dZ_L^{~j}=0\mod \theta, \quad L>r .
\eeq
Since $dx_\ell^{~j}, dz_\ell^{~j}$ and $dX_L^{~j}, dZ_L^{~j}$ form coframes of $T^{1,0}_{P_0}\Sigma_r(X)$ and $T^{1,0}_{H(P_0)}\Sigma_{r'}(X')$, respectively, \eqref{equation-1} and \eqref{dx-dz} imply
$$
	 H_*(T^{1,0}_{P_0}\Sigma_r(X))=T^{1,0}_{H(P_0)}\widetilde \Sigma_r
$$
where 
$$\widetilde \Sigma_r:=\Sigma_{r'}(X')\cap \left\{\left(\left(\begin{array}{c}
x \\
0 \end{array}\right);y;\left(\begin{array}{c}
z \\
0 \end{array}\right)\right):x,z\in { M^{\mathbb C}(r, n-r)}, ~y\in { M^{\mathbb C}(n-r,n-r)}\right\}.$$

Since $n-r\geq 2$,  $\ell$ and $L$ are independent of the choice of $j=1,\ldots,n-r$, by the same argument of \cite{K21}, we obtain
$$
	dX_\ell^{~j}-dx_\ell^{~j}=dZ_\ell^{j}-dz_\ell^{~j}+\xi _\ell^{~k}\theta_k^{~j}=0, \quad \ell=1,\ldots, r,
$$
for some smooth functions $\xi_\ell^{~k}$ and 
$$ dX_L^{~j}=dZ_L^{~j}=0, \quad L>r .$$
After a frame change of the form \eqref{last-change} in Appendix, we obtain
$$ dX_\ell^{~j}-dx_\ell^{~j}=dZ_\ell^{~j}-dz_\ell^{~j}=dX_L^{~j}=dZ_L^{~j}=0.$$
In particular, together with \eqref{Psi^1},
$$T_{H(P_0)}H(\Sigma_r(X))=dH(T_{P_0} \Sigma_r(X))=T_{H(P_0)}\widetilde\Sigma_r.$$ 
More generally, we can choose smooth functions $g\colon \Sigma_r(X)\to G\cap \Aut(\Sigma_r(X))$, $g'\colon \Sigma_r(X)\to G'\cap \Aut(\Sigma_{r'}(X'))$ such that 
\beq\label{rank1}
dH\circ  g(P)=g'(P)\circ  Id,\quad \forall P\in \Sigma_r(X). 
\eeq
Since {$\Sigma_r(X)$ is a generic CR manifold, we obtain} 
$$ T_PD_r(X)=T_P\Sigma_r(X)+J(T_P\Sigma_r(X)).$$
Therefore \eqref{rank1}
implies that 
\beq\label{G"-orbit}
H_*(T_{P}D_r(X))=T_{H(P)}g'(P)\cdot \widetilde D_r,\quad P\in \Sigma_r(X),
\eeq
where 
$$\widetilde D_r:=\left\{\left(\left(\begin{array}{c}
x \\
0 \end{array}\right);y;\left(\begin{array}{c}
z \\
0 \end{array}\right)\right):x,z\in { M^{\mathbb C}(r, n-r)}, ~y\in { M^{\mathbb C}(n-r,n-r)},~~y-y^t+x^tz-z^tx=0\right\}.$$
{Since the CR structure of $\Sigma_r(X)$ is homogeneous, the same computation holds for {a} general {point} $P\in \Sigma_r(X)$, i.e.,}
$H_*(T_{P}D_r(X))$ is contained in the $G'$-orbit of $T_P\widetilde D_r$ for all $P\in \Sigma_r(X)$. 
Since $H$ is holomorphic, $G'$ acts holomorphically on $TD_{r'}(X')$ and $\Sigma_r(X)$ is a generic CR manifold in $D_r(X)$, we obtain that for all $P\in D_r(X)$, $T_{H(P)}H(D_r(X))$ is contained in the $G'$-orbit of $T_P \widetilde D_r,$ i.e.,
$$T_{H(P)}H(D_r(X))=T_{H(P)}\widetilde H(D_r(X))$$
for some standard embedding $\widetilde H$. 

Now fix $P\in \Sigma_r(X)$ and choose a maximal rank one subspace $M\subset D_r(X)$ passing through $P$.
By \eqref{rank1}, $H$ sends rank one vectors in $T_P\Sigma_r(X)$ to rank one vectors and hence all vectors in
$H_*(T_PM)$ are rank one vectors. Since $H$ is holomorphic and $\Sigma_r(X)$ is nondegenerate, we obtain 
$$[H_*(v)]\subset \mathscr C_{H(P)}(Gr(n-r, \mathbb C^m)),\quad \forall v\in T_PM$$ 
Since 
$${\rm rank}~ Gr(n-r, \mathbb C^m)\geq{\rm  rank}~ Z_{\tau}\geq 2,\quad \tau\in \mathcal D_0(X)$$
and $\dim M\geq 3$, by \cite{ChHo04}, we obtain 
$$ H(M\cap \Sigma_r(X))\subset M'\cap \Sigma_{r'}(X')$$
for some maximal rank one subspace $M'$ in $Gr(n-r, \mathbb C^m)$. Furthermore 
$M\cap \Sigma_r(X)$ is a nondegenerate hyperquadric in $M$ with mixed Levi-signature and $H$ maps every projective line in $M\cap \Sigma_r(X)$ into a projective line, by Lemma~\ref{hyperquadric}, $H$ restricted to $M$ is a projective linear map. In particular, $H$ maps projective lines to projective lines.
\epf

%Hence the image $H(D_r(X))$ satisfy the conditions in Theorem~\ref{Mok-image}. In particular, $H(D_r(X))$ is equivalent to $D_r(X)$ in $D_{r'}(X')$ and therefore, we may regard $H$ as a holomorphic self map of $D_r(X)$. Since $H$ satisfies the condition in Lemma~\ref{H-resp-autom}, $H$ is a standard embedding. 
%\epf

\bl\label{exist_subgrassmannian}
For $X = LGr_n$ and $X' = Gr(q', p')$ or $X=OGr_n$ and $X'=OGr_{n'}$, 
{assuming $r > 1$ let $ U \subset D_r(X)$ be a connected open set and $H\colon U \to D_{r'}(X')$ be a subgrassmannian respecting holomorphic immersion such that 
$$H(\Sigma_r(X)\cap U)\subset \Sigma_{r'}(X')$$
and
$$
 H(U)\not\subset \Sigma_{r'}(X').
$$
Then there exists a subgrassmannian $M$ of $D_{r'}(X')$ isomorphic to $Gr(n-r, \mathbb C^{2n})$ if $X=LGr_n$, {isomorphic to $OGr(2[n/2]-2r, \mathbb C^{2n})$ if $X=OGr_n$} such that ${ H(U)} \subset
M$.}
\el
\begin{proof}
First we assume that
$X=LGr_n$ and $X'=Gr(q',p')$ so that $D_r(X)=SGr(n-r, \mathbb C^{2n})$ and $D_{r'}(X')=Gr(q'-r', \mathbb C^{p'+q'})$. 
In the proof of Lemma~\ref{one jet of H}, we can choose a subgrassmannian of $D_{r'}(X')$ isomorphic to $Gr(n-r, \mathbb C^m)$ that contains
$H(D_r(X))$. Hence we may assume that 
$D_{r'}(X')=Gr(n-r,\mathbb C^m)$. 

Let
$$ Z=H(D_r(X)).$$
For $P\in Z$, choose a unique minimal subgrassmannian $M_P$ passing through $P$ such that
\beq\label{T_PZ}
 T_P Z\subset  T_PM_P.
\eeq
By Lemma~\ref{one jet of H}, 
$M_P$ is of the form $Gr(n-r, V_P)$ for some $V_P\subset \mathbb C^m$ with $\dim V=2n.$
Therefore we can choose a Grassmannian frame $Z_1,\ldots, Z_{n-r}, X_{n-r+1},\ldots,X_{m}$ of $Gr(n-r, \mathbb C^m)$ 
such that
$$ { {\rm Span}_\mathbb C} \{Z_1,\ldots,Z_{n-r}\}=P$$
and 
$$ P+{ {\rm Span}_\mathbb C}\{X_{n-r+1},\ldots,X_{2n}\}=V_P. $$
Let  $\{\mu_\alpha^{~H}\}$ be a collection of one forms such that
$$ dZ_\alpha=\mu_\a^{~H}X_H\mod P.$$
Then by \eqref{T_PZ},
$$ T_PZ\subset \{\mu_\a^{~H}=0, H=2n+1,\ldots,m\}.$$
Furthermore, since
$$ T_PZ=H_*(T_P D_r(X))$$
for some standard embedding $H: D_r(X)\to D_{r'}(X')$, we can choose $X_H, H=n-r+1,\ldots, X_m$ such that
$$ T_PZ= \{\mu_\a^{~H}=0, H=2n+1,\ldots,m\}\cap \{\mu_\a^{~n-r+\b}-\mu_\b^{~n-r+\a}=0,\a,\b=1,\ldots n-r\}.$$
Since we choose a Grassmannian frame, we obtain
$$ d\mu_\a^{~H}=\mu_\a^{~K}\wedge \Omega_K^{~H}\mod \mu_\b^{~H},~\b=1,\ldots,n-r$$
for some one forms $\Omega_K^{~H}$ such that
$$ dX_K=\Omega_K^{~H}X_H\mod P.$$
Therefore on $TZ$, we obtain
$$0=\sum_{k=n-r+1}^{2n}\mu_\a^{~k}\wedge \Omega_k^{~H}.$$
Since $\mu_\a^{~k},~k=n-r+1,\ldots,2n$ are linearly independent for all fixed $\a$, by Cartan's lemma we obtain
$$
\Omega_k^{~H}=0 \mod \{\mu_\a^{~\ell},\ell=n-r+1,\ldots,2n\}.$$
Since $k$ is independent of $\a=1,\ldots,n-r$ and $n-r\geq 2$, we obtain
$$
\Omega_k^{~H}=0$$
which implies
$$dZ_\alpha=dX_j=0\mod V_P,\quad \a=1,\ldots,n-r,~ j=n-r+1,\ldots,2n,$$
i.e., $V_P$ is independent of $P$.

{
Now assume that $X=OGr_n$ and $X'=OGr_{n'}$ so that $D_r(X)=OGr(2[n/2]-2r, \mathbb C^{2n})$ and $D_{r'}(X')=OGr(2[n'/2]-2r',\mathbb C^{2n'})$.
Since we may regard $OGr(2[n'/2]-2r',\mathbb C^{2n'})$ as a submanifold in $Gr(2[n'/2]-2r',\mathbb C^{2n'})$, by the same argument as above, 
we obtain that there exists a subspace $W\subset \mathbb C^{2n'}$ of dimension $2n$ such that 
$$ H(D_r(X))\subset W_0\oplus Gr(2[n/2]-2r, W)$$
for some $W_0$.} {Let $a := \dim(W_0)$, $b := (2[n/2]-2r)$ so that $a + b = 2([n'/2] - 2r')$, and let the base point $P$ correspond to $W_0 \oplus E_0$, where $[E_0] \in Gr(b,W)$. In what follows let $V'$ denote any element in $Gr(b,W)$ such that $W_0\oplus V' \in H(D_r(X))$.
{Since 
$$H(D_r(X))\subset D_{r'}(X')=OGr(2[n'/2]-2r',\mathbb C^{2n'}), $$} {
we have $$S_{n'}(W_0 \oplus V';W_0\oplus V') = 0$$ 
whenever $W_0\oplus V' \in H(D_r(X))$.
In particular, $W_0 \subset \mathbb C^{2n'}$ is an $S_{n'}$-isotropic $a$-plane, $V' \subset \mathbb C^{2n'}$ is an  $S_{n'}$-isotropic $b$-plane, and $W_0$ and $V'$ are orthogonal with respect to $S_{n'}$, i.e., $S(W_0,V') = 0$. } {We claim that actually $S(W_0,W) = 0$.
From Lemma~\ref{one jet of H} it follows readily that $S_{n'}|_W$ is nondegenerate.  Suppose there exists some $w \in W$ such that $w$ is not orthogonal to $W_0$ with respect to $S_{n'}$.  Then, for any $S_{n'}$-isotropic $n$-plane $V''$ in $W$ containing $w$ $S(W_0,V'') \neq 0$, so that $[W_0 \oplus V'']\not\in OGr(2[n'/2]-2r',\mathbb C^{2n'})$, hence ${[W_0 \oplus V'']} \notin H(U)$.  Define $\mathscr S := (W_0 \oplus OGr(n-r,W))\cap OGr(2[n'/2]-2r',\mathbb C^{2n'})$. Then, $\mathscr S \subsetneq W_0 \oplus OGr(n-r,W)$, so that $\dim(H(U)) \le \dim(\mathscr S) < \dim(OGr(n-r,W)) = \dim(U)$, a contradiction since we know that $H$ is a holomorphic immersion. Our claim follows, and we conclude that $H(U)$ is an open subset of the subgrassmannian $M := W_0 \oplus OGr(n-r,W)$ isomorphic to $OGr(2[n/2]-2r,\mathbb C^{2n})$, as desired.  The proof of \ref{exist_subgrassmannian} is completed.}}
\epf

\emph{Proof of Proposition~\ref{H respects}}: \ \rm
If $X$ and $X'$ are of the same type, then as in the proof of Lemma~\ref{exist_subgrassmannian}
there exists a subgrassmannian $Y$ in $X'$ which is biholomorphic to $X$
such that $H(D_r(X)) \subset D_r(Y) $. Hence we may consider $H$ as a map from $D_r(X)$ into $D_r(X)$.
By Theorem~9 in \cite{M08b} and Lemma~\ref{one jet of H}, $H$ is an 
automorphism of $D_r(X)$. Hence 
we obtain the proposition in these cases.

{ 
From now on we assume 
$X=SGr(n-r, \mathbb C^{2n})$ 
and $X'=Gr(n-r,\mathbb C^{2n})$. 
By Lemma~\ref{one jet of H}, we may further assume $H(0;I_{n-r};0) = (0;I_{n-r};0)$
and $dH|_{(0;I_{n-r};0)}=Id$. Since by Lemma~\ref{one jet of H} $H$ is a rational map preserving minimal rational curves,  $H$ is a holomorphic immersion into $X'$ by Proposition~\ref{immersion}.
Then the following lemma and Proposition~\ref{standard} will complete the proof.

\begin{Lem}\label{Hs}
There exists a family of holomorphic maps
$\{H_s\}\colon SGr(n-r,\mathbb C^{2n})\rightarrow Gr(n-r,\mathbb C^{2n})$ with $s\in 
{ \mathbb C^*}$ which converges to a standard embedding on a big Schubert cell $\mathcal W \cong M^\mathbb C(n+r,n-r)$ as $s$ tends to $0$ with respect to the compact-open topology. {Moreover, there exists a $\mathbb C^*$-action $\Psi := \{\Psi_s\}_{s \in \mathbb C^*}$ on $Gr(n-r,\mathbb C^{2n})$ such that $\Psi$ fixes 
$(0;I_{n-r};0)$, preserves $SGr(n-r,\mathbb C^{2n}) \subset Gr(n-r,\mathbb C^{2n})$ as a set and such that $H_s(x;y;z) = \Psi_{\frac{1}{\phantom{.}s\phantom{.}}}(H(\Psi_s(x;y,z)))-(0;I_{n-r};0)$.} 
\end{Lem}
}
\begin{proof}
Choose local {coordinates} $(x;y;z)$ of $Gr(n-r,\mathbb C^{2n})$ defined on a big Schubert cell $\mathcal W \cong M^{\mathbb C}_{n+r, n-r}\subset Gr(n-r,\mathbb C^{2n})$
with $x, z\in { M^{\mathbb C}(r, n-r)}, y\in { M^{\mathbb C}(n-r,n-r)}$ so that $SGr(n-r,\mathbb C^{2n})$ is defined locally by
\begin{equation}\nonumber
	y-y^t+x^tz-z^tx=0.
\end{equation}
%For $s\in \mathbb C^*$, consider the automorphism
%$\phi_s$ of $SGr(n-r,\mathbb C^{2n})$ given by
%\begin{equation}
%\phi_s(x;y;z) = \left(sx; s^2y; sz\right)
%\end{equation}
%in the coordinate $(x;y;z)$.
Let $(X;Y;Z)$ be local {coordinates} of $Gr(n-r,\mathbb C^{2n})$ such that $\Sigma_{r}(Gr(n-r,\mathbb C^{2n}))$ can be expressed by
$$
-I_{n-r} - X^*X+Y^*Y + Z^*Z=0,
$$
where $Y\in M^{\mathbb C}_{n-r,n-r}$, $X\in M^{\mathbb C}_{r, n-r}$, $Z\in M^{\mathbb C}_{r, n-r}$.
Let $(H_X, H_Y, H_Z)$ be the coordinate expression of $H$ with respect to $(X;Y;Z)$.
Then by Lemma~\ref{one jet of H}, we may assume 
\beq\label{H-power}
H=(x;y;z)+O(\|(x;y-I_{n-r};z)\|^2).
\eeq
Moreover, since we have $H(\Sigma_r(X))\subset \Sigma_r(X')$, we obtain
$$-I_{n-r}-H^*_XH_X+H^*_YH_Y+H^*_ZH_Z=u\cdot\left(-I_{n-r}-x^*x+y^*y+z^*z\right)$$
for some $C^\omega$ function $u$. Hence by power series expansion, 
\begin{equation}\label{xz-derivatives}
0=-H_X^*\frac{\partial^{|\alpha|+|\beta|} H_X}{\partial x^{\alpha}\partial z^{\beta}}
+ H_Y^*\frac{\partial^{|\alpha|+|\beta|} H_Y}{\partial x^{\alpha}\partial z^{\beta}}
+H_Z^*\frac{\partial^{|\alpha|+|\beta|} H_Z}{\partial x^{\alpha}\partial z^{\beta}}
= \frac{\partial^{|\alpha|+|\beta|} H_Y}{\partial x^{\alpha}\partial z^{\beta}}
\end{equation}
at $(0;I_{n-r};0)$ for any multi-indices $\alpha$, $\beta$.
Let 
$$
H_Y= I_{n-r}+\widetilde H_Y=I_{n-r}+\sum_{|\alpha|\geq 1} B_\a w^\a,
$$ 
with $w=(x, y-I_{n-r},z)$ be the power series expansion of $H_Y$ at $(0;I_{n-r};0)$.
Then \eqref{xz-derivatives} implies
\beq\label{H_Y comp}
\widetilde H_Y=y-I_{n-r}+O(\|( x,z)\|^3+\|y-I_{n-r}\|^2).
\eeq

Now for $0\neq s\in \mathbb C$, define a holomorphic map $H_s$ on $X$ whose restriction on the big Schubert cell $M^{\mathbb C}_{n+r, n-r}\cap SGr(n-r,\mathbb C^{2n})$ is given by
\begin{equation}\nonumber%\label{F_s}
H_s(x;y;z) 
= \left( 
\frac{1}{s} H_X(w_s);
I_{n-r}+\frac{1}{s^2}\widetilde H_Y(w_s); 
\frac{1}{s}H_Z(w_s)
\right),
\end{equation}
where $w_s=(sx; s^2(y-I_{n-r}); sz)$.
In particular, 
$H_s\colon SGr(n-r, \mathbb C^{2n})\rightarrow Gr(n-r, \mathbb C^{2n})$ is a holomorphic immersion.
%which is given by 
%Note that for any compact subset $K\subset\subset
%M^{\mathbb C}_{n+r, n-r}\subset Gr(n-r, \mathbb C^{2n})$, there exists $s$ with 
%sufficiently small $|s|$ so that 
%$H(sx;s^2(y-I_{n-r}); sz)$ belongs to
%the big Schubert cell on where the coordinate $(X;Y;Z)$ is defined for any $(x;y;z)\in K$.
%$$
%F_s:=T^{-1}\cdot\psi_s^{-1}\cdot T\cdot H\cdot \phi_s \colon SGr(n-r,\mathbb C^{2n})\rightarrow Gr(n-r, \mathbb C^{2n}).$$
Furthermore, by \eqref{H-power} and \eqref{H_Y comp}, we obtain
\begin{equation}\nonumber
H_{s}(x;y;z) =(x;y;z)+ O(s),
\end{equation}
implying that
$H_{s}$ converges uniformly to $H_0(x;y;z):=(x;y;z)$ on any compact subset $K\subset M^{\mathbb C}(n+r,n-r)\cap SGr(n-r, \mathbb C^n)$
as $s$ tends to $0$.

{Defining $\Psi_s(x;y;z) := w_s+(0;I_{n-r};0) = (sx;s^2(y-I_{n-r});sz)+(0;I_{n-r};0)$  on the big Schubert cell $\mathcal W$, for $s \in \mathbb C^*$ we have $H_s(x;y;z) = \Psi_{\frac{1}{\phantom{.}s\phantom{.}}}(H(\Psi_s(x;y;z))- (0;I_{n-r};0)$.  It is clear that $\Psi := \{\Psi_s\}_{s\in \mathbb C^*}$ fixes $(0;I_{n-r};0)$ and that it is a $\mathbb C^*$ action on $\mathcal W$.  Furthermore, from the defining equation $y-y^t+x^tz-z^tx=0$ for $SGr(n-r,\mathbb C^{2n}) \cap \mathcal W$, it follows readily that $\Psi$ preserves $SGr(n-r,\mathbb C^{2n}) \cap \mathcal W$ as a set. To complete the proof of Proposition~\ref{Hs} it remains to check that each $\Psi_s$ extends to an automorphism of $Gr(n-r;\mathbb C^{2n})$ yielding hence a $\mathbb C^*$-action on the latter manifold.  

Writing $\Theta_s(x;y;z) := (sx;s^2y;sz)$ we have $\Psi_s(x,y;z) = \Theta_s(x;y-I_{n-r};z)+(0;I_{n-r};z) = T_{P_0}\circ \Theta_s\circ  T_{-P_0}$, where $P_0 = (0;I_{n-r};0)$ and $T_Q(w) = w + Q$, for $Q \in \mathcal W$, is a Euclidean translation on $\mathcal W$. {Recall} that $G' = {\rm Aut}(Gr(n-r,\mathbb C^{2n}))$.  With respect to the Harish-Chandra decomposition $\mathfrak g' = \mathfrak m^{\prime +} \oplus \mathfrak k^{\prime \mathbb C} \oplus \mathfrak m^{\prime -}$ of the Lie algebra $\mathfrak g'$ of $G'$, a Euclidean translation in Harish-Chandra coordinates extends to an element of the commutative Lie subgroup $M^{\prime +} = {\rm exp}\left(\mathfrak m^{\prime +}\right) \subset G'$, thus $\{\Psi_s\}_{s \in \mathbb C^*}$ is a conjugate of $\{\Theta_s\}_{s \in \mathbb C^*}$ in $G'$ and it suffices to check the latter is a $\mathbb C^*$-action. If in place of the coordinates $(x;y;z)$ we use the matrix form $
\Gamma = \Small\begin{pmatrix}x\\y\\z
\end{pmatrix}
 \normalsize \in M^\mathbb C(n+r,n-r)
$ 
as coordinates for points on $\mathcal W$, then $\Theta_s(\Gamma) = D_s\Gamma$, 
%where $\cdot$ signifies matrix multiplication, 
for some invertible (diagonal) matrix $D_s \in M^\mathbb C(n+r,n+r)$.  Now $K^{\prime\mathbb C} = {\rm exp}\left(\mathfrak k^{\prime\mathbb C}\right)$ consists of invertible linear transformations $\Gamma \mapsto A\Gamma B$ where $A$ resp $B$ is an invertible $(n+r) \times (n+r)$ resp. $(n-r) \times (n-r)$ matrix, hence $\Theta_s \in K^{\prime\mathbb C} \subset G'$ for $s \in \mathbb C^*$.  As a consequence, $\Theta = \{\Theta_s\}_{s\in \mathbb C^*}$ and hence $\Psi = \{\Psi_s\}_{s\in \mathbb C^*}$ are $\mathbb C^*$-actions on $Gr(n-r,\mathbb C^{2n})$, as desired.  The proof of Proposition~\ref{Hs} is complete.}
\end{proof} 
%and it follows that for each $s \in \mathbb C^*$, and for any point $P := (x;y,z) \in M^\mathbb C(n+r,n-r)$, the linear isomorphism $\Theta_s: M^\mathbb C(n+r,n-r) \to M^\mathbb C(n+r,n-r)$ preserves the rank of $P$ as a matrix.  Any such a map necessarily preserves the underlying Grassmann structure on $Gr(n-r,\mathbb C^{2n})$ (since $d\Theta_s$ transforms matrices of rank 1 in $T_P(\mathcal W) \cong M^\mathbb C(n+r,n-r)$ to matrices of rank 1 in $T_{\Theta_s(P)}(\mathcal W)$),
%and must therefore extend to an automorphism of the latter manifold, by \cite{Oc70}, or by the Cartan-Fubini extension theorem of \cite{HM01}.  
%$
%\begin{pmatrix}x\\y\\z
%\end{pmatrix}
%$

%Since $\{H_s\}$ is a normal family on the big Schubert cell, 
%we complete the proof.

%Since we have 
%\begin{equation}\nonumber
%H_{s,Y}(x;y;z) 
%= y +  \sum_{|\beta| =2} B_\beta(sx,s^2(y-I_{n-r}),sz)^\beta
%+ O(s),
%\end{equation}
%by \eqref{xz-derivatives} 
%$H_{s,Y}$ converges uniformly to 
%$y$ on any compact subset in $M^{\mathbb C}_{n+r, n-r}$.
%
%As a result, $H_0$ is a standard embedding 
%$H_0(x;y;z)=(x;y;z)$ on the big Schubert cell. Since $H_0$ can be extended to whole $SGr(n-r,\mathbb C^{2n})$ as the identity map, we obtain the lemma.

{
We note that in standard notation the $\mathbb C^*$-action $\Theta$ is generated by an element $L$ of the Cartan subalgebra $\mathfrak h' \subset \mathfrak g' \cong \mathfrak{sl}(2n,\mathbb C)$ such that ${\rm ad}(L)$ preserves the Lie subalgebra $\mathfrak g' \subset \mathfrak g$, $\mathfrak g' \cong \mathfrak{sp}(n,\mathbb C)$ and such that the restriction of ${\rm ad}(L)$ to $\mathfrak{sp}(n,\mathbb C)$ defines on the latter the structure of a graded Lie algebra associated to the marked Dynkin diagram $(C_n,\alpha_{n-r})$, in the notation of {\cite{Y93}}, which is the graded Lie algebra structure on $\mathfrak{sp}(n,\mathbb C)$ with parabolic subalgebra $\mathfrak p$ underlying the rational homogeneous manifold $G/P \cong SGr(n-r,\mathbb C^{2n})$. Thus $\mathfrak g = \mathfrak g_{-2} \oplus  \mathfrak g_{-1} 
\oplus  \mathfrak g_{0} \oplus  \mathfrak g_{1} \oplus  \mathfrak g_{2}$, $T_0(G/P) = \mathfrak g/\mathfrak p \cong \mathfrak g_1 \oplus \mathfrak g_2$, $[L,v_1] = v_1$ for $v_1 \in \mathfrak g_1$ and $[L,v_2] = 2v_2$, which explains the different exponents in $\Theta_s(x;y;z) = (s;s^2y;sz)$.  Thus ${\rm ad}(L)|_\mathfrak g$ defines the standard $\mathbb C^*$-action {$\Theta$} at $0 = eP \in G/P$ with 0 as the isolated fixed point serving as a 1-parameter group of {dilations} which replaces the 1-parameter group of dilations in the case of irreducible Hermitian symmetric spaces of the compact type in \cite{M19} defined by the Euler vector field and expressible in terms of Harish-Chandra coordinates
as scalar multiplications $\Theta_s(x) = sx$ for $s \in \mathbb C$.}

\section{Induced moduli map}\label{Induced moduli map}
{We start with some relevant general facts about subvarieties of irreducible Hermitian symmetric spaces of the compact type $M$.  A characteristic subspace $\Gamma$ of $M$ is an {\it invariantly geodesic complex submanifold} of $M$ according to \cite{MT92} in the sense that it is totally geodesic in $(M,s)$ with respect to any choice of K\"ahler-Einstein metric $s$ on $M$ (Section~\ref{Hermitian symmetric spaces}).  Equivalently, fixing a big Schubert cell $\mathcal W$, $\mathcal M \cong \mathbb C^m$ in terms of Harish-Chandra coordinates, $S \subset M$, $0 \in S$, is invariantly geodesic in $M$ if and only if for any $\gamma \in P$, $\gamma(P)\cap \mathcal W$ is a linear subspace of $\mathbb C^m$. It follows that the set of invariantly geodesic complex submanifolds of $M$ is closed under taking intersections.  In the case where $M$ is the Grassmann manifold $Gr(a,b)$, $0 = [V_0]$, writing $T_0(M) = V_0^*\otimes \mathbb C^{a+b}/V_0 =: A \otimes B$, for an invariantly geodesic complex submanifold $S \subset M$ passing through $0$ we have $T_0(S) = A'\otimes B'$, where $A' \subset A$, $B' \subset B$ are linear subspaces.  Given any family $\{S_\alpha\}$ of invariantly geodesic complex submanifolds of $Gr(a,b)$, $T_0(S_\alpha) =: A_\alpha \otimes B_\alpha$, the intersection $S : = \bigcap\left\{S_\alpha\right\}$ is determined by $T_0(S) = A \otimes B$, where $A := \bigcap\left\{A_\alpha\right\}$, $B :=\bigcap\left\{S_\alpha\right\}$. $S \subset M$ is a subgrassmannian. In the case of $M = LGr_n$, writing $T_0(M) = S^2V_0$, a characteristic subspace $\Gamma$ passing through $0 \in LGr_n$ is determined by $T_0(\Gamma) = S^2A$ for some linear subspace $A \subset V_0$, hence the intersection of any family of characteristic subspaces is necessarily a characteristic subspace.  In the case where $M = OGr_n$, writing $T_0(M) = \Lambda^2V_0$, a characteristic subspace $\Gamma$ passing through $0 \in LGr_n$ is determined by $T_0(\Gamma) = \Lambda^2A$ for some linear subspace $A \subset V_0$ of even codimension, hence the intersection $S$ of any family of characteristic subspaces passing through $0 \in M$ is determined by $T_0(S) = \Lambda^2A$. $S \subset M$ is a characteristic subspace if and only if $A \subset V_0$ is of even {codimension}, otherwise embedding $OGr_n$ into $OGr_{n+1} := M'$ as usual, $S \subset M'$ is a characteristic subspace.}

Let now $\Omega$ and $\Omega'$ be irreducible bounded symmetric domains of type I, II or III and let $f:\Omega\to\Omega'$ be a proper holomorphic map. In this section, we define induced moduli maps $f^\sharp_r$, {$f^\sharp_{r, \frac12}$}, $f^\flat_r$ and  {$f^\flat_{r,\frac12}$} on $\mathcal{D}_r(X)$, {$\mathcal{D}_{r,\frac12} (X)$}, $D_r(X)$ and  {$D_{r,\frac12}(X)$}, respectively. 
%\medskip

Let $r> 0$ be fixed. Consider a manifold 
$$\mathcal{U}_r(X):=\{(P, \sigma)\in X\times \mathcal{D}_r(X)\colon P\in X_\sigma\}\subset X\times  \mathcal{D}_r(X).$$
Then, there is a canonical double fibration 
$$\pi_1\colon\mathcal{U}_r(X)\to X,\quad
\pi_2\colon \mathcal{U}_r(X)\to \mathcal{D}_r(X).$$
Define 
$j\colon \mathcal{U}_r(X)\to \mathcal{G}(n_r, TX)$ with $n_r=\dim T_PX_\sigma$ by $j(P,\sigma)=T_P X_\sigma$, where $ \mathcal{G}(n_r, TX)$ is a Grassmannian bundle over $TX$.
%{
%Similarly we can define $\mathcal U_{r,\frac12}(X)$ for the type II manifold $X$ and the map $j\colon \mathcal U_{r,\frac12}(X)\rightarrow \mathcal G(n_{r, \frac12}, TX)$.}
Then, $j$ is a $G$-equivariant holomorphic embedding such that $j(\mathcal {U}_r(X))=\mathcal{NS}_r(X)$.
%{and $j(\mathcal U_{r,\frac12}(X)) = \mathcal NS_{r,\frac12}(X)$.}
%Consider the first canonical embedding $X\to \mathbb{P}^M$. Then for each $(P,\sigma)\in \mathcal{U}_r(X)$,
%there exists a unique projective space 
%$\mathbb{P}^{m_r}_\sigma\subset \mathbb{P}^M$ depending only on $\sigma$ such that 
%\beq\label{projective}
%	j(P,\sigma)= T_P\left(\mathbb{P}^{m_r}_\sigma\cap X\right). 
%\eeq
%More precisely, any subgrassmannian $\sigma=\{x\in Gr(m,n): A\subset x\subset B\}$ in $Gr(m,n)$ is a smooth Schubert cycle,

We will define $f^\sharp_r$ and {$f^{\sharp}_{r,\frac12}$} as follows. {For} each $\sigma\in \mathcal{D}_r(\Omega)$ and $\gamma\in \mathcal D_{r,\frac12}(\Omega)$,
define $f^\sharp_r(\sigma)$ and $f^\sharp_{r,\frac12}(\gamma)$ by 
\begin{equation}\label{intersection}
	X'_{f^\sharp_r(\sigma)}:=
	\bigcap_{\sigma'} X'_{\sigma'} {\quad \text{ and }\quad 
	X'_{f^\sharp_{r,\frac12}(\gamma)}:=
	\bigcap_{\gamma'} X'_{\gamma'},}
\end{equation}
where the intersection is taken over all characteristic subspaces $X'_{\sigma'}$ of $X'$ containing $f(\Omega\cap X_\sigma)$ {and $X'_{\gamma'}$ of $X'$ containing $f(\Omega\cap X_\gamma)$,
respectively.} {We remark that since the intersection of subgrassmannians is also a subgrassmannian, the maps $f^\sharp_r$ and $f^\sharp_{r,\frac12}$ in \eqref{intersection} are well defined. Furthermore, since $f$ is a proper holomorphic mapping and hence characteristic subdomains of $\Omega$ are mapped to characteristic subdomains of $\Omega'$ (\cite[Proposition 1.1]{Ts93}), the ranks of $X'_{f^\sharp_r(\sigma)}$ and $X'_{f^\sharp_{r,1/2}(\gamma)}$ should be strictly less than the rank of $X'$.}

Then, there exists a flag manifold $\mathcal{F}(a_r,b_r;V_{X'})$ such that $f^\sharp_r(\sigma)\in \mathcal{F}(a_r,b_r;V_{X'})$ for {a general member} $\sigma\in \mathcal{D}_r(\Omega)$,
{where $V_{X'}$ is a suitable vector space according to the type of $X'$, see Section~\ref{moduli spaces}}. 
%{ "If $X'$ is of type II or III, then intersection of { any family of} characteristic subspaces is again a characteristic subspace. Hence there exists $i_r$ such that $f^\sharp_r(\sigma)\in \mathcal D_{i_r}(X')=\mathcal F_{i_r}(X').$ If $X'$ is of type I, then"-Remove this}
Denote this $\mathcal{F}(a_r,b_r;V_{X'})$ by $\mathcal{F}_{i_r}(X')$, where $i_r$ is defined by
$
	i_r:=q'-a_r
$
{if $X'$ is one of $Gr(q',p')$ and $LGr_{q'}$, $i_r:=2[n'/2]-a_r$ if $X'$ is $OGr_{n'}$.  If $X'$ is one of $Gr(q', p')$ and $LGr_{q'}$, then $i_r\leq q'-1$. If $X'$ is $OGr_{n'}$, then $i_r\leq 2[n'/2]-2$.}
{Similarly, we define $\mathcal F(a_{r,\frac12}, b_{r,\frac12};V_{X'})$ and denote it by $\mathcal F_{i_{r,\frac12}}(X')$, where $i_{r,\frac12}$ is defined by $i_{r,\frac12} = 2[n'/2]-a_{r,\frac12}$.}
Define  
$$
	\mathcal{F}_{i_r}(\Omega'):=\{\sigma'\in \mathcal{F}_{i_r}(X')\colon X'_{\sigma'}\cap\Omega'\neq\emptyset\},
$${   
$$
	\mathcal{F}_{i_{r,1/2}}(\Omega'):=\{\sigma'\in \mathcal{F}_{i_{r,1/2}}(X')\colon X'_{\sigma'}\cap\Omega'\neq\emptyset\},
$$}
$$
	\mathcal{F}_{i_r}(S_m(X')):=\{\sigma'\in \mathcal{F}_{i_r}(X')\colon X'_{\sigma'}\cap S_m(X')\text{  is open in  }X'_{\sigma'}\}.
$$
{For the definition of $S_m(X')$, we refer the reader to Section~\ref{Hermitian symmetric spaces}, pp. 8.}
\bl
$f^\sharp_r\colon \mathcal{D}_r(\Omega)\to \mathcal{F}_{i_r}(\Omega')$ {and $f^\sharp_{r,\frac12} \colon \mathcal{D}_{r,\frac12}(\Omega)\to \mathcal{F}_{i_{r,\frac12}}(\Omega')$  are meromorphic maps.}
\el
\bpf
{
Since the proof for the map $f^\sharp_{r,\frac12}$ is similar to that for $f^\sharp_{r}$, we will only give a proof for $f^\sharp_{r}$.}
Consider a map $\mathcal F_{r}:\mathcal{U}_r(\Omega)\to \mathcal{F}_{i_r}(X')$ defined by
$$\mathcal F_r(P, \sigma)=f^\sharp_r(\sigma).$$
Suppose $\mathcal F_r$ is a meromorphic map. Then by taking a local holomorphic section of the fibration $\pi_2:\mathcal{U}_r(\Omega)\to \mathcal D_r(\Omega)$, we can complete the proof.

Let 
$$ \mathcal M:=\{(y, \sigma')\in X'\times \mathcal F_{i_r}(X'): y\in X'_{\sigma'}\}.$$
Then as above, there exist a double fibration 
$$\pi_1'\colon\mathcal{M}\to X',\quad
\pi_2'\colon \mathcal{M}\to \mathcal{F}_{i_r}(X')$$
and a holomorphic embedding of $\mathcal M$ into $\mathcal G(n'_r, TX')$ for $n'_r=\dim X'_{\sigma'}$. Hence we may regard $\mathcal M$ as a closed submanifold of $\mathcal G(n'_r, TX')$.

{We identify a small neighborhood of $E$ in $X'$ as a submanifold in the matrix space via the property
$T_EX'\subset \textup{Hom}(E, V_{X'}/E)$}. Fix a point $P_0\in {\Omega}$ and let $E=f(P_0)$. Let $(P, \sigma)\in \mathcal U_r(\Omega)$ for $P$ sufficiently close to $P_0$. Consider a subspace
$$ \mathcal N_{(P,\sigma)}^k:={ {\rm Span}_\mathbb C}\left\{\partial^\alpha \left(f\big|_{X_\sigma}\right )(P):|\alpha|\leq k\right\}\subset  \textup{Hom}(E,V_{X'}/E).$$
Then there exists an integer $k_0$ such that for {a general pair} $(P,\sigma)$, 
$$\mathcal N_{(P,\sigma)}^{k}=\mathcal N_{(P,\sigma)}^{k+1},\quad k\geq k_0.$$
Define
$$R_{(P,\sigma)}:={{\rm Span_\mathbb C}}\left\{ {\rm Im}(A): A\in \mathcal N_{(P,\sigma)}^{k_0}\right\},
\quad
K_{(P,\sigma)}:=\bigcap\left\{{ {\rm Ker}}(A): A\in \mathcal N_{(P,\sigma)}^{k_0}\right\}.$$
Then
$$Gr_{(P,\sigma)}:=\left\{A\in \textup{Hom}\left(E, R_{(P,\sigma)}\right):{\rm Ker}(A)\supset K_{(P,\sigma)}\right\}$$
is a linear subspace in $\textup{Hom}(E, V_{X'}/E)$ such that
$$ T_{f(P)}X'_{f^\sharp_r(\sigma)}=Gr_{(P,\sigma)}$$
for {a general pair} $(P,\sigma)$ by minimality of $X'_{f^\sharp_r(\sigma)}$.
Moreover the defining function of $Gr_{(P,\sigma)}$ depends meromorphically on the $k_0$-th jet of $f$ at $P$ and $T_PX_\sigma$. Hence the closure of 
$$ \left\{\left(P,\sigma,f(P),Gr_{(P,\sigma)}\right):(P,\sigma)\in \mathcal{U}_r(\Omega)\backslash S(f^\sharp_r) \right\}$$
in $\mathcal{U}_r(\Omega)\times\mathcal M$ is an analytic variety whose defining function depends meromorphically on the $k_0$-th jet of $f$,
where we let
$$S(f_r^\sharp):=\{(P, \sigma):  \dim Gr_{(P, \sigma)}\text{ is not maximal}\},$$
implying that $\mathcal F_r$ is a meromorphic map.
\epf

\bl \label{rational extension}
%When $\Omega$ is of the type~I or III, 
$f^\sharp_r$ has a rational extension $f^\sharp_r\colon \mathcal{D}_r(X)
\to \mathcal{F}_{i_r}(X')$ .
%When $\Omega$ is of the type~II, 
%$f^\sharp_{2r}$ has a rational extension 
%$f^\sharp_{2r} \colon \mathcal{D}_{2r}(X)\to \mathcal{F}_{i_{2r}}(X')$.
\el

\bpf
By using Lemma~\ref{NS}, the same proof of Proposition~2.6 in \cite{MT92} can be applied (cf. Section~\ref{Hermitian symmetric spaces}).
\epf

Since $f^\sharp_r$ is rational and $\mathcal D_r(S_k(\Omega))$ is not contained in any complex subvariety, we obtain 
$${ {\rm Dom}}(f^\sharp_r)\cap \mathcal{D}_r(S_k(X))\neq \emptyset,$$
{where $S_k(X)$ is a $G_o$-orbit consisting of boundary components of rank $k$ in the boundary of $\Omega\subset X$ (see Section~\ref{Hermitian symmetric spaces}).}
%If $X$ is of type~II, then
%$$Dom(f^\sharp_{2r})\cap \mathcal{D}_{2r}(S_{k}(X))\neq \emptyset.$$

\bl\label{Z-tau-fix}
For each $k\geq r$, there exists $m_k$ depending only on $k$ such that
$$
f^\sharp_r(\mathcal{D}_r(S_k(X))\cap { {\rm Dom}}(f^\sharp_r))\subset \mathcal{F}_{i_r}(S_{m_k}(X')).
$$
%if $X$ is of type~I or III and
%\beq\label{type II}
%	f^\sharp_{2r}(\mathcal{D}_{2r}(S_{k}(X))\cap Dom(f^\sharp_{2r}))\subset \mathcal{F}_{i_{2r}}(S_{m_k}(X'))
%\eeq
%if $X$ is of type~II.
\el
\bpf
We will prove the lemma when $X$ is of type~I. The same proof can be applied to other types.
 
Let $\sigma_0\in \mathcal{D}_r(S_k(X))\cap { {\rm Dom}}(f^\sharp_r)$. Then $X_{\sigma_0}\cap S_k$ is a complex manifold in $S_k$.
Therefore we can choose a totally geodesic subspace of $\Omega$ of the form $\Delta^{q-k}\times \Omega_{0}$ such that 
$X_{\sigma_0}\cap S_k=\{t_0\}\times \Omega_{0}$ for some $t_0\in (\partial\Delta)^{q-k}$. 
Choose a sequence $t_j\in \Delta^{q-r}$, $j=1,2,\ldots,$ converging to $t_0$ and let $\sigma_j\in \mathcal D_r(\Omega)$ be the characteristic subspaces such that  $X_{\sigma_j}\cap \Omega=\{t_j\}\times \Omega_{0}$. Fix a point $x_0\in \Omega_0$. By passing to a subsequence, we may assume that ${ f(t_j,x_0)},~j=1,2,\ldots,$ converges. Since $f$ is proper, the limit $y={\lim\limits_{j\to\infty}} f(t_j,x_0)$ is in the boundary of $\Omega'$. Since $\Omega'$ is convex, there exists a complex linear supporting function $H$ of $\Omega'$ such that $h(y)=0$. Since $h\circ f$ is bounded, we may assume that $h_j:=h\circ f \big|_{\{t_j\}\times \Omega_0}$ is a convergent sequence that converges to $H$.
Since $h_j$ never vanishes while its limit vanishes at $x_0$, $H$ is a trivial function, i.e., cluster points of $\{f(t_j, x): ~j=1,2,\ldots\}$ for any $x\in \Omega_0$ is in the zero set of $h$. Since $h$ is arbitrary, the limit set of $f(\{t_j\}\times \Omega_0)$ should be in a boundary component of $\Omega'$ which contains $y$. Let $S_m(X')$ be a boundary orbit containing $y$.
Since $\sigma_0\in {\rm Dom}(f^\sharp_r)$, we may assume $f^\sharp_r(\sigma_j)$ converges to $f^\sharp_r(\sigma_0).$
Then, $X'_{f^\sharp_r(\sigma_0)}$ contains the limit set of $f(\{t_j\}\times\Omega_0)$, which implies $f^\sharp_r(\sigma_0)\in \mathcal{F}_{i_r}(S'_{m}(X'))$. In particular, we obtain
$$
f^\sharp_r(\sigma)\in \mathcal{F}_{i_r}(S_{m}(X'))
$$
for {a general member} $\sigma\in \mathcal{D}_r(S_k(X))\cap {{\rm Dom}}(f^\sharp_r)).$ By continuity of $f^\sharp_r$, we obtain
$$
f^\sharp_r(\mathcal{D}_r(S_k(X))\cap {{\rm Dom}}(f^\sharp_r)))\subset \mathcal{F}_{i_r}(S_{m}(X')).
$$

Next we will show that $m$ depends only on $k$. Since $S_k(X)$ is foliated by boundary components of rank $k$, for any $\sigma\in \mathcal{D}_r(S_k)$, there exists a unique $\mu\in \mathcal{D}_k(S_k)$ such that $X_\sigma\cap S_k\subset X_\mu\cap S_k$. Then $f^\sharp_r(\sigma)$ should be contained in $f^\sharp_k(\mu)$. Hence $m$ depends only on $k$.
\epf
%\begin{Rem}
%If $X$ is of type~II, then \eqref{type II} implies
%$$
%	f^\sharp_{2r-1}(\mathcal{D}_{2r-1}(S_{k}(X))\cap Dom(f^\sharp_{2r-1}))\subset \mathcal{F}_{i_{2r-1}}(S_{m_k}(X')).
%$$
%Hence for type~II case it still holds that
%$$
%	f^\sharp_{r}(\mathcal{D}_{r}(S_{k}(X))\cap Dom(f^\sharp_{r}))\subset \mathcal{F}_{i_{r}}(S_{m_k}(X'))
%$$
%for some $m_k$ depending only on $k$, whether $r$ is even or not.
%\end{Rem}

Now consider all moduli maps
$$
	f^\sharp_r\colon \mathcal{D}_r(X)\to \mathcal{F}_{i_r}(X'),\quad r=1,\ldots,q-1.
$$

\bl\label{increasing}
For each $r$, we have $i_{r-1}<i_r.$ Furthermore, {if $X$ is of type II,} {then $i_{r-1}<i_{r-1,1/2}<i_r$ for $r=2,\ldots,q-1$.}
\el

\bpf
By definition, we obtain $i_{r-1}\leq i_r$.
Suppose $i_{r-1}=i_{r}$. Let $\tau\in \mathcal{D}_{r-1}(\Omega)\cap {{\rm Dom}}(f^\sharp_{r-1})$ and let $\sigma\in \mathcal Z_\tau=\mathcal{Z}_{\tau}^r.$ By Lemma~\ref{Z-tau}, we obtain 
$$
	f^\sharp_r(\sigma)\in\mathcal{Z}'_{f^\sharp_{r-1}(\tau)},
$$
which implies that as a subspace of $V_{X'}$, 
$$	
	pr'\circ f^\sharp_r(\sigma)\subset pr'\circ f^\sharp_{r-1}(\tau),
$$
where 
$pr'\colon\mathcal{F}(a,b;V_{X'})\to Gr(a, V_{X'})$
is a projection map defined by 
$$
	pr'(V_1, V_2)=V_1.
$$
Since $i_r=i_{r-1}$ by assumption, we obtain
$$
	\dim pr'\circ f^\sharp_r(\sigma)=\dim pr'\circ f^\sharp_{r-1}(\tau)
$$
and hence 
$$
	pr'\circ f^\sharp_r(\sigma)=pr'\circ f^\sharp_{r-1}(\tau),
$$
i.e., $ pr'\circ f^\sharp_r$ is constant on $\mathcal {Z}_{\tau}$.
Since $\mathcal{D}_r(\Omega)$ is $\mathcal{Z}_{\tau}$-connected, we obtain that $pr'\circ f^\sharp_r$ is a constant map.
On the other hand, by Lemma~\ref{Z-tau-fix}, we obtain
$$
	f^\sharp_r(\mathcal{D}_r(X))\cap \mathcal{F}_{i_r}(S_{k}(X'))\neq \emptyset
$$
for some $k$,
which implies
$$
	pr'\circ f^\sharp_r(V)=pr'(\mu')
$$
for some fixed $\mu'\in \mathcal{F}_{i_r}(S_k(X'))$. In particular, 
$$
	f(\Omega)\subset S_k(X')
$$
contradicting the assumption that $f$ is a proper holomorphic map between $\Omega$ and $\Omega'$. 

{Now suppose $\Omega = D^{II}_n$ and $i_{r-1, \frac12} = i_r$. Then by the similar argument given above, we obtain that $pr'\circ f_r^\sharp$ is a constant map which is a contradiction.
Suppose $i_{r-1} = i_{r-1, \frac12}$. Then again by the similar argument, we obtain that $pr'\circ f_{r-1,\frac12}^\sharp$ is a constant map on $\mathcal{D}_{r-1,\frac12}(\Omega)$. Since 
$$X_\mu=\bigcup_{\sigma\in Q_\mu^{1/2}}{X_\sigma},\quad \mu\in \mathcal D_r(X),$$
$pr'\circ f^\sharp_{r}$ is also constant which is a contradiction.}
\epf
\medskip

Recall that
$$ D_r(X)=pr(\mathcal{D}_r(X)),\quad \Sigma_r(X)=pr(\mathcal{D}_r(S_r(X))),$$
where 
$pr\colon\mathcal{F}(a,b;V_X)\to Gr(a, V_X)$
is a projection map defined by 
$$
	pr(V_1, V_2)=V_1.
$$
Define 
$$
D_r(S_m(X)) : = pr(\mathcal D_r(S_m(X))).
$$
Define
$$
F_{i_r}(X'):=pr'(\mathcal{F}_{i_r}(X')),
$$
$$
{ F_{i_{r,1/2}}(X') := pr'(\mathcal F_{r,1/2}(X')),}
$$
$$
F_{i_r}(\Omega'):=pr'(\mathcal{F}_{i_r}(\Omega')),
$$
$$
{F_{i_{r,1/2}}(\Omega'):=pr'(\mathcal{F}_{i_{r,1/2}}(\Omega')),}
$$
$$
F_{i_r}(S_m(X')):=pr'(\mathcal{F}_{i_r}(S_m(X'))),
$$
where $pr'\colon\mathcal{F}(a,b;V_{X'})\to Gr(a, V_{X'})$
is a projection map defined as above. 
{$F_{i_r}(X')$ is one of $Gr(a_r, \mathbb C^{p'+q'}),~ OGr(a_r, \mathbb C^{2n'}),~SGr(a_r, \mathbb C^{2n'})$ according to the type of $X'$
{and $F_{i_{r,1/2}}(X')$ is $SGr(a_{i_{r,1/2}}, \mathbb C^{2n'})$.}
Note that $F_{i_r}(X')$, $F_{i_r}(\Omega')$ and $F_{i_r}(S_m(X'))$ can be expressed as subsets of $D_{r'}(Y)$, $D_{r'}(\Omega_Y)$ and $D_{r'}(S_{m'}(Y))$, respectively for suitable Hermitian symmetric space $Y$ and its dual bounded symmetric domain $\Omega_Y\subset Y$. For instance, if $X'$ is one of the type I and III, then we can choose $Y$ to be $X'$ itself and if $X'$ is of type II and $n'-a_r$ is odd, then we may regard $OGr(a_r, \mathbb C^{2n'})$ as a submanifold in $OGr(a_r, \mathbb C^{2n'+2})=D_{r'}(OGr_{n'+1})$ for suitable $r'$ by embedding $OGr_n$ into $OGr_{n+1}$ in a usual way.  
}

{

Suppose $X$ is of type II or III. Since $pr\colon \mathcal{D}_r(X)\to D_r(X)$ is a biholomorphic map, $f^\flat_r:=pr'\circ f^\sharp_r\circ pr^{-1}$ is a rational map on $D_r(X)$ such that 
$$
pr'\circ f^\sharp_r=f^\flat_r\circ pr.
$$
If $X$ is of type I, then we have the following lemma.

\bl\label{commute}
Suppose $i_r=i_{r-1}+1.$ Then there exists either a holomorphic or an anti-holomorphic map $f^\flat_r$ defined on a neighborhood  {U} of $\Sigma_r(X)\cap pr({ {\rm Dom}} (f^\sharp_r))$, {$f_r^\flat: U \to F_{i_r}(X')$,} such that
$$
pr'\circ f^\sharp_r=f^\flat_r\circ pr.
$$
Moreover, $f^\flat_r$ has a rational extension to $D_r(X)$. 
\el

\bpf
By Lemma~\ref{one to one}, we can define a smooth map by 
$$
	f^\flat_r:=pr'\circ f^\sharp_r\circ pr^{-1}\colon \Sigma_r(X)\cap pr({ {\rm Dom}}(f^\sharp_r))\to F_{i_r}({X'}).
$$
We will show that $f^\flat_r$ is either a CR or a conjugate CR map.
Then by Lemma~\ref{bracket gen} and analytic disc attaching method (\cite{BER99}), $f^\flat_r$ extends holomorphically or anti-holomorphically to a neighborhood of $\Sigma_r(X)\cap pr({ {\rm Dom}}(f^\sharp_r))$.

Fix a point $Z_0\in \Sigma_r\cap Dom(f^\flat_r)$. Then $(Z_0, Z_0^*)\in \mathcal {D}_r(S_r)\cap Dom(f^\sharp_r).$
Choose an open neighborhood $U$ of $(Z_0, Z_0^*)$ such that $f^\sharp_r$ is holomorphic in $U$. 
Define $F$ on $U$ by 
$$F(A,B):=pr'\circ f^\sharp_r(A, B),\quad (A,B)\in U.$$ 
Since 
$$Z\in \Sigma_r\to \phi_Z:=\langle \cdot, Z\rangle\in Gr(q-r, (\mathbb{C}^{p+q})^*)\sim Gr(p+r, \mathbb C^{p+q})$$
is a conjugate CR map,
to show that $f^\flat_r$ is CR or conjugate CR, it is enough to show that $F$ depends only on $A$ or only on $B$, respectively.
Suppose that on $U$,
$$F(A, B)=F(A, C)$$
for all $B,C$ having $B\cap C$ of codimension one in $B$ as well as in $C$. Since any two points in $Gr(p+r, \mathbb C^{p+q})$ are connected by a chain $B_i, i=1,\ldots,\ell_0$
such that $B_i\cap B_{i+1}$ is of codimension one in $B_i$ and in $B_{i+1}$, $F$ is independent of $B$. Similarly, if
$$F(A,B)=F(C, B)$$
for all $A,C$ having $A\cap C$ of codimension one in $A$ as well as in $C$, then $F$ is independent of $A$.

Assume that none of the above equalities hold, i.e.,
\beq\label{eqF}
F(A, B)\neq F(A, D),\quad F(A,B)\neq F(C,B)
\eeq
for general $A,B,C,D$ such that $(A+C, B\cap D)\in \mathcal {D}_{r-1}(X).$
We may assume that $(A+C, B\cap D)\in Dom(f^\sharp_{r-1})$.
Since 
$$X_{(A,B)}\cap X_{(C,D)}=X_{(A+C, B\cap D)},$$ 
by the definition of $f^\sharp_r$ and $f^\sharp_{r-1}$,
we obtain
$$F(A,B)+F(C,D)\subset pr'\circ f^\sharp_{r-1}(A+C, B\cap D).$$
Since $F$ is not constant and 
$$\dim F(A, B)=q'-i_r=q'-i_{r-1}-1=\dim pr'\circ f^\sharp_{r-1}(A+C, B\cap D)-1,$$
we obtain
$$F(A, B)\subsetneq F(A,B)+F(C,D)= pr'\circ f^\sharp_{r-1}(A+C, B\cap D).$$
By the same argument using \eqref{eqF}, we obtain
$$F(A,B)+F(A,D)=F(A,B)+F(C,B)= pr'\circ f^\sharp_{r-1}(A+C, B\cap D).$$
Choose another $(C',D')$ such that $(A+C', B\cap D')\in \mathcal D_{r-1}(X)$. 
Then we obtain
\beq\label{ab}
 pr'\circ f^\sharp_{r-1}(A+C', B\cap D)=F(A,B)+F(A,D)=pr'\circ f^\sharp_{r-1}(A+C, B\cap D).
\eeq
Similarly, we obtain
$$ pr'\circ f^\sharp_{r-1}(A+C', B\cap D)=F(A,B)+F(C',B)=pr'\circ f^\sharp_{r-1}(A+C', B\cap D'),$$
implying together with \eqref{ab} that
$$pr'\circ f^\sharp_{r-1}(A+C, B\cap D)=pr'\circ f^\sharp_{r-1}(A+C', B\cap D').$$
Now by fixing $(C',D')$ and changing $(A,B)$ with $(A',B')$, we obtain 
$$pr'\circ f^\sharp_{r-1}(A+C, B\cap D)=pr'\circ f^\sharp_{r-1}(A'+C', B'\cap D').$$
Since any two characteristic subspaces of rank $r-1$ is connected by a chain $(X_i, Y_i), i=1,\ldots, \ell$ such that 
$$\dim (X_i+X_{i+1})=\dim X_i+1,\quad \dim Y_i\cap Y_{i+1}=\dim Y_i-1,$$
$pr'\circ f^\sharp_{r-1}$ is constant, contradicting the assumption that $f$ is proper.
Therefore $F$ depends only on $A$ or only on $B$.

Suppose $f^\flat_r$ is a CR map. Let 
$$\Gamma^\sharp_r:=\overline{\{(x, f^\sharp_r(x))\colon x\in { {\rm Dom}}(f^\sharp_r)\}}$$
be the closure of the graph of $f^\sharp_r$. Since $f^\sharp_r$ is a rational map, $\Gamma^\sharp_r$ and its image under the map 
$$\pi=pr\times pr':\mathcal D_r(X)\times \mathcal F_{i_r}(X')\to D_r(X)\times F_{i_r}(X')$$
are irreducible closed varieties. Moreover, since $f^\flat_r$ satisfies
$$pr'\circ f^\sharp_r=f^\flat_r\circ pr,$$
we obtain
$$ \{(y,f^\flat_r(y)):y\in { {\rm Dom}}(f^\flat_r)\cap \Sigma_r\}\subset \pi(\Gamma^\sharp_r)$$
as an open set.
Therefore, $f^\flat_r$ extends to $D_r(X)$ as a meromorphic map whose graph is a dense open subset of $\pi(\Gamma^\sharp_r)$. 
Since $D_r(X)$ is a rational variety, by \cite{Chow}, $f^\flat_r$ is also rational. 
By the same argument, $f^\flat_r$ extends rationally if $f^\flat_r$ is conjugate CR. 
\epf

}

Note that since $f$ is proper, we obtain
\beq\nonumber
f^\flat_r(D_r(\Omega))\cap F_{i_r}(S_m(X'))=\emptyset,\quad \forall m\geq 1.
\eeq
Moreover, by Lemma~\ref{Z-tau-fix}, we obtain
\beq\nonumber
( f^\flat_r)^{-1}\left ( F_{i_r}(S_m(X'))\right)\subset D_r(S_\ell(X))
\eeq
for some $m\geq \ell$.

Fix $r$. For $\tau'\in \mathcal{F}_{i_s}(X')$ with $s<r$, define
$$
	\mathcal{Z}_{\tau'}':=\{\sigma'\in \mathcal{F}_{i_r}(X')\colon X_{\sigma'}'\supset X_{\tau'}'\},
	\quad
	Z_{\tau'}'=pr'(\mathcal{Z}_{\tau'}'),
$${
$$
	(\mathcal{Z}_{\tau'}^{1/2}) ':=\{\sigma'\in \mathcal{F}_{i_{r,1/2}}(X')\colon X_{\sigma'}'\supset X_{\tau'}'\},
	\quad
	(Z_{\tau'}^{1/2})'=pr'(\mathcal{Z}_{\tau'}^{1/2})
$$}
and for $\mu'\in \mathcal{F}_{i_s}({X'})$ with $s>r$, define
$$
	\mathcal{Q}_{\mu'}':=\{[\sigma']\in \mathcal{F}_{i_r}({X'})\colon X_{\sigma'}'\subset X_{\mu'}'\},
	\quad
	Q_{\mu'}'=pr'(\mathcal{Q}_{\mu'}').
$$

\bl\label{Z-tau}
Let $s<r$. Then $f^\flat_r$ satisfies
$$
	f^\flat_r(Z_{\tau}\cap { {\rm Dom}}(f^\flat_r))\subset Z'_{f^\sharp_{s}(\tau)},\quad \tau\in  \overline{\mathcal{D}_s(\Omega)}\cap { {\rm Dom}}(f^\sharp_s)
$$
and
{
$$
	f^\flat_r(Z_{\tau}\cap { {\rm Dom}}(f^\flat_r))\subset Z'_{f^\sharp_{s,1/2}(\tau)},\quad \tau\in  \overline{\mathcal{D}_{s,1/2}(\Omega)}\cap { {\rm Dom}}(f^\sharp_{s,1/2}).
$$
Similarly, $f^\flat_{r,1/2}$ satisfies 
$$
	f^\flat_{r,1/2}(Z_{\tau}\cap { {\rm Dom}}(f^\flat_{r,1/2}))\subset Z'_{f^\sharp_{s}(\tau)},\quad \tau\in  \overline{\mathcal{D}_s(\Omega)}\cap { {\rm Dom}}(f^\sharp_s)
$$
and
$$
	f^\flat_{r,1/2}(Z_{\tau}\cap { {\rm Dom}}(f^\flat_{r,1/2}))\subset Z'_{f^\sharp_{s,1/2}(\tau)},\quad \tau\in  \overline{\mathcal{D}_{s,1/2}(\Omega)}\cap { {\rm Dom}}(f^\sharp_{s,1/2}).
$$
}
\el
\bpf
First assume that $\tau\in \mathcal{D}_s(\Omega).$
Choose $\sigma\in \mathcal{D}_r(\Omega)$ such that $\sigma\in \mathcal Z_{\tau}$, i.e., 
$$\emptyset\neq X_\tau\cap\Omega\subset X_\sigma\cap\Omega.$$
Since
$$
	f(X_\tau\cap \Omega)\subset f(X_\sigma\cap \Omega),
$$
$f(X_\tau\cap\Omega)$ is contained in any characteristic subspace containing $f(X_\sigma\cap \Omega)$. Since $X'_{f^\sharp_s(\tau)}$ is the intersection of all characteristic subspaces containing $f(X_\tau\cap\Omega)$, we obtain
$$
X'_{f^\sharp_s(\tau)}\subset Y
$$  
for any characteristic subspace $Y$ containing $f(X_\sigma\cap\Omega).$
Since $f^\sharp_r(\sigma)$ is the intersection of all characteristic subspaces containing $f(X_\sigma\cap\Omega)$, we obtain
$$
	X'_{f^\sharp_s(\tau)}\subset X'_{f^\sharp_r(\sigma)},
$$
i.e.,
\begin{equation}\label{Z-tau-1}
	pr'(f^\sharp_r(\sigma))\in Z'_{f^\sharp_s(\tau)},
\end{equation}
which implies
$$
	f^\flat_r(Z_{\tau}\cap { {\rm Dom}}(f^\flat_r))\subset 
	Z'_{f^\sharp_s(\tau)}.
$$

Let $\tau\in \partial \mathcal{D}_s(\Omega)\cap { {\rm Dom}}(f^\sharp_s)$. Choose a sequence $\tau_j, ~j=1,2,\ldots$ in $ \mathcal{D}_s(\Omega)\cap { {\rm Dom}}(f^\sharp_s)$ that converges to
$\tau$. Since $pr(\sigma)\in  Z_{\tau}$ if and only if $pr(\tau)\subset pr(\sigma)$ as subspaces of $V_X$, for any $pr(\sigma)\in Z_{\tau},$ there exists a sequence $pr(\sigma_j)\in Z_{\tau_j},~j=1,2,\ldots,$ that converges to $pr(\sigma)$. By \eqref{Z-tau-1}, we obtain
$$
	pr'(f^\sharp_s(\tau_j))\subset pr'(f^\sharp_r(\sigma_j)).
$$ 
By taking {limits}, we obtain
$$
	pr'(f^\sharp_s(\tau))\subset pr'(f^\sharp_r(\sigma)),
$$
i.e.,
$$ 
	pr'(f^\sharp_r(\sigma))\in Z'_{f^\sharp_s(\tau)}.
$$
{
The same argument can be applied to other cases, which completes the proof.
}
\epf

Similarly, we obtain
\bl\label{Q-mu}
Let $s>r$.  Then $f^\flat_r$ satisfies
$$
	f^\flat_r(Q_{\tau}\cap { {\rm Dom}}(f^\flat_r))\subset Q'_{f^\sharp_{s}(\tau)},\quad \tau\in  \overline{\mathcal{D}_s(\Omega)}\cap { {\rm Dom}}(f^\sharp_s)
$$
and
{
$$
	f^\flat_r(Q_{\tau}\cap { {\rm Dom}}(f^\flat_r))\subset Q'_{f^\sharp_{s,1/2}(\tau)},\quad \tau\in  \overline{\mathcal{D}_{s,1/2}(\Omega)}\cap { {\rm Dom}}(f^\sharp_{s,1/2}).
$$
Similarly, $f^\flat_{r,1/2}$ satisfies 
$$
	f^\flat_{r,1/2}(Q_{\tau}\cap { {\rm Dom}}(f^\flat_{r,1/2}))\subset Q'_{f^\sharp_{s}(\tau)},\quad \tau\in  \overline{\mathcal{D}_s(\Omega)}\cap { {\rm Dom}}(f^\sharp_s)
$$
and
$$
	f^\flat_{r,1/2}(Q_{\tau}\cap { {\rm Dom}}(f^\flat_{r,1/2}))\subset Q'_{f^\sharp_{s,1/2}(\tau)},\quad \tau\in  \overline{\mathcal{D}_{s,1/2}(\Omega)}\cap { {\rm Dom}}(f^\sharp_{s,1/2}).
$$
}
\el

{
\bl\label{b-comp}
Let $\Omega_\rho$ be a general rank $s$ boundary component of $\Omega$
and let $\sigma\in \mathcal D_r(S_s(X))$ be a general point such that $\Omega_\sigma\subset \Omega_\rho$.
Suppose there exists a boundary component $\Omega'_{\mu'}$ of $\Omega'$ such that 
$$\Omega_{f^\sharp_r(\sigma)}'\subset\Omega'_{\mu'}.$$
Then 
for all general $\nu\in \mathcal D_r(S_s(X))$ such that $\Omega_\nu\subset\Omega_\rho,$
$$\Omega_{f^\sharp_r(\nu)}'\subset \Omega'_{\mu'}.$$
As a consequence,
$$f^\flat_r(Q_\rho\cap {\rm Dom}(f^\flat_r))\subset Q'_{\mu'}.$$
\el

\bpf
Let $\Omega_\nu\subset\Omega_\rho.$
Choose a sequence $\{\rho_j\}_j\subset \mathcal D_s(\Omega)\cap{\rm Dom}(f^\sharp_s)$ as in the proof of Lemma~\ref{Z-tau-fix} that converges to $\rho.$ Since $\Omega_\sigma$ and $\Omega_\nu$ are contained in $\Omega_\rho$, we can choose sequences $\{\sigma_j\}_j$ and $\{\nu_j\}_j$ converging to $\sigma$ and $\nu$, respectively such that $\Omega_{\sigma_j}\cup \Omega_{\nu_j} \subset \Omega_{\rho_j}$.
Since $\Omega_{\sigma_j}$ and $\Omega_{\nu_j}$ are contained in the same characteristic subdomain of $\Omega$, we can choose $x_j\in \Omega_{\sigma_j}$ and $y_j\in \Omega_{\nu_j}$ such that Kobayashi distance between $x_j$ and $y_j$ is bounded above by a fixed constant $C$ independently of $j$.  
Since $f$ is holomorphic, Kobayashi distance between $f(x_j)$ and $f(y_j)$ is bounded above by the same constant $C$. Therefore any cluster points of $\{f(x_j)\}$ and $\{f(y_j)\}$ should be contained in the same boundary component. 
Hence by the same argument as in the proof of Lemma~\ref{Z-tau-fix}, $\Omega'_{f^\sharp_r(\sigma)}$ and $\Omega'_{f^\sharp_r(\nu)}$ should be contained in the same boundary component.
\epf
}

%\bl\label{condition-T}
%Let $(N,M)$ be one of $\left(SGr(k, \mathbb{C}^{2n}), Gr(k, \mathbb{C}^m)\right)$, $m\geq 2n$, 
%$\left(OGr(k, \mathbb{C}^{2n}), OGr(k, \mathbb{C}^{2m})\right)$, $\left(SGr(k, \mathbb{C}^{2n}), SGr(k, \mathbb{C}^{2m})\right)$,$m\geq n$. Then $\varpi:\mathscr C( N)\to N$ satisfies Condition (T), i.e.,
%
%\el
%
%***********************
%
%
%\bl
%Let $(N,M)$ be one of $\left(SGr(k, \mathbb{C}^{2n}), Gr(k, \mathbb{C}^m)\right)$, $m\geq 2n$, 
%$\left(OGr(k, \mathbb{C}^{2n}), OGr(k, \mathbb{C}^{2m})\right)$, $\left(SGr(k, \mathbb{C}^{2n}), SGr(k, \mathbb{C}^{2m})\right)$,$m\geq n$.
%
%
%
%\el
%
%
%
%
%
%
%\bt\label{Mok-image}
%Let $(N,M)$ be a pair in Lemma~\ref{rat-sat} and let $S$ be a germ of a complex submanifold of $M$.
%Suppose that 
%$$T_P S=T_P N,\quad \forall P\in S.$$
%Suppose further that 
%$$\mathscr C_P(S)=\mathscr C_P( M)\cap T_PS,\quad\forall P\in S.$$ 
%Then $S=H(N)$ for some $H\in \Aut(M)$.
%\et

\section{Rigidity of the induced moduli map}\label{Rigidity of the induced moduli map}

Let $(\Omega, \Omega')$ be a pair of bounded symmetric domains with rank $q$ and $q'$, respectively that satisfies the conditions in Theorem~\ref{main} or Theorem~\ref{nonexistence}. Suppose {that $X$ and $X'$ are one of the type I and III, $i_r\geq i_{r-1}+2$ for all $r=1,\ldots,q-1$,} where we let $i_0=0$.
%By Lemma \ref{increasing}, this occurs only if $X$ is of type II and $X'$ is of type I or III. 
Since
$$i_{q-1}\leq q'-1<2q-2=2(q-1),$$ 
this is impossible. Hence there exists $r\geq 1$ such that $i_r=i_{r-1}+1$. 
{Similarly, by Lemma \ref{increasing}, we obtain that
if $ X$ and $X'$ are of the type II,
{then $2\leq i_r\leq 2(2q-3)$ and there exists $ r$ such that $i_{r}=i_{r-1,\frac12}+1$ or $i_{r, \frac12}=i_{r}+1$.} 
If $X$ is of the type II, $X'$ is one of the type I and III, then the only possible case is $q'=2[n/2]-1$ and $i_1=1$, $i_{r}=i_{r-1}+2,~r>1.$
In this section, we will show the rigidity of the induced moduli map $f^\flat_r$ for such $r$. More precisely, we will prove the following.

\bl\label{rigid moduli}
There exists $r$ such that $f^\flat_r$ or $\overline{f^\flat_r}$ extends to a trivial embedding. 
\el

The proof of Lemma~\ref{rigid moduli} will be given in several steps. 
Let $r$ be an integer such that
\beq\label{condition in lemma}
i_r=	i_{r-1}+1
\eeq
{whenever $X'$ is of type I or III.}
{If $X$ and $X'$ are both of type II, then we let $r=1$ if $i_1=2$ and 
we let $1<r$ be an integer such that
\beq\nonumber
 i_{r-1,\frac 12}=i_{r-1}+1\quad  {\text or }\quad  i_{r}=i_{r-1,\frac12}+1
\eeq
if $i_1>2$. 

{ From now on we assume that $f^\flat_r$ is holomorphic. The same argument can be applied to the case when $X$ is of type I and $f^\flat_r$ is anti-holomorphic.}
}

\medskip
\noindent{\bf Proof of Lemma~\ref{rigid moduli} when $r=1${:}} In this case 
we obtain 
%$$a_r=\begin{cases}
% 	q'-i_r=q'-1 & \text{ if } X' = Gr(q',p')\text{ or } SGr_{q'}\\
%	2q'-i_r=2q'-1& \text{ if } X' = OGr_{n'}\text{ with }q'=[n'/2]
%\end{cases}.
%$$
%Since $a_r\leq 2q'-2$ if $X'$ is of type II by Table~\ref{table 1}, $X'$ should be of type I or III and 
$$f^\flat_1(D_1(X))\subset pr'(\mathcal D_1(X')).$$
In particular,
$f$ sends minimal discs of $\Omega$ into balls in $\Omega'$. Hence by \cite{M08b}, and \cite{N15a}, $f$ is a totally geodesic isometric embedding and preserves the variety of minimal rational tangents. 
{Let $0\in \Omega$ be a general point. Assume that $f(0)=0$. Since $df$ preserves VMRT, $df_0:T_0(X)\to T_0(X')$ is an embedding that preserves rank one vectors.
For instance, if $X=LGr_n$ and $X'=Gr(q',\mathbb C^{p'+q'})$, then $df_0$ satisfies
$$
	[df_0][S^2 v]=[a\otimes b]
$$
for some $a$ and $b$. 
Consider 
$$
	[df_0][S^2 (v_0+t v_1)]=[a_t\otimes b_t],\quad t\in \mathbb R.
$$
By comparing the coefficient of $t^k$,} {we obtain that either one of $a_t$ and $b_t$ is constant or $a_t=a_0+ta_1$ and $b_t=b_0+tb_1$. In the first case, we obtain 
that $[df_0]$ maps $\mathbb PT_0X$ into $\mathscr C_0(X')$.}  {Since the holomorphic map $f: \Omega \to \Omega'$ is already known be a totally geodesic isometric embedding, it would follow that $S := f(\Omega) \subset \Omega'$ is a Hermitian symmetric subspace of rank-1, which is impossible given that $\Omega$ is not biholomorphic to a complex unit ball.}  Hence the second case holds. }{
Since $v_0$ and $v_1$ are arbitrary, we obtain 
$$ [df_0][S^2 v]=[L_1(v)\otimes L_2(v)]$$
for some linear embeddings $L_1$ and $L_2$. After composing {with} a suitable automorphism of $X'$, we may assume {without loss of generality} 
$$
	[df_0][S^2 v]=[\imath_1(v)\otimes \imath_2 (v)],$$
where $\imath_1:\mathbb C^n\to \mathbb C^{p'}$ and $\imath_2:\mathbb C^n\to \mathbb C^{q'}$ are trivial embeddings. Since $f$ is an isometric embedding and the set of all rank one vectors spans $T_0(X)$, this implies that 
$f: D^{III}_n\to D^I_{p',q'}$ is a trivial embedding. The same argument can be applied to {the} other cases.}

\medskip
\noindent{\bf Proof of Lemma~\ref{rigid moduli} when $2\leq r<q-1${:}}
{In} this case, as subgrassmannians in $D_r(X)$ and $F_{i_r}(X')$, respectively, we have
\beq\label{rank 2}
{\rm rank}~ Z_{\tau}\geq 2,\quad \tau\in \mathcal D_0(X)
\eeq
and
$$
{\rm rank}~Z'_{\tau'}\geq 2, \quad \tau'\in \mathcal D_0(X').
$$
If $X$ and $X'$ are of type II, then as subgrassmannians in $D_{r-1}(X)$ and $F_{i_{r-1}}(X')$, respectively, we have
\beq\nonumber
{\rm rank}~ Z_{\tau}\geq 2,\quad \tau\in \mathcal D_0(X)
\eeq
and
$$
{\rm rank}~Z'_{\tau'}\geq 2, \quad \tau'\in \mathcal D_0(X').
$$
Therefore the following two lemmas and Lemma~\ref{H respects} will complete the proof.

\bl\label{lines}
{
If $X$ is of type I or type III, then $f^\flat_r\colon {\rm Dom}(f^\flat_r)\subset D_r(X)\to F_{i_r}(X')$
respects subgrassmannian distributions.
If $X$ is of type II, then
 $f^\flat_r\colon {\rm Dom}(f^\flat_r)\subset D_r(X)\to F_{i_r}(X')$
or $f^\flat_{r-1}\colon {\rm Dom}(f^\flat_{r-1})\subset D_r(X)\to F_{i_{r-1}}(X') $
respects subgrassmannian distributions.
}
\el

\bpf
Suppose that $X$ is of the type I or III and $i_r=i_{r-1}+1$. Then by Table~\ref{table 2} and Lemma~\ref{Z-tau}, we can show that $f^\flat_r$ maps all rank one vectors in $T Z_{\tau}$, $\tau\in \mathcal D_0(X)$ into rank one vectors in $T Z'_{f^\sharp_0(\tau)}$.
Then by Mok's result (\cite{M08b}) and \eqref{rank 2}, we obtain that either $f^\flat_r$ restricted to each general maximal subgrassmannians in $D_r(X)$ is a standard embedding or the {image} of $f^\flat_r$ is contained in a fixed rank one subspace in $F_{i_r}(X')$. But since $f$ is proper, the latter case does not happen.

{
Suppose that $X$ and $X'$ are of the type II. Note that in this case, $f^\sharp_r=f^\flat_r$. Suppose that $i_r=i_{r-1,\frac12} +1$. Then by the similar argument above we can show that $f_r^\flat\colon {\rm Dom}(f_r^\flat)\subset D_r(X)\to F_{i_r}(X')$ respects subgrassmannian distribution{}.
Now suppose $i_{r-1,\frac12}=i_{r-1}+1$.
Then by the similar argument, we can show that $f^\flat_{r-1,\frac12}$ respects subgrassmannian distributions.
Let $\tau\in D_0(\Omega)$ so that $Z_\tau\subset D_{r-1}(\Omega)$. Then it is enough to show that $f^\sharp_{r-1}$ is a standard map on $Z_\tau$
for all $\tau\in D_0(\Omega)$.
%Since $f^\flat_{r-1,\frac12}$ respects subgrassmannian distribution, we may assume that $f^\flat_{r-1, \frac12}$ restricted to $Z_\tau^{r-1,\frac12}$ is a standard embedding.
Let
$$Z_\tau^{r-1}=Gr(a, V).$$
Then 
$$Z_\tau^{r-1,\frac12}=Gr(a-1, V)$$
and by assumption, $f^\sharp_{r-1,\frac12}:Gr(a-1, V)\to Gr(b, V')$ is a standard embedding for some $Gr(b, V')=Z_{\tau'}'$.
For a fixed $\xi\in Gr(a-1, V)$, consider a rank one subspace
$$L_\xi:=\{[\xi\oplus W]\in Gr(a, V): W\in Gr(1, V), W \not\subset \xi\}.$$
Then for each $[\xi\oplus W]\in L_\xi$, there exists $\sigma_W\in  Z_\tau^{r-1,\frac12}$ such that
$$X_{[\xi\oplus W]}=X_{\xi}\cap X_{\sigma_W},$$
where $X_{[\xi\oplus W]}$ is the rank $r-1$ characteristic subspace corresponding to $[\xi\oplus W]$ and 
$X_{\xi}$ and $X_{\sigma_W}$ are totally invariantly geodesic subspaces corresponding to $\xi$ and $\sigma_W$, respectively.
By the definition of $f^\sharp_{r-1}$, for $\eta_W=[\xi\oplus W]\in L_\xi,$ we have 
\begin{equation}\nonumber
X'_{f_{r-1}^\sharp(\eta_W)}  =\bigcap_{\eta'}X'_{\eta'}
\subset  X'_{f_{r-1, \frac12}^\sharp (\xi)}\cap X'_{f_{r-1,\frac12}^\sharp(\sigma_W)}
\end{equation}
where the first intersection is taken over all characteristic subspaces $X'_{\eta'}$ containing $f(\Omega \cap X_{\eta_W})$.
Since $i_{r-1, \frac12}=i_{r-1}+1$, this inclusion implies 
$$X'_{f_{r-1}^\sharp(\eta_W)}  =X'_{f_{r-1, \frac12}^\sharp (\sigma_0)}\cap X'_{f_{r-1,\frac12}^\sharp(\sigma_W)}=X'_{f^\sharp_{r-1,\frac12}(\xi)+ f^\sharp_{r-1,\frac12}(\sigma_W)}$$
and $f^\sharp_{r-1, \frac12}(\xi)$ is a codimension one subspace of $f^\sharp_{r-1,\frac12}(\xi)+ f^\sharp_{r-1,\frac12}(\sigma_W)$. 
Here $f^\sharp_{r-1,\frac12}(\xi)+ f^\sharp_{r-1,\frac12}(\sigma_W)$ is the smallest subspace in $V'$ that contains $f^\sharp_{r-1,\frac12}(\xi)\cup f^\sharp_{r-1,\frac12}(\sigma_W)$.
Moreover since $f^\sharp_{r-1,\frac12}$ is a standard embedding,
we obtain that on $L_\xi$, $f^\sharp_{r-1,\frac12}(\xi) + f^\sharp_{r-1,\frac12}(\sigma_W)$ is either constant or of the form
$$f^\sharp_{r-1}(\xi)\oplus \phi(W)$$
for some projective linear embedding $\phi:Gr(1, V)\to Gr(1, V')$. In the first case, since $\tau$ and $\xi$ are arbitrary, $f^\sharp_{r-1}$ is constant on $D_{r-1}(\Omega)$, which is impossible. Therefore the second case holds and 
$$\left\{f^\sharp_{r-1,\frac12}(\xi)+ f^\sharp_{r-1,\frac12}(\sigma_W):W\in Gr(1, V), W\not\subset\xi\right\}\neq 
\left\{f^\sharp_{r-1,\frac12}(\widetilde \xi)+ f^\sharp_{r-1,\frac12}(\sigma_W):W\in Gr(1, V), W\not\subset\widetilde\xi\right\}$$
if $\xi\neq \widetilde\xi.$
Since $\xi$ is arbitrary, $f^\sharp_{r-1}$ restricted to $Z_\tau^{r-1}$ is a standard embedding by \cite{M08a}.}
\epf

\rm
{
We may assume that $f(\Omega)$ is not contained in any proper totally invariantly geodesic subspace of $\Omega'$. Let $V\in X (=\mathcal D_0(X))$. Let $Z_V=Gr(a_r, V)$. Since $f^\flat_r$ respects {subgrassmannians}, there exists subspaces $W_0, W_1$
such that on a big {Schubert} cell, $f^\flat_r$ is given by
\beq\label{map-1}
(x)\in Z_V\to W_0\oplus (x)\in W_0\oplus Gr(a_r, W_1)
\eeq 
or
\beq\label{map-2}
(x)\in Z_V \to W_0\oplus(x^t)\in W_0\oplus Gr(b_r, W_1),
\eeq
where $b_r=r$ if $X$ is one of the type I and III, $b_r=n-2[n/2]+2r$ if $X=OGr_n$.
Suppose \eqref{map-2} holds. Since $D_r(X)$ is $Z_{\tau}$-connected with $\tau\in \mathcal{D}_0(X)$, as in the proof of Lemma~\ref{H respects}, there exist subspaces $W_0', W_1', W_2'\subset V_{X'}$ independently of $\tau\in \mathcal{D}_0(X)$ with $\dim W_1'=b_r>0$ such that for $\tau\in \mathcal D_0(X)$,
$$ 
	f^\flat_r(Z_{\tau})\subset W_0'\oplus Gr(c_r, W_1'\oplus W_2'),\quad c_r=\dim W_2' .
$$
On the other hand, since $f^\flat_r$ maps $Z_V$ to $Z'_{f(V)}=Gr(a_{i_r}, f(V))$ for $V\in \Omega$, in view of \eqref{map-2}, we obtain
$$ W_1'\subset f(V),\quad \forall V\in \Omega.$$
Therefore $f(\Omega)$ is contained in a totally invariantly geodesic subspace of $\Omega'$, which is a contradiction. Hence $f^\flat_r$ on $Z_\tau$ is of the form \eqref{map-1} and there {exists a subspace} $W_2$ such that
$$ f^\flat_r(D_r(X))\subset W_0\oplus Gr(a_r, W_2),$$
where $W_0$ is given in \eqref{map-1}. Since
$$ f^\flat_r(D_r(\Omega))\subset F_{i_r}(\Omega'),$$
we obtain 
$$I_{p',q'}\big|_{W_0}>0.$$
Write
$$f^\flat_r=W_0\oplus H.$$
Choose $I_{p',q'}$-isotropic subspace $\widetilde W_0$ such that $\dim \widetilde W_0=\dim W_0$ and $I_{p',q'}(\widetilde W_0,W_2)=0.$
Then, we obtain the following lemma.

\bl\label{H preserves}
$H$ satisfies
\beq\label{Sigma-pres}
	\widetilde W_0\oplus H(\Sigma_r(X))\subset \Sigma_{i_r}(X'),
\eeq
\beq\label{Sigma_perpendicular}
\widetilde W_0\oplus  H(D_{r}(X))\not\subset \Sigma_{r'}(X').
\eeq
\el

\bpf
By Lemma~\ref{Z-tau-fix}, Lemma~\ref{commute}, there exists $m$ 
%with $m> i_r$ if $X'$ is of type~I or III and $m>[i_r/2]$ if $X'$ is of type~II 
such that 
\beq\label{lemma Z-tau-fix}
	f^\flat_r(\Sigma_r(X))\subset F_{i_r}(S_m(X')).
\eeq
Since $I_{p',q'}\big|_{W_0}>0$, to show \eqref{Sigma-pres}, it is enough to show that $m\leq i_r-\dim W_0.$
Suppose that \eqref{Sigma-pres} does not hold. Then $m>i_r-\dim W_0$.
Let $V_0\in \Sigma_r(X)$ be a general point. Choose $\sigma_0\in \mathcal D_r(S_r(X))$ such that $V_0=pr(\sigma_0)$. By \eqref{lemma Z-tau-fix}, there exists a unique boundary component 
$\Omega_{\mu_0'}'$ of $\Omega'$ with rank $m$ such that $\Omega'_{f^\sharp_r(\sigma_0)}\subset \Omega'_{\mu_0'}$. 
Since $m>i_r-\dim W_0$, 
$pr'({\mu_0'})$ is a proper subspace of $H(V_0)$. 
Since $\Omega_{f^\sharp_r(\sigma_0)}'$ is contained in a unique boundary component, $pr'(\mu_0)$ is the unique maximal $I_{p',q'}$-isotropic subspace of $H(V_0)$.
In what follows, we will show that 
$$f^\flat_r(D_r(\Omega))\subset Q_{\mu_0'}',$$
which is a contradiction to the assumption that $f$ is proper.
 
Choose a general $\tau\in \mathcal  S_0(\Omega)$ such that $V_0 (=pr(\sigma_0))\in  Z_\tau\subset \Sigma_r(X)$.
Write \beq\nonumber
Z_{\tau}=Gr(n_r, V_{\tau})
\eeq
for suitable $V_{\tau}\subset V_X$. Since $f^\flat_r$ respects {subgrassmannian} {distributions} and $f^\flat_r$ restricted to $Z_\tau$ satisfies \eqref{map-1},
we obtain 
$$f^\flat_r(Z_\tau)= W_0\oplus Gr(n_r, L_\tau)$$
for some $L_\tau$.
Then
there exists a unique subspace $R\subsetneq V_0$ such that
$$f^\flat_r(\{V\in Z_\tau: V\supset R \})=\{V'\in f^\flat_r(Z_\tau): V'\supset pr'(\mu_0')\}.$$ 
%Since $I_{p',q'}\big|_{V_0}>0$, \eqref{lemma Z-tau-fix} implies that $A_\tau\oplus L_\tau$ contains a nontrivial isotropic space of $I_{p',q'}$. 
%Suppose $N_{\tau}\subsetneqq L_{\tau}$ for some $\tau$. Since the dimension of $N_{\tau}$ depends continuously on $\tau$, we obtain
%$$N_\tau\subsetneqq L_\tau$$
%for general $\tau\in \mathcal D_0(S_0(X))$.
%In fact, \eqref{lemma Z-tau-fix} implies that maximal null spaces of $I_{p',q'}$ in $L_\tau$ for $\tau\in \mathcal D_0(S_0(X))$ with $ Z_{\tau}\cap Dom(f^\flat_r)\neq\emptyset$ are of dimension $(q'-m)$ if $X'$ is of type~I or III and of dimension $(2q'-2m)$ if $X'$ is of type~II.
%Fix $\tau\in \mathcal D_0(S_0(X))$ and a point 
%since $I_{p',q'}\big|_{V_0}>0$, we may assume $N\subset L_{\tau_0}$. Since $N$ is a maximal null space in $L_{\tau_0}$, 
%for any $\sigma'\in \mathcal D_{i_r}(S_m(X'))$ with $N_0\subset pr'(\sigma')$ as subspaces of $V_{X'}$, we obtain 
%$\Omega_{\sigma'}'\subset \Omega_0'$ and vice versa.
%%Since $\Omega_{\sigma_0'}'\subset \Omega'_0$, we obtain $N\subset f^\flat_r(W_0)$ as subspaces in $V_{X'}$ and
%Since $f^\flat_r$ restricted to each $Z_{\tau}$ is a biholomorphism onto $V_0\oplus Gr(a_r, L_\tau)$, 
%there exists a unique maximal subspace $M_\tau\subset pr(\sigma_0)$ such that
%$$f^\flat_r(W)\supset N_0$$ 
%for all $W\in Z_{\tau}$ with $M_\tau\subset W$, where $V_\tau$ is given in \eqref{L__tau}.
Since $R$ is a subspace of $V_0$, we obtain
$$I_{p,q}(R, R)=0.$$
Hence there exists a unique boundary component $\Omega_{\rho}=X_{\rho}\cap \partial\Omega$ of rank $s>r$ such that
$pr( \rho)=R$ and $\partial \Omega_{\rho}\supset \Omega_{\sigma_0}$. 

Consider
$$Q_{\rho}=\{pr(\sigma)\in D_r(X): X_\sigma\subset X_{\rho}\}.$$
By definition, we obtain
$$H(V)\supset pr'(\mu_0'),\quad V\in Q_{\rho}\cap Z_\tau.$$
Since $Z_\tau$ is of rank $\geq 2$ and 
$$R\subsetneq V_0\subsetneq V_\tau,$$ 
$ Q_{\rho}\cap  Z_\tau$ contains a rank one subspace of dimension at least $2$.
Since $f^\flat_r$ on each $Z_\tau$ satisfies \eqref{map-1}, 
we obtain
$$ f^\flat_r(\{V\in D_r(X): V\supset R\})\subset \{V'\in F_{i_r}(X): V'\supset pr'(\mu_0')\},$$ 
i.e.,
\beq\label{intersect}
f^\flat_r( Q_{\rho})\subset Q'_{\mu_0'}
\eeq
Choose a general $\sigma\in \mathcal Q_\rho$ such that $\Omega_\sigma\subset \Omega_\rho$. Then $\Omega'_{f^\sharp_r(\sigma)}$ is contained in a rank $m'\geq m$ boundary component of $\Omega'$. By \eqref{intersect}, we obtain that $m'=m$.
Since $\Omega_\rho$ and $\Omega'_{\mu_0'}$ are rank $s$ and rank $m$ boundary components of $\Omega$ and $\Omega'$, respectively, by Lemma~\ref{Z-tau-fix}, we obtain
$$ f^\sharp_r(\mathcal D_r(S_s(X)\cap{\rm Dom}(f^\flat_r))\subset \mathcal F_{i_r}(S_m(X')).$$

Let
$$ A:=(f^\flat_r)^{-1}(Q'_{\mu_0'})\cap D_r(S_s(X)).$$
Then $A$ is a nonempty set containing $\{pr(\nu)\in  Q_\rho:\Omega_\nu\subset \Omega_\rho\}$.
Let $\nu\in A$ be a general point. Then by definition
$$\Omega_{f^\sharp_r(\nu)}'\subset \Omega_{\mu_0'}'.$$
Choose a rank $s$ boundary component $\Omega_{\widetilde\rho}$ of $\Omega$ such that $\Omega_\nu\subset{\Omega_{\widetilde\rho}}$ and choose a general $\widetilde \sigma$
such that $\Omega_{\widetilde\sigma}$ is a rank $r$ boundary component of $\Omega_{\widetilde\rho}$. Then by Lemma~\ref{b-comp}, 
we obtain
$$\Omega_{f^\sharp_r(\widetilde\sigma)}\subset \overline{\Omega'_{\mu_0'}}.$$
On the other hand,
by \eqref{lemma Z-tau-fix}, 
$\Omega'_{f^\sharp_r(\widetilde\sigma)}$ should be contained in a rank $m$ boundary component of $\Omega'$.
Since $\Omega'_{\mu_0'}$ is a rank $m$ boundary component of $\Omega'$, we obtain
$$\Omega'_{f^\sharp_r(\widetilde\sigma)}\subset \Omega'_{\mu_0'}.$$
Since $\Omega_{\widetilde\sigma}$ is a boundary component of $\Omega_{\widetilde\rho}$, by the same argument as above, we obtain
$$ f^\flat_r(Q_{\widetilde\rho})\subset Q'_{\mu_0'}.$$ 
Since any two points $\sigma_1, \sigma_2 \in \Sigma_r$ are connected by $Q_{\widetilde\rho}$-chain for $\widetilde\rho\in \mathcal D_s(S_s(X)).$
we obtain
$$f^\flat_r(\Sigma_r\cap{\rm Dom}(f^\flat_r))\subset Q'_{\mu_0'}.$$
Since $\Sigma_r(X)$ is a Levi nondegenerate generic CR manifold, we obtain
$$f^\flat_r(D_r(X)\cap {\rm Dom}(H))\subset Q_{\mu_0'}'.$$

Next suppose \eqref{Sigma_perpendicular} does not hold. Then there exists $m$ such that $f^\flat_r(D_r(X))\subset D_{i_r}(S_m(X')).$
Hence we obtain $f^\sharp_r(\mathcal D_r(\Omega))\subset \mathcal D_{i_r}(S_m(X'))$, which contradicts the assumption that $f$ is proper.
\epf
}

\noindent{\bf Proof of Lemma~\ref{rigid moduli} when $ r=q-1${:}} Assume that $X'$ is of type I or III. If $ i_1=1$, then $r=1$ satisfies the condition \eqref{condition in lemma}. {By the proof of Lemma~\ref{rigid moduli} in the case of $r = 1$, {then $f$ is a standard embedding.} We may therefore assume without loss of generality that $i_1>1$.}  If $i_{q-1}<q'-1$, then since $1<i_1$ and $i_{q-1}<q'-1\leq 2q-3$, $i_{q-2} = i_{q-1}-1 < 2q-4$, hence} there must necessarily exist another {$r$ satisfying $2\le r <q-1$} such that $i_r=i_{r-1}+1$, {which has already been taken care of in the above.}

{Without loss of generality we may therefore assume that $i_{q-1}=q'-1$, in which case $i_{q-1}<2(q-1)$ and hence $X$ cannot be of type II. Therefore $X$ is of type I or III and $i_{q-1}=i_{q-2}+1$, which implies that $f^\flat_{q-1}$ maps $Z_{\tau},~\tau\in \mathcal{D}_{q-2}(S_{q-2}(X))$ to $Z'_{\tau'},~\tau'\in \mathcal{D}_{q'-2}(X')$. By Lemma~\ref{Z-tau-rank one}, $Z_{\tau}$, $\tau\in \mathcal{D}_{q-2}(S_{q-2}(X))$ and  $Z'_{\tau'},~\tau'\in \mathcal{D}_{q'-2}(X')$ are projective lines in $\Sigma_{q-1}(X)$ and $D_{q'-1}(X')$, respectively. Hence $f^\flat_{q-1}$ sends projective lines in $\Sigma_{q-1}(X)$ to projective lines in 
$D_{{q'-1}}(X')$.}
{
Note that $ f^\flat_{q-1}$ maps $\Sigma_{q-1}$ to $\Sigma'_{q'-1}$. Since $\Sigma_{q-1}$ and $\Sigma'_{q'-1}$ are Levi nondegenerate CR hyperquadrics and $f^\flat_{q-1}(D_{q-1}(X))$ is not contained in $\Sigma_{q'-1}'$, $f^\flat_{q-1}$ restricted to $\Sigma_{q-1}$ is a transversal CR map at a general point. In particular, $f^\flat_{q-1}$ is of maximal rank at a general point.}  
{Therefore Lemma~\ref{hyperquadric} completes the proof.

Assume now that $X'$ is of type II. Since the pair $(X, X')$ satisfies the hypothesis in Theorem~\ref{main} or Theorem~\ref{nonexistence}, $X$ {{must necessarily be} of type II. 
%and $i_r\leq 2q'-2$. 
%Suppose $f^\sharp_1(\mathcal D_1(S_  $2\leq i_1\leq i_{q-1}\leq 2q'-2$ and $X$ should be of type II . 
{
Therefore $Z_\tau$ and $Z'_{\tau'}$ are of rank greater or equal to $2$.
Therefore by the same argument as in the case of $r<q-1$, we can show that $f_{q-2}^\flat$ is a trivial embedding if $i_{q-2,\frac12}=i_{q-2}+1$ and $f^\flat_{q-1}$ is a trivial embedding if $i_{q-1}=i_{q-2,\frac12}+1$.}
\medskip

By Lemma~\ref{rigid moduli}, we can choose $r>1$ such that $f^\flat_r$ is a trivial holomorphic embedding. Moreover, if $r<q-1$, then there exists a natural embedding of $\imath:V_X\to V_{X'}$ given by $f^\flat_r$ such that 
$$
	f^\flat_r(D_r(X))\subset V_0\oplus Gr(a_r, \imath(V_X))
$$ 
and $f^\flat_r=V_0\oplus S_r$, where $a_r=q-r$ if $X$ is of type I or III and $a_r=2(q-r)$ if $X$ is of type II and $S_r\colon D_r(X)\to Gr(a_r, \imath(V_X))$ is a trivial embedding. We will identify $V_X$ with $\imath(V_X)$ and regard $V_X$ as a subspace of $V_{X'}$.

%From now on, we will regard $V_X$ as a subspace of $V_{X'}$.

\bl\label{standard embedding}
$i_{q-1}=i_{q-2}+1$ and there exists $V_0\subset V_{X'}$ such that 
%If $X$ is of type~I or III, then
%$i_{q-1}=i_{q-2}+1$
%and 
$$f^\flat_{q-1}=V_0\oplus S_{q-1}:D_{q-1}(D)\to V_0\oplus Gr(1, V_0^\perp)$$
if $X$ is of type I or III
and
$$f^\flat_{q-1}=V_0\oplus S_{q-1}:D_{q-1}(D)\to V_0\oplus Gr(2, V_0^\perp)$$
if $X$ is of type II.
%If $X$ is of type~II, then
%$i_{q-1}=i_{q-2}+2$
%and
%$$f^\flat_{q-1}=V_0\oplus I_{2q-2}.$$
\el

\bpf
First we assume that $X$ is of type~I or III. Then by assumption on the pair $(X, X')$ in Theorem~\ref{main} or Theorem~\ref{nonexistence}, $X'$ is of type I or III, too. If $i_{q-1}=i_{q-2}+1=q'-1,$ then it is clear. Suppose $i_{q-1}>i_{q-2}+1$ or $i_{q-1}<q'-1$. Then we can choose $r<q-1$ such that $i_{r}=i_{r-1}+1$.
Hence it is enough to show that if $r<q-1$ and $i_r=i_{r-1}+1$, then
$i_{r+1}=i_r+1$ and $f^\flat_{r+1}=V_0\oplus S_{r+1}.$
Let $\mu\in \mathcal{D}_{r+1}(\Omega)$ be a general point. 
Let $V_\mu$ be a subspace of $V_X$ of dimension $q-r-1$ such that
$$
	Q_{\mu}=\{V\in D_r(X)\colon V_\mu\subset V\}.
$$
Since $f^\flat_r$ preserves $Q_\mu$, we obtain
$$
	f^\flat_{r}(Q_\mu)\subset Q'_{f^\sharp_{r+1}(\mu)}.
$$
Let $L_\mu\subset L_{X}$ be a minimal subspace such that
$$ Q'_{f^\sharp_{r+1}(\mu)}\cap f^\flat_r(D_r(X))=V_0\oplus \{V'\in Gr(a_r, V_X): L_\mu\subset V'\}.$$
Since $S_r$ is a standard embedding, we obtain $\dim L_\mu=\dim V_\mu$. 
We will show that 
$$ pr'(f^\sharp_{r+1}(\mu))=V_0\oplus L_\mu,$$
which will imply
$$i_{r+1}=q'-\dim V_0+r+1=i_r+1$$
and
$$f^\flat_r=V_0\oplus S_{r+1}.$$

By assumption on $f^\flat_r$ and Lemma~\ref{Q-mu}, we obtain
$$
	f^\flat_r(Q_{\mu})= V_0\oplus
	\{ V\in Gr(a_r, V_X)\colon L_\mu\subset V\}
	\subset Q'_{f^\sharp_{r+1}(\mu)}.
$$
Since by definition
$$Q'_{f^\sharp_{r+1}(\mu)}=\{V'\subset V_{X'}:pr'(f^\sharp_{r+1}(\mu))\subset V'\},$$
we obtain
$$
	pr'(f^\sharp_{r+1}(\mu))\subset V_0\oplus L_\mu
$$
as a subspace.
On the other hand, for any $\sigma\in \mathcal{D}_r(\Omega)$ with $pr(\sigma)\in Q_{\mu}$, we obtain
$$
	pr'\circ f^\sharp_r(\sigma)
	=f^\flat_r\circ pr(\sigma)
	\in f^\flat_r(Q_{\mu})
	=	V_0\oplus \{ V\in Gr(a_r, V_X)\colon L_\mu\subset V\}, 
$$
which implies
$$
	f^\sharp_r(\sigma)\in \{(V_0\oplus V_1, V_2)\in \mathcal{F}_{(a'_r, b'_r)}(\Omega')\colon L_\mu\subset V_1\}.
$$
Since 
$$
	f(\Omega_\mu)\subset \bigcup_{\sigma\in \mathcal{Q}_{\mu}} f(\Omega_\sigma),
$$
we obtain
$$
	f(\Omega_\mu)\subset X'_{(V_0\oplus L_\mu, W)}
$$
for some $W\subset V_{X'}.$ Since $f^\sharp_{r+1}(\mu)$ is the smallest Hermitian symmetric subspace that contains $f(\Omega_\mu)$, we obtain
$$
	V_0\oplus L_\mu \subset pr'(f^\sharp_{r+1}(\mu))
$$
completing the proof.
The same argument can be applied to the type~II case.
\epf

%\noindent{\it Proof of Lemma~\ref{rigid moduli} when $X$ is of the type~II}: When X is of the type~II, the similar argument
%can be applied and hence we will omit the proof.

\section{Proof of Theorems}\label{Proof of Theorems}

\subsection{Proof of Theorem~\ref{main}}
By Lemma~\ref{standard embedding} we obtain $f^\flat_{q-1}=V_0\oplus S_{q-1}\colon D_{q-1}(X)\to F_{i_{q-1}}(X')$ is a trivial embedding. Then we obtain 
$$f=V_0\oplus \hat f\colon \Omega\to V_0\oplus \Omega''$$
for some subdomain $\Omega''$ of $\Omega'$ with rank $\leq q'$. By replacing $f:\Omega\to \Omega'$ with $\hat f:\Omega\to \Omega''$, we may assume that  
$f^\flat_{q-1}\colon D_{q-1}(X)\to Gr(1, V_{X'})\subset F_{i_{q-1}}(X')$ is a trivial embedding
if $X$ is of type~I or III and  
$f^\flat_{q-1}\colon D_{q-1}(X)\to Gr(2, V_{X'})\subset F_{i_{q-1}}(X')$ is a trivial embedding
if $X$ is of type~II. Let $j:V_X\to V_{X'}$ be a linear embedding induced by $f^\flat_{q-1}$. Then $j$ defines a standard holomorphic embedding $g\colon X\to X'$ such that
$g^\flat_{q-1}=f^\flat_{q-1}$.
%We will show that $g$ restricted to $\Omega$ becomes a proper holomorphic map into $\Omega'$. 

\bl\label{decomp-f}
Let $g\colon X \to X'$ be the standard holomorphic embedding induced by $j:V_X\to V_{X'}$ and $Y\subset X'$ be the maximal Hermitian symmetric subspace such that  
$g(X)\times Y$ is a totally geodesic subspace of $X'$. Then there exists a holomorphic mapping $h\colon \Omega\to Y$ such that
$$
 	f=g\times h\colon \Omega\to g(\Omega)\times Y.
$$
\el

\bpf
Assume that $f(0)=g(0)$. Assume further that $\Omega$ and $\Omega'$ satisfy the condition 2), i.e., $\Omega$ is of type~III and $\Omega'$ is of type~I. 
Since {$f^\flat_{q-1}=g^\flat_{q-1}$ is induced by a standard holomorphic embedding, by Lemma~\ref{Z-tau-fix}}, we obtain
$$f^\flat_{q-1}(\Sigma_{q-1}(X))\subset \Sigma_{q'-1}(X').$$
%For a maximal characteristic subdomain $\Omega_\sigma\subset \Omega$, choose an extremal disc $\Delta_\sigma\subset \Omega$ such that $\Delta_\sigma\times \Omega_\sigma$ is a totally geodesic subspace of $\Omega$ and hence $\partial\Delta_\sigma\times \Omega_\sigma\subset S_{q-1}$. Let 
%$$\Omega_{\sigma(t)}:=\{t\}\times \Omega_\sigma,~t\in \overline\Delta_\sigma.$$ 
%Since
%$f$ is proper, by Lemma~\ref{Z-tau-fix}, there exists an integer $m\geq i_{q-1}$ independent of $t$ such that 
%$$f^\flat_{q-1}\cdot pr([\sigma(t)])\in D_{i_{q-1}}(S_{m}(X')),\quad \forall t\in \partial\Delta_\sigma.$$ 
%Since $X'$ is of type~I, we obtain $Gr(1, V_{X'})=D_{q'-1}(X')$ and hence $i_{q-1}=q'-1$.
%, which implies $i_{q-1}=q'-1$ 
%%we obtain $m=q'-1$, i.e.,
%%$$
%%	\tilde g_{q-1}(\Sigma_{q-1}(X))=f^\flat_{q-1}(\Sigma_{q-1}(X))\subset \Sigma_{q'-1}(X')
%%$$
%and hence 
%$$
%	\tilde g_{q-1}(\Sigma_{q-1}(X))=f^\flat_{q-1}(\Sigma_{q-1}(X))\subset \Sigma_{q'-1}(X').	
%$$
Moreover, since $pr'\colon \mathcal{D}_{q'-1}(S_{q'-1}(X'))\to \Sigma_{q'-1}(X')$ is one to one, for each $\sigma\in \mathcal D_{q-1}(S_{q-1}(X))$, there exists a unique maximal boundary component 
$M_{\sigma}$ of $\Omega'$ such that
$$
	g(\Omega_{\sigma})\subset \Omega'_{g^\sharp_{q-1}(\sigma)}\subset M_{\sigma}.
$$
Note that since $f^\flat_{q-1}=g^\flat_{q-1}$ and $M_\sigma$ is a maximal boundary component, we obtain
\beq\label{contain}
 	\Omega'_{f^\sharp_{q-1}(\sigma)}\subset M_{\sigma}.
\eeq

For a maximal characteristic subdomain $\Omega_\sigma\subset \Omega$, choose a minimal disc $\Delta_\sigma\subset \Omega$ passing through $0$ such that $\Delta_\sigma\times \Omega_\sigma$ is a totally geodesic subspace of $\Omega$ and hence $\partial\Delta_\sigma\times \Omega_\sigma\subset S_{q-1}(\Omega)$. Let 
$$\Omega_{\sigma(t)}:=\{t\}\times \Omega_\sigma,~t\in \overline\Delta_\sigma.$$
Since $g:X\to X'$ is a standard embedding and 
$$
	g\left(\Omega_{\sigma(t)}\right) \subset M_{\sigma(t)},\quad \forall t\in \partial\Delta_\sigma,
$$
there exists a minimal disc $\Delta'_\sigma$ of $\Omega'$ such that
\beq\label{union of image}
	g\left(\Omega_{\sigma(t)}\right) \subset \Delta'_\sigma\times g(\Omega_\sigma)\subset \Delta'_\sigma\times M_{\sigma},\quad \forall t\in \Delta_\sigma,
\eeq
Since $f^\flat_{q-1}=g^\flat_{q-1}$, by \eqref{contain} and \eqref{union of image}, we obtain
\beq\label{f-contained}
	f\left(\Omega_{\sigma(t)}\right)\subset \Omega'_{f^\sharp_{q-1}([\sigma(t)])}
	\subset \Delta'_\sigma\times M_{\sigma},\quad \forall t\in \Delta_\sigma.
\eeq
Define 
$$
	Z:=\bigcap_{\sigma}(\Delta'_{\sigma})^\perp,
$$ 
where the intersection is taken over all minimal disc $\Delta_\sigma$ passing through $0$, $\Delta'_{\sigma}$ is the minimal disc given in \eqref{union of image} and $(\Delta'_{\sigma})^\perp$ is the maximal characteristic subspace passing through $f(0)$ such that $T_{f(0)}(\Delta'_\sigma)^\perp= \mathcal N_{[v]}$, $v\in T_0\Delta'_\sigma$.
Then by \eqref{union of image}, $Z$ is a maximal Hermitian symmetric space such that $g(X)\times Z$ is totally geodesic in $X'$. We let $Y=Z$.

Choose the minimal Hermitian symmetric subspace $X_{(V_1, V_2)}'\subset X'$ of rank $q$ such that $g(X)\subset X_{(V_1, V_2)}'$.
Considering $0$ as a subspace,
decompose $0$ into $V_1\oplus W_1$.
Choose a local coordinate system of $X'$ at $f(0)$ such that $f=(F_1, F_2)$ satisfies 
$$
	F_1\colon \Omega\to X_{(V_1, V_{X'})}',
	\quad
	F_2\colon \Omega\to X_{(W_1, V_{X'})}'.
$$
By \eqref{f-contained} and induction on dimension, we can show that for any properly embedded maximal polydisk $\Delta^q\subset \Omega$, 
there exist a $q$-dimensional polydisk $\widetilde \Delta^q\subset X'_{(V_1, V_{X'})}$ and a subdomain $\Omega''\subset \Omega'\cap X'_{(W_1, V_{X'})}$ of rank $q'-q$ orthogonal to $\widetilde \Delta^q$  such that $\widetilde \Delta^q\times \Omega''$ is totally geodesic and 
\beq\nonumber
	f(\Delta^q)\subset \widetilde\Delta^q\times \Omega'',
\eeq
which implies that on $\Delta^q\subset \Omega$,
$$ 
	\langle F_1, F_2\rangle_{p',q'}\equiv 0.
$$ 
By differentiating it, we obtain 
$$
	\langle\partial F_1, F_2\rangle_{p',q'}\equiv 0
$$
on $\Delta^q$.
Since $\Delta^q$ is arbitrary, we obtain 
\beq\label{F2-orthogonal}
	\langle \partial F_1,F_2\rangle_{p',q'}\equiv 0.
\eeq

On the other hand, since $f$ is proper, by \eqref{f-contained}, we obtain
$$
	\lim_{x\in \Delta^q\to p\in \partial(\Delta^q)}f(x)
	\subset\partial(\widetilde\Delta^q)\times\Omega''\subset \partial\Omega'.
$$
In particular, $F_1\colon \Omega\to X_{(V_1, V_{X'})}'\cap\Omega'$ is proper. 
Then by \cite{Ts93}, $F_1$ is a totally geodesic isometric embedding. 
Since 
$ f^\flat_{q-1}=g^\flat_{q-1}$, we obtain $\partial F_1=\partial g$. Hence by complexifying \eqref{F2-orthogonal}, we obtain that
$ F_2(\Omega)$ is contained in a subdomain of $\Omega'$ orthogonal to $g(\Omega)$, i.e.,
$f(\Omega)\subset g(\Omega)\times Y$
and  
$$F_1\equiv g.$$
The same argument can be applied to the case when $\Omega$ and $\Omega'$ satisfy the condition (1).

We have proven that writing $F=F_1\times F_2\colon \Omega\to \Omega_1'\times \Omega_2'$, $F_1: \Omega \to \Omega'$ is a standard embedding, and it follows that $F: \Omega \to \Omega_1' \times \Omega_2'$ is a holomorphic totally geodesic isometric embedding with respect to Kobayashi metrics.  By Mok (\cite[Theorem 3.1]{M22}), the holomorphic embedding $\imath: \Omega_1' \times \Omega_2' \to \Omega'$ is a holomorphic totally geodesic isometric embedding with respect to Kobayashi metrics.  It follows that $f: \Omega \to \Omega'$ is also a holomorphic totally geodesic isometric embedding with respect to Kobayashi metrics, as desired.
\epf

\vskip 0.3cm
\noindent
{\it Remark} \ Given a complex manifold $X$ hyperbolic with respect to the Kobayashi metric, a point $x \in X$, and a nonzero real tangent vector $v \in T_x^{\mathbb R}(X)$, there can be more than one germ of real geodesic curve $\gamma: (-\varepsilon,\varepsilon) \to X$ such that $\gamma(0) = x$ and $\gamma'(0) = v$.  We say that a complex submanifold $S \subset X$ is totally geodesic to mean that given any two distinct points $x_1, x_2\in S$, there always exist a real geodesic curve $\gamma$ on $X$ joining $x_1$ to $x_2$ such that the image of $\gamma$ lies on $S$ (while there may be other real geodesic curves on $X$ joining $x_1$ and $x_2$ that do not entirely lie on $S$).

\subsection{Proof of Theorem~\ref{nonexistence}}
First assume that $\Omega$ and $\Omega'$ satisfy the condition 1). Suppose that there exists a proper holomorphic map $f\colon D^{I}_{p,q}\to D^{III}_{q'}$ with $2\leq q\leq q'<2q-1$. By composing a standard embedding $j:D^{III}_{q'}\to D^I_{q',q'}$, we may assume that $f\colon D^I_{p,q}\to D^I_{q',q'}$ is a proper holomorphic map. Then by Theorem~\ref{main}, $f$ is of he form $g\times h$, where $g:D^I_{p,q}\to D^I_{q',q'}$ is a standard holomorphic map and $h\colon \Omega\to \Omega''$ is a holomorphic map for some subdomain $\Omega''\subset D^I_{q',q'}$ orthogonal to $g(D^I_{p,q})$. Since $f(D^I_{p,q})\subset D^{III}_{q'}$, this implies that $D^{III}_{q'}$ contains a rank $q$ characteristic subspace that contains $D^I_{p,q}$, which is impossible.

Next assume that $\Omega$ and $\Omega'$ satisfy the condition 2).  By the same reason as above, we may assume that $\Omega'$ is of type~I. Suppose there exists a proper holomorphic map $f\colon D^{II}_{n}\to D^{I}_{p',q'}$ with $2\leq q'<2[n/2]-1$. Since $\Omega'$ is of type~I, we obtain 
$i_1=1$, $i_{q-1}=q'-1$ and $i_r=i_{r-1}+2$ for all $r=2,\ldots,q-1.$ 
Since $i_1=1$, $f$ preserves VMRT and therefore is a standard embedding.
Then by the same argument in the proof of Lemma~\ref{lines}, we obtain that for all $r=1,\ldots,[n/2]-1$ and all $\tau\in \mathcal D_0(X)$, $f^\flat_r$ restricted to $Z_{\tau}$ is a standard embedding.  
In particular, $f^\flat_2\colon Z_{\tau}\cap Z_{\rho}\to Z_{f^\sharp_2({\tau)}}\cap Z_{f^\sharp_2(\rho)}$ is a standard embedding from a Grassmannian of rank 3 to a Grassmannian of rank 2 if $\dim Z_{\tau}\cap Z_{\rho}>0$, which is impossible.

\section{Appendix}\label{Appendix}
For $X = Gr(q, p),$ see \cite{K21}.
Let $p,q$ be positive integers such that $q\leq p$.
Define a Hermitian inner product $\langle~,~\rangle_{p,q}$ in $\mathbb{C}^{p+q}$ by
\begin{equation}\nonumber
\langle u,v\rangle_{p,q}
:=u_1\bar v_1+\cdots+u_q\bar v_q- u_{q+1}\bar
v_{q+1}-\cdots-u_{p+q}\bar v_{p+q},
\end{equation}
for $u=(u_1,\ldots,u_{p+q})$ and $v=(v_1,\ldots,v_{p+q})$.
Recall
\begin{equation*}
\begin{aligned}
\Sigma_r(Gr(q,p))&=\{Z\in Gr(q-r,\mathbb C^{p+q}):\langle ~, \rangle_{p,q} |_Z=0 \}~ \text{ for } r\leq q, \\
\Sigma_r(OGr_n) &= \{ Z\in Gr(2[n/2]-r, \mathbb C^{2n}) :
\langle~, \rangle_{n,n}|_Z=0, ~ S_n|_Z=0 \}~\text{ for } r\leq n, \\
\Sigma_r(LGr_n) &= \{ Z\in Gr(n-r,\mathbb C^{2n}) :
\langle~, \rangle_{n,n}|_Z=0, ~ J_n|_Z=0 \}~ \text{ for } r\leq n.
\end{aligned}
\end{equation*}

For $X = Gr(q,p), \, OGr_n$ or $LGr_n$, let $\ell$ denote $q-r$, $2[n/2]-r$, or $n-r$, $G$ denote $SU(p,q)$, $SO(n,n)$ or $Sp(n)$, 
and $\frak g$ denote $su(p,q)$, $so(n,n)$ or $sp(n)$ respectively.
If $X = OGr_n$ or $LGr_n$, then $p=q=n$.
For $X = Gr(q,p), \, OGr_n$ or $LGr_n$, 
a \emph{Grassmannian frame adapted to} $\Sigma_r(X)$, or simply
$\Sigma_r(X)$-\emph{frame} is a frame $\{ Z_1,\ldots,Z_{p+q}\}$ of
$\mathbb{C}^{p+q}$ with $\det(Z_1,\ldots,Z_{p+q})=1$ such that
\begin{equation}\label{structure}
\langle Z_\alpha,Z_{p+q-\ell+\beta}\rangle_{p,q}=\langle
Z_{p+q-\ell+\a},Z_\b\rangle_{p,q}=~\delta_{\alpha\beta},~~
 \langle Z_{\ell+j},Z_{\ell+k}\rangle_{p,q} =\3\delta_{jk},
\end{equation}
for $\a,\,\b=1,\ldots,\ell,~j,\,k=1,\ldots, p+q-2\ell$ and
\begin{equation*}\label{structure2}
\langle Z_\L,Z_\G\rangle_{p,q}=0~\text{   otherwise,   }
\end{equation*}
where $\3\delta_{jk}=\delta_{jk}$ if $\min(j,k)\leq q-\ell$,
$\3\delta_{jk}=-\delta_{jk}$ otherwise, and the capital Greek indices $\L,\G,\Omega$ etc.\ run
from $1$ to $p+q$, i.e., the scalar product $\langle\cdot,\cdot\rangle_{p,q}$
in basis $\{Z_1,\ldots,Z_{p+q}\}$ is given by the matrix
$$
\begin{pmatrix}
0&0&0& I_\ell\\
0&I_{q-\ell}&0&0\\
0&0&-I_{p-\ell}&0\\
I_\ell&0&0&0\\
\end{pmatrix}.
$$
We use the notation
\begin{eqnarray*}
Z&:=&(Z_1,\ldots,Z_\ell),\\
X=(X_1,\ldots,X_{p+q-2\ell})&:=&(Z_{\ell+1},\ldots,Z_{p+q-\ell}),\\
Y=(Y_1,\ldots,Y_\ell)&:=&(Z_{p+q-\ell+1},\ldots,Z_{p+q}).
\end{eqnarray*}

Let $\mathcal{B}_r(X)$ be the set of all
$\Sigma_r(X)$-frames. Then $\mathcal{ B}_r(X)$ can be identified with
$G$ by the left action. By abuse of notation, we also denote by $Z$
the $q$-dimensional subspace
of $\mathbb{C}^{p+q}$ spanned by $Z_1,\ldots,Z_q$.
Then we can regard $\mathcal{B}_r(X)$ as a bundle over $\Sigma_r(X)$
with respect to a natural projection $(Z, X, Y)\to Z$. The Maurer-Cartan form
$\pi=(\pi_\L^{~\G})$ on $\mathcal{B}_r(X)$ is a $\frak g$-valued one form given by the equation
\begin{equation}\nonumber
dZ_\L=\pi_\L^{~\G}Z_\G
\end{equation}
satisfying the structure equation
\begin{equation*}\label{struc-eq}
d\pi_\L^{~\G}=\pi_\G^{~\Omega}\wedge\pi_\Omega^{~\G}.
\end{equation*}
We use the block matrix representation
with respect to the basis $(Z,X,Y)$ to write
\begin{equation*}\label{pi}
\begin{pmatrix}
\pi_{\a}^{~\b} & \pi_{\a}^{~\ell+j} & \pi_{\a}^{~p+q-\ell+\b}\\
\pi_{q+k}^{~\b} & \pi_{\ell+k}^{~\ell+j}  & \pi_{\ell+k}^{~p+q-\ell+\b}\\
\pi_{p+q-\ell+\a}^{~\b} & \pi_{p+q-\ell+\a}^{~\ell+j}  & \pi_{p+q-\ell+\a}^{~p+q-\ell+\b}\\
\end{pmatrix}
= :
\begin{pmatrix}
\psi_{\a}^{~\b} & \theta_{\a}^{~j} & \phi_{\a}^{~\b}\\
\sigma_{k}^{~\b} & \omega_{k}^{~j} & \theta_{k}^{~\b}\\
\xi_{\a}^{~\b} & \sigma_{\a}^{~j}   &\3\psi_{\a}^{~\b}\\
\end{pmatrix},
\end{equation*}
which satisfies the symmetry relations
\begin{equation*}\label{symmetries}
\begin{pmatrix}
\psi_{\a}^{~\b} & \theta_{\a}^{~j} & \phi_{\a}^{~\b}\\
\sigma_{k}^{~\b} & \omega_{k}^{~j} & \theta_{k}^{~\b}\\
\xi_{\a}^{~\b} & \sigma_{\a}^{~j} & \3\psi_{\a}^{~\b}\\
\end{pmatrix}
=-
\begin{pmatrix}
\3\psi_{\bar\b}^{~\bar\a} & \3\delta_{j}^i\theta_{\bar i}^{~\bar\a}  & \phi_{\bar\b}^{~\bar\a}\\
\3\delta_{i}^k\sigma_{\bar\b}^{~\bar i} &\3\delta_{i}^k \omega_{\bar j}^{~\bar i} & \3\delta_{i}^k\theta_{\bar\b}^{~\bar i}\\
\xi_{\bar\b}^{~\bar\a} &\3\delta_{j}^i \sigma_{\bar i}^{~\bar\a} & \psi_{\bar\b}^{~\bar\a}\\
\end{pmatrix}
\end{equation*}
that follow directly by differentiating \eqref{structure}.
For a change of frame given by
$$
\begin{pmatrix}
\2Z\\
\2X\\
\2Y
\end{pmatrix}
:=U
\begin{pmatrix}
Z\\
X\\
Y
\end{pmatrix}
,$$
$\pi$ changes via
$$\widetilde \pi=dU\cdot U^{-1}+U\cdot\pi\cdot U^{-1}.$$
%or, equivalently,
%\begin{equation}\Label{pi-change}
%\2\pi\cdot U = dU + U\cdot\pi.
%\end{equation}

{
If $X =LGr_n$, $\{Z_1.\ldots, Z_{2n}\}$ satisfies 
\begin{equation*}\label{SGr}
J_n(Z_\alpha, Z_\b) = 0,\quad \a,\b=1,\ldots,\ell.
\end{equation*}
We may regard $\Sigma_r(X)$ as a submanifold of $\Sigma_r(Gr(n,n))$.
Since $\Sigma_r(LGr_n)$ is a generic CR manifold in $SGr(n-r, \mathbb C^{2n})$, we obtain 
$$\mathbb CT_P \Sigma_r(X)/(T^{1,0}_P\Sigma_r(X)+T^{0,1}_P\Sigma_r(X))=T_P SGr(n-r, \mathbb C^{2n})/D\cong S^2U^*,$$
where $D$ and $U^*$ are defined in Section~\ref{Rigidity of the pair}. Therefore we obtain a reduction of frame by
\beq\label{frame reduction}
\phi_\a^{~\b}-\phi_\b^{~\a}=0
\eeq
and 
$\phi_\a^{~\b}+\phi_\b^{~\a},~\a,\b=1,\ldots,\ell$ span the contact forms.  
That is, the set of all $\Sigma_r(Gr(n,n))$-frames adapted to $\Sigma_r(X)$ is the maximal integral manifold of \eqref{frame reduction}.
If $X=OGr_n$, then
$\{Z_1.\ldots, Z_{2n}\}$ satisfies 
\begin{equation*}\label{ogr}
S_n(Z_\alpha, Z_\beta) = 0,\quad \a,\b=1,\ldots,\ell
\end{equation*}
and  
$$\mathbb CT_P \Sigma_r(X)/(T^{1,0}_P\Sigma_r(X)+T^{0,1}_P\Sigma_r(X))=T_P OGr(2([n/2]-r),\mathbb C^{2n})/D\cong \Lambda^2 E^*,$$
where for $P=[E],$ 
$$D=E\otimes(E^\perp/E),\quad E^*=\mathbb C^{2n}/E^\perp.$$
Therefore we obtain a reduction of frame by
$$\phi_\a^{~\b}+\phi_\b^{~\a}=0$$ 
and 
$\phi_\a^{~\b}-\phi_\b^{~\a},~\a,\b=1,\ldots,\ell$ span the contact forms. 
}

There are several types of frame changes.

\begin{Def}\label{changes}
{\rm We call a change of frame}
\begin{enumerate}
\item[i)]change of position {\rm if}
$$
\widetilde Z_\alpha=W_\alpha^{~\beta}Z_\beta,\quad
\widetilde Y_\alpha=V_\alpha^{~\beta}Y_\beta,\quad
\widetilde X_j=X_j,
$$
{\rm where $W=(W_\alpha^{~\beta})$ and $V=(V_\alpha^{~\beta})$ are
$\ell\times \ell$ matrices satisfying $\overline{V^t}W=I_\ell$}
 and if $X = OGr_n$ or $LGr_n$, $W$ and $V$ are symmetric or skew-symmetric, respectively;

\item[ii)]change of real vectors {\rm if}
$$
\widetilde Z_\alpha=Z_\alpha,\quad
\widetilde X_j=X_j,\quad
\widetilde Y_\alpha=Y_\alpha+H_\alpha^{~\beta}Z_\beta,
$$
%or
%\begin{equation}
%\begin{pmatrix}
%\2Z_{\a}\\
%\2X_{j}\\
%\2Y_{\a}
%\end{pmatrix}
%=
%\begin{pmatrix}
%I_{q} & 0 & 0\\
%0 & I_{n} & 0\\
%H_{\a}^{~\b}&0& I_{q}
%\end{pmatrix}
%\begin{pmatrix}
%Z_{\b}\\
%X_{k}\\
%Y_{\b}
%\end{pmatrix},
%\end{equation}
{\rm where $H=(H_\alpha^{~\beta})$ is a Hermitian matrix};

\item[iii)]dilation {\rm if}
$$
\widetilde Z_\alpha=\lambda_{\alpha}^{-1}Z_\alpha,\quad
\widetilde Y_\alpha=\lambda_\alpha Y_\alpha,\quad
\widetilde X_j=X_j,
$$
{\rm where $\lambda_\alpha>0$};

\item[iv)]rotation {\rm if}
$$
\widetilde Z_\alpha=Z_\alpha,\quad
\widetilde Y_\alpha=Y_\alpha,\quad
\widetilde X_j=U_j^{~k}X_k,
$$
{\rm where $(U_j^{~k})$ is an $SU(q-\ell, p-\ell)$ matrix.}
\end{enumerate}
\end{Def}

Change of position in Definition~\ref{changes} sends $\phi$ and $\theta$ to
$$
\widetilde
\phi_\alpha^{~\beta}=W_\alpha^{~\gamma}\phi_\gamma^{~\delta}W^{*}{}_{\delta}^{~\b},
\quad W^{*}{}_{\delta}^{~\b}=\overline{W_{\beta}^{~\delta}},\quad
\widetilde\theta_\alpha^{~j}=W_\alpha^{~\beta}\theta_\beta^{~j}.
$$
Dilation changes $\phi_\a^{~\b}$, $\theta_\a^{~j}$ to
$$
\widetilde
\phi_\alpha^{~\beta}=\frac{1}{\lambda_\a\lambda_\b}\phi_\a^{~\b}
,\quad
\widetilde\theta_\alpha^{~j}=\frac{1}{\lambda_\a}\theta_\a^{~j},
$$
while rotation remains $\phi_\a^{~\b}$ unchanged and changes $\theta_\a^{~j}$ to
$$
\widetilde\theta_\alpha^{~j}=\theta_\a^{~k}U_k^{~j}.
$$

Finally, we will use the change of frame given by
\begin{equation*}\label{last-change}
\widetilde Z_\alpha=Z_\alpha,\quad
\widetilde X_j=X_{j} + C_j^{~\beta}Z_\beta,\quad
\widetilde Y_\alpha=Y_\alpha+A_\alpha^{~\beta}Z_\beta+B_\alpha^{~j}X_j
\end{equation*}
%or
%\begin{equation*}\label{psi-change}
%\begin{pmatrix}
%\2Z_{\a}\\
%\2X_{j}\\
%\2Y_{\a}
%\end{pmatrix}
%=
%\begin{pmatrix}
%I_{q} & 0 & 0\\
%C_{j}^{~\b} & I_{p-q} & 0\\
%A_{\a}^{~\b}& B_{\a}^{~j}& I_{q}
%\end{pmatrix}
%\begin{pmatrix}
%Z_{\b}\\
%X_{k}\\
%Y_{\b}
%\end{pmatrix}
%\end{equation*}
such that
$$C_j^{~\alpha}+B_j^{~\alpha}=0$$
and
$$A_\alpha^{~\beta} + \overline{A_\beta^{~\alpha}}
+B_\alpha^{~j}B_j^{~\beta}=0,$$
where
$$B_j^{~\alpha}:=\3\delta_{jk}\overline{B_\alpha^{~k}}.$$
Then the new frame $(\widetilde Z,\widetilde X,\widetilde Y)$ is an $\Sigma_r(X)$-frame and the related one forms $\widetilde\phi_\alpha^{~\beta}$ remain the same, while $\widetilde\theta_\alpha^{~j}$ change to
$$\widetilde\theta_\alpha^{~j}=\theta_\alpha^{~j}-\phi_\alpha^{~\beta}B_\beta^{~j}.$$

\end{document}